\documentclass[10pt]{preprint}

\usepackage[margin=3.5cm]{geometry}
\usepackage[utf8x]{inputenc}
\usepackage[english]{babel}
\usepackage[full]{textcomp}
\usepackage{placeins}
\usepackage[osf]{newtxtext}

\usepackage{amssymb}
\usepackage{amsmath}
\usepackage{amsthm}

\usepackage{mhequ}
\usepackage{booktabs}
\usepackage{tikz}
\usepackage{mathrsfs}
\usepackage[noadjust]{cite}
\usepackage{microtype}
\usepackage{comment}
\usepackage{slashed}
\usepackage{mathtools}
\usepackage{centernot}
\usepackage{footnote}
\usepackage{enumerate}
\usepackage[shortlabels]{enumitem}
\usepackage{stackrel}
\usepackage{longtable}
\usepackage{cprotect}
\usepackage{xstring}
\usepackage{bbm}
\usepackage{dsfont}
\usepackage{float}
\usepackage[colorlinks=true, pdfstartview=FitV, linkcolor=colorLink, citecolor=colorCite, urlcolor=colorLink, linktocpage=true, hypertexnames=false]{hyperref}

\usepackage{bm}

\usepackage[capitalize]{cleveref}

\crefname{enumi}{Theorem}{Theorems}
\Crefname{enumi}{Theorem}{Theorems}
\newenvironment{thmenumerate}[1][]{%
  \begin{enumerate}[label=\upshape(\roman*),ref=\thetheorem(\roman*\,),leftmargin=*,font=\upshape,#1]
    \renewcommand{\theenumi}{\thetheorem(\roman{enumi})}%
}{%
  \end{enumerate}
}

\usepackage{upgreek}
\usepackage{parskip}
\usepackage{graphicx}
\usepackage{tikz}
\usetikzlibrary{positioning,arrows.meta}

\makeatletter
\renewcommand{\paragraph}{%
	\@startsection{paragraph}{4}%
	{\z@}{1.5ex \@plus 1.5ex \@minus .2ex}{-0.7em}%
	{\normalfont\normalsize\bfseries}%
}
\makeatother

\makeatletter
\def\thm@space@setup{%
	\thm@preskip=\parskip \thm@postskip=0pt
}
\makeatother

\setlist[itemize]{leftmargin=5mm}

\DeclareSymbolFont{timesoperators}{T1}{ptm}{m}{n}
\SetSymbolFont{timesoperators}{bold}{T1}{ptm}{b}{n}
\DeclareMathAlphabet{\mathbb}{U}{jkpsyb}{m}{n}
\SetMathAlphabet{\mathbb}{bold}{U}{jkpsyb}{bx}{n}

\allowdisplaybreaks
\definecolor{colorLink}{RGB}{0,100,162}
\definecolor{colorCite}{RGB}{8,124,100}
\def\R{\mathbb{R}}

\def\RR{\mathbb{R}}
\def\R{\mathbb{R}}
\def\PP{\mathbb{P}}
\def\cP{\mathcal{P}}
\def\EE{\mathbb{E}}

\def\to{\rightarrow}

\def\i{\infty}
\def\d{\text{d}}
\def\d{{\rm{d}}}

\newcommand{\bb}[1]{\mathbb{#1}}
\newcommand{\ca}[1]{\mathcal{#1}}
\newcommand{\scr}[1]{\mathscr{#1}}
\newcommand{\worknote}[1]{}
\newcommand{\dist}{\mathrm{dist}}

\let\alpha\upalpha
\let\beta\upbeta
\let\delta\updelta
\let\gamma\upgamma
\let\mu\upmu
\let\eta\upeta
\let\nu\upnu
\let\rho\uprho
\let\chi\upchi
\let\xi\upxi
\let\zeta\upzeta
\let\tau\uptau
\let\varphi\upvarphi
\let\lambda\uplambda
\let\theta\uptheta
\let\pi\uppi
\let\Upsilon\Upupsilon
\let\Theta\Uptheta
\let\Psi\Uppsi
\let\Xi\Upxi

\def\dash{\leavevmode\unskip\kern0.18em--\penalty\exhyphenpenalty\kern0.18em}
\def\slash{\leavevmode\unskip\kern0.15em/\penalty\exhyphenpenalty\kern0.15em}

\makeatletter
\renewcommand{\operator@font}{\mathgroup\symtimesoperators}
\makeatother

\makeatletter
\DeclareRobustCommand{\TitleEquation}[2]{\texorpdfstring{\StrLeft{\f@series}{1}[\@firstchar]$\if%
		b\@firstchar\boldsymbol{#1}\else#1\fi$}{#2}}
\makeatother

\makeatletter
\newcommand{\pushright}[1]{\ifmeasuring@#1\else\omit\hfill$\displaystyle#1$\fi\ignorespaces}
\newcommand{\pushleft}[1]{\ifmeasuring@#1\else\omit$\displaystyle#1$\hfill\fi\ignorespaces}
\makeatother

\renewcommand{\bar}{\overline}
\renewcommand{\hat}{\widehat}
\renewcommand{\tilde}{\widetilde}

\makeatletter
\newcommand{\oset}[3][0ex]{%
	\mathrel{\mathop{#3}\limits^{
			\vbox to#1{\kern-2\ex@
				\hbox{$\scriptstyle#2$}\vss}}}}
\makeatother

\usetikzlibrary{shapes.misc}
\usetikzlibrary{shapes.symbols}
\usetikzlibrary{shapes.geometric}
\usetikzlibrary{decorations}
\usetikzlibrary{decorations.markings}
\usetikzlibrary{calc}
\usetikzlibrary{external}
\usetikzlibrary{arrows}
\usetikzlibrary{patterns}
\theoremstyle{plain}

\newtheorem{theorem}{Theorem}[section]
\crefname{theorem}{Theorem}{Theorems}

\newtheorem{corollary}[theorem]{Corollary}
\crefname{corollary}{Corollary}{Corollaries}

\newtheorem{lemma}[theorem]{Lemma}
\crefname{lemma}{Lemma}{Lemmas}

\newtheorem{proposition}[theorem]{Proposition}
\crefname{proposition}{Proposition}{Propositions}

\crefname{claim}{Claim}{Claims}

\crefname{theoremA}{Theorem}{Theorems}

\crefname{propositionA}{Proposition}{Propositions}

\theoremstyle{definition}

\newtheorem{definition}[theorem]{Definition}
\crefname{definition}{Definition}{Definitions}

\crefname{notation}{Notation}{Notations}

\crefname{acknowledgements}{Acknowledgements}{Acknowledgements}

\newtheorem{assumption}[theorem]{Assumption}
\crefname{assumption}{Assumption}{Assumptions}

\newtheorem{remark}[theorem]{Remark}
\crefname{remark}{Remark}{Remarks}

\crefname{observation}{Observation}{Observations}

\crefname{example}{Example}{Examples}

\crefname{manualassumptioninner}{Assumption}{Assumptions}

\crefname{customthm}{Theorem}{Theorems}

\numberwithin{equation}{section}

\usepackage[textsize=scriptsize,linecolor=white,bordercolor=red,backgroundcolor=yellow]{todonotes}
\usepackage{autonum}

\def\rm{\mathrm}
\def\Law{\mathrm{Law}}
\usepackage[most]{tcolorbox}
\usepackage{caption}
\allowdisplaybreaks

\begin{document}
	\title{Weak synchronisation for McKean--Vlasov SDEs}

	\author{Benjamin Gess$^{1}$, Rishabh S.\ Gvalani$^{2}$, Shanshan Hu$^{3}$}
	
	\institute{
		Institut f\"{u}r Mathematik, Technische Universit\"{a}t Berlin \& Max--Planck Institute for Mathematics in the Sciences, Leipzig, Email: \href{mailto:benjamin.gess@tu-berlin.de}{\color{black} \texttt{benjamin.gess@tu-berlin.de}} 
		\and
		School of Mathematics, University of Edinburgh, Email: \href{mailto:rgvalani@ed.ac.uk}{\color{black}\texttt{rgvalani@ed.ac.uk}}
		\and
		Institut f\"{u}r Mathematik, Technische Universit\"{a}t Berlin, Email: \href{mailto:shanshan.hu@tu-berlin.de}{\color{black} \texttt{shanshan.hu@tu-berlin.de}} 
	}
	
	\maketitle	
	\begin{abstract}
Synchronisation by noise for McKean–Vlasov stochastic differential equations is investigated. A transfer principle is introduced by which synchronisation by noise and diagonal mixing can be transferred from an associated limiting frozen-diffusion SDE to a genuinely law-dependent McKean–Vlasov SDE. The usefulness of this principle is demonstrated through an application to ensemble Kalman sampling, thereby providing its first application in a sampling context.\\
Synchronisation for SDEs with multiplicative noise is then revisited from the perspective of sampling. Existing general frameworks providing sufficient conditions for synchronisation by noise are refined and extended to a control-oriented setting on noncompact state spaces, and coefficient-level conditions are derived by which these criteria can be verified. Motivated by extrapolation schemes in sampling, the freedom in the construction of couplings for a fixed sampler is emphasised, together with the advantage of choosing couplings that exhibit favourable synchronisation properties.
	\end{abstract}

	\setcounter{tocdepth}{2}
	\tableofcontents

\section{Introduction}
\label{sec:literature-contributions}

Synchronisation by noise refers to the global asymptotic stability of solutions to SDEs starting from different initial conditions and driven by the same realisation of the noise \cite{FlandoliGessScheutzow.2017.PTRF511}. In the framework of random dynamical systems (RDS), weak synchronisation means that the minimal weak point attractor is almost surely a singleton,
\begin{equation}
  A(\omega)=\{a(\omega)\},
  \qquad \PP\text{-a.s.}
\end{equation}
which entails that solutions starting from different points approach one another in probability.  Besides being a fundamental property of stochastic flows, synchronisation provides a natural coupling mechanism in sampling. In particular, it may force the law of the common-noise two-point motion to concentrate on the diagonal, a property relevant to variance reduction for coupled estimators, see \cite{LPP15} and Section \ref{intro:RR} below.

The purpose of this paper is twofold: First, we refine and extend the established general criteria for synchronisation by noise to a control-oriented setting on noncompact state spaces, and derive coefficient-level conditions by which these criteria can be verified for multiplicative diffusions. The main focus of this part of the work, is its implications for sampling, emphasising the role of the choice of coupling, and its synchronising behaviour, for a fixed one-point sampler.  Second, and more importantly, we study synchronisation by noise for McKean--Vlasov SDEs. We
introduce a transfer principle, and a corresponding two-point statement, which allows to transfer synchronisation and diagonal mixing of a limiting frozen diffusion to a genuinely law-dependent, asymptotically stationary McKean--Vlasov SDE. The usefulness of this principle is illustrated by an application to the ensemble Kalman sampler, which gives a first sampling application of the transfer principle. We discuss these contributions in turn.

\textbf{Multiplicative-noise SDEs and their relevance to sampling: } The basic mechanism leading to synchronisation of SDEs has been thoroughly studied in the literature, see e.g. \cite{FlandoliGessScheutzow.2017.PTRF511} for the case of unbounded state space and additive noise, \cite{LPP15} for a Lyapunov function type argument applicable also for certain multiplicative noise SDEs, as well as the connection to Richardson--Romberg extrapolation, and \cite{Baxendale1991} for a general criterion for synchronisation on compact state space.

The first contribution of this work is an abstract criterion for synchronisation slightly generalising the abstract framework introduced in \cite{FlandoliGessScheutzow.2017.PTRF511} and \cite{Baxendale1991}. Precisely,  \cref{thm:weaksynchronisationRDS} assumes strong mixing and local weak stability on a contracting region, together with direct positive-probability reachability of that region by each pair of trajectories. Compared with the criterion of \cite{FlandoliGessScheutzow.2017.PTRF511}, this replaces pointwise strong swift
transitivity and the separate global pairwise recurrence assumption by a single reachability condition. Subsequently, in \cref{thm:weaksynchronisationSDE} and Section \ref{sec:synchr_SDE}, we employ this criterion to develop sufficient coefficient level conditions for  multiplicative SDEs.

The main emphasis of this part of the paper lies in its implications for sampling. Recent work \cite{LelievrePavliotisRobinSantetStoltz2025} has demonstrated that an appropriate choice of state-dependent diffusion matrix in a Gibbs sampler can significantly increase the associated spectral gap and thereby improve sampling efficiency. Here we show that this does not exhaust the available design freedom: even after the diffusion matrix, and hence the one-point generator, has been fixed, its factorisation into noise fields still determines the synchronous coupling and may decisively alter its long-time two-point behaviour. Indeed, for the one-dimensional double-well model considered in \cite[Section~4.1, Figure~3]{LelievrePavliotisRobinSantetStoltz2025}, we show that the canonical principal-square-root realisation is not weakly synchronising, whereas an alternative overcomplete realisation of exactly the same one-point diffusion yields a synchronising synchronous coupling (see \cref{sec:optimized-multiplicative-synchronisation}). This example provides a concrete application of the general theory of synchronisation by noise for multiplicative SDEs developed here. Thus, in coupled sampling procedures based on common noise, such as those used in Richardson--Romberg extrapolation, care should be taken not only in selecting an effective multiplicative diffusion but also in choosing a noise realisation whose induced synchronous coupling synchronises.

\paragraph{McKean--Vlasov equations.}

The main novelty of the paper concerns the analysis of synchronisation by noise for coupled microscopic--macroscopic system
\begin{subequations}
\label{system00}
\begin{align}
  \d Y_t
  &=
  b(Y_t,\mu_t)\,\d t
  +\sigma(Y_t,\mu_t)\,\d W_t,
  \qquad
  Y_0=y\in\RR^d,
  \label{eq:MVRDE00}\\
  \partial_t\mu_t
  &=
  -\nabla\cdot\bigl(b(\cdot,\mu_t)\mu_t\bigr)
  +\frac12 D^2:
  \bigl((\sigma\sigma^T)(\cdot,\mu_t)\mu_t\bigr),
  \qquad
  \mu_0=\mu\in\ca P_2(\RR^d).
  \label{eq:MVPDE00}
\end{align}
\end{subequations}
Here, the measure curve \(\mu_t\) evolves deterministically, while the particle
coordinate is driven by the noise. The problem of sychronisation by noise in this context is motivated from extrapolation schemes in sampling, see Section \ref{intro:RR} below. 

In contrast to the case of non-law-dependent SDEs, since the coefficients of \eqref{system00} depend on the evolving law, the particle coordinate is not governed by a fixed one-point
Markov semigroup. Likewise, the common-noise two-point motion is generally not
a time-homogeneous diffusion on the ordinary product space. Classical
synchronisation and confluence theorems therefore cannot be applied directly,
even when \(\mu_t\) converges to equilibrium.

Following the RDS construction for McKean--Vlasov SDEs in
\cite{gess2025random}, we lift \eqref{system00} to the particle--measure
space,
\[
  (y,\mu)\longmapsto
  \bigl(Y_t^{y,\mu}, S_t\mu\bigr).
\]
This state space agrees with that used to linearise nonlinear Fokker--Planck
equations in \cite{RenRocknerWang.2022.JDE1}. The corresponding lifted Markov
semigroup records simultaneously the particle and measure coordinates.

 The new step is an asymptotic freezing argument. If
\begin{equation}
    S_t\mu\xlongrightarrow[t\to\infty]{\ca{W}_2}\mu^\infty,
\end{equation}
we compare this nonautonomous particle equation with the frozen SDE
\begin{equation}
  \d Y_t^\infty=
  b(Y_t^\infty,\mu^\infty)\,\d t
  +\sigma(Y_t^\infty,\mu^\infty)\,\d W_t.
  \label{eq:frozenIntro}
\end{equation}
If
$\mu_\omega^\infty$ denotes the statistical equilibrium of this frozen
diffusion, \cref{thm:weak_synchronisation_MVSDE} identifies the limiting random attractor of the lifted
system as
\begin{equation}
     \operatorname{supp}(\mu_\omega^\infty)
  \times\{\mu_\infty\} \, ,
\end{equation}
conditional on strong mixing of the lifted semigroup of the coupled system \eqref{system00}. In particular, synchronisation of the frozen SDE implies conditional
synchronisation of the original McKean--Vlasov equation in the particle
coordinate. 
 
  There exist
global dissipativity assumptions (see \cref{rem:global_dissipativity_E_0}) that guarantee strong mixing for the lifted semigroup. However, these assumptions are not satisfied by the covariance-modulated sampling
dynamics. The second route is an asymptotic-autonomy argument established in \cref{prop:characterisation_of_E_0}. Exponential
convergence of \(S_t\mu\), quantitative mixing of the frozen equation,
suitable continuity of \(b\) and \(\sigma\sigma^T\) in the measure variable,
and uniform-in-time moment estimates imply tightness and establish strong mixing for the lifted semigroup. Weak synchronisation of
the frozen equation can then be imported without imposing global contraction
on the original law-dependent dynamics. The result is modular: additive-noise
theorems from \cite{FlandoliGessScheutzow.2017.PTRF511}, Lyapunov-exponent criteria, or
the control-based multiplicative-noise results above may all be used for the
frozen equation.

As an application, we consider the mean-field ensemble Kalman sampler for the
Gaussian target $\mu^\infty=\mathsf N_{{\sf m},{\sf C}}$. For the quadratic
potential \(V(y)=\frac12|y-{\sf m}|_{\sf C}^2\), its particle equation is
\[
  \d Y_t
  =
  -\operatorname{Cov}(\mu_t){\sf C}^{-1}(Y_t-{\sf m})\,\d t
  +\sqrt{2\operatorname{Cov}(\mu_t)}\,\d W_t.
\]
Available estimates for the covariance-modulated nonlinear Fokker--Planck
equation give exponential convergence of \(\mu_t\) to the target
\cite{BurgerEtAl2025}. Freezing at equilibrium gives the additive
Ornstein--Uhlenbeck equation
\[
  \d Y_t^\infty
  =
  -(Y_t^\infty-{\sf m})\,\d t
  +\sqrt{2{\sf C}}\,\d W_t,
\]
whose synchronisation is classical. The transfer theorem yields conditional
weak synchronisation of the ensemble Kalman dynamics on
\(\RR^d\times\ca P_{2,+}(\RR^d)\), and
\begin{equation}
  \Law(Y_t^1,Y_t^2)
  \xrightarrow[t\to\infty]{w}
  \mu^\infty_\Delta.
  \label{eq:MVTwoPointMixingIntro}
\end{equation}identifies the limiting common-noise coupling
with the diagonal coupling. \\
This diagonal two-point limit is the dynamical input required to extend
common-noise variance-reduction arguments to  Richardson--Romberg extrapolation in nonlinear samplers. The rigorous validation of this is subject of upcoming work.

\subsection{Relevance to Richardson--Romberg extrapolation in sampling}
\label{intro:RR}

Ensemble-based methods are widely used to sample probability distributions arising in Bayesian inverse problems and machine learning. A representative example is the ensemble Kalman sampler (EKS), which evolves an interacting particle ensemble and uses its empirical covariance to adapt the dynamics to the geometry of the target distribution. In the mean-field limit, an individual particle is described by a McKean--Vlasov SDE, motivating the study of numerical sampling methods for such dynamics and their long-time behaviour.

In practice, sampling from the invariant probability measure requires a time discretisation, typically an Euler--Maruyama scheme with decreasing stepsizes. The stepsize balances convergence rate against discretisation bias: larger stepsizes make the total simulation time $T_n$ grow faster, but accumulated time-discretisation errors may leave a non-vanishing bias in the limiting distribution.

Richardson--Romberg extrapolation addresses this trade-off by coupling two Euler--Maruyama schemes with stepsizes $t_k$ and $t_k/2$ and taking an appropriate linear combination of their empirical measures. It cancels the leading discretisation bias and thereby
permits larger step sizes at the central-limit scale.  Its variance is not,
however, determined solely by the one-point dynamics: it involves the
invariant probability measure of the coupled two-point motion driven by the
same noise, \begin{subequations}\label{eq:twopointmotionIntro}
	\begin{align}
		\d X_t=\,\,&b(X_t, \Law(X_t))\,\d t
		+\sigma(X_t, \Law(X_t))\,\d W_t,\label{eq:twopointmotionIntro_1}\\
		\d Y_t=\,\,&b(Y_t, \Law(Y_t))\,\d t
		+\sigma(Y_t, \Law(Y_t))\,\d W_t.
	\end{align}
\end{subequations}
This leads to the principal structural question addressed here:
when is that two-point invariant measure uniquely determined?

If $\pi$ denotes the unique invariant probability measure of the one-point motion and $\pi^{(2)}$ \emph{an} invariant measure of the two-point motion, the asymptotic variance of the extrapolated estimator contains a correlation term of the form
\begin{align}\label{eq:varianceIntro}
	\int_{\RR^d\times\RR^d}
	\left\langle
	\sigma^T\nabla g_f(x),
	\sigma^T\nabla g_f(y)
	\right\rangle
	\,\pi^{(2)}(\d x,\d y)\,,
\end{align}
where $g_f$ is the solution to the Poisson equation $\mathcal{A} g_f =f$ where $\mathcal{A}$ is the generator associated with \eqref{eq:twopointmotionIntro_1} with the measure argument frozen at $\pi$, for a suitable test function $f$. Note that since $\pi^{(2)}$ may not be unique it depends on the choice of initial condition.
The one-point ergodic behaviour does not control this term. In particular, marginal convergence
\[
	\Law(X_t)\xrightarrow[t\to\infty]{}\pi,
	\qquad
	\Law(Y_t)\xrightarrow[t\to\infty]{}\pi
\]
neither determines the limiting joint distribution of $(X_t,Y_t)$ nor guarantees uniqueness of its invariant probability measure, since the common noise may preserve non-trivial correlations. The sampling application therefore requires the stronger property
\begin{align}\label{eq:generalisedstrongmixingIntro}
	\Law(X_t,Y_t)
	\longrightarrow
	\pi_\Delta,
	\qquad
	\pi_\Delta
	\coloneqq
	\pi\circ(x\mapsto(x,x))^{-1},
\end{align}
for all initial laws $\Law(X_0),\Law(Y_0)\in \ca{P}(\RR^d)$ . This is a generalised version of  strong mixing (see \cref{matierial_def:strong_mixing}) which concentrates the limiting two-point law on the diagonal and indeed implies that $\pi_\Delta$ is the unique invariant measure of the two-point motion. The correlation term in \eqref{eq:varianceIntro} then reduces to its one-point counterpart, so the Richardson--Romberg estimator cancels the leading-order bias without increasing the asymptotic variance relative to the original Euler--Maruyama estimator.

\emph{The central question is therefore how to establish this two-point mixing property.}

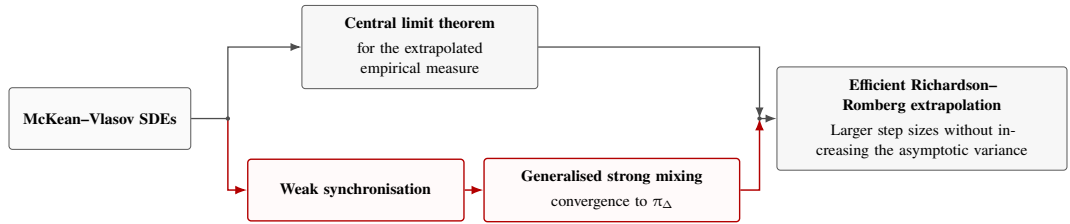
\begin{figure}[h]
	\centering
	\resizebox{\textwidth}{!}{%
		\begin{tikzpicture}[
			font=\small,
			>=Latex,
			model/.style={
				draw=black!65,
				fill=black!2,
				rounded corners=2pt,
				align=center,
				inner sep=7pt,
				minimum height=1.25cm,
				text width=3.2cm
			},
			main/.style={
				draw=black!65,
				fill=black!3,
				rounded corners=2pt,
				align=center,
				inner sep=7pt,
				minimum height=1.25cm,
				text width=4.3cm
			},
			sync/.style={
				draw=red!70!black,
				fill=red!2,
				line width=0.8pt,
				rounded corners=2pt,
				align=center,
				inner sep=7pt,
				minimum height=1.25cm,
				text width=3.9cm
			},
			mixing/.style={
				draw=red!70!black,
				fill=red!2,
				line width=0.8pt,
				rounded corners=2pt,
				align=center,
				inner sep=7pt,
				minimum height=1.25cm,
				text width=4.7cm
			},
			conclusion/.style={
				draw=black!65,
				fill=black!3,
				rounded corners=2pt,
				align=center,
				inner sep=7pt,
				minimum height=1.45cm,
				text width=5.4cm
			},
			arrow/.style={
				->,
				draw=black!70,
				line width=0.75pt
			},
			redarrow/.style={
				->,
				draw=red!70!black,
				line width=0.85pt
			},
			line/.style={
				draw=black!70,
				line width=0.75pt
			},
			redline/.style={
				draw=red!70!black,
				line width=0.85pt
			}
		]

			% -------------------------------------------------
			% Nodes
			% -------------------------------------------------

			\node[model] (mv) at (0,0)
				{\textbf{McKean--Vlasov SDEs}};

			\node[main] (clt) at (6.5,1.45)
				{\textbf{Central limit theorem}\\[1mm]
				 for the extrapolated empirical measure};

			\node[sync] (sync) at (5.2,-1.45)
				{\textbf{Weak synchronisation}};

			\node[mixing] (mixing) at (10.4,-1.45)
				{\textbf{Generalised strong mixing}\\[1mm]
				 convergence to $\pi_\Delta$};

			\node[conclusion] (rr) at (16.7,0)
				{\textbf{Efficient Richardson--Romberg extrapolation}\\[1mm]
				 Larger step sizes without increasing
				 the asymptotic variance};

			% Fork and merge points
			\coordinate (split) at (2.6,0);
			\coordinate (merge) at (13.4,0);

			% -------------------------------------------------
			% Fork from the McKean--Vlasov dynamics
			% -------------------------------------------------

			\draw[line] (mv.east) -- (split);

			\draw[arrow]
				(split)
				|- (clt.west);

			\draw[redarrow]
				(split)
				|- (sync.west);

			% -------------------------------------------------
			% Lower dynamical branch
			% -------------------------------------------------

			\draw[redarrow]
				(sync.east) -- (mixing.west);

			% -------------------------------------------------
			% Merge of the two ingredients
			% -------------------------------------------------

			\draw[arrow]
				(clt.east)
				-| (merge);

			\draw[redarrow]
				(mixing.east)
				-| (merge);

			\draw[arrow]
				(merge) -- (rr.west);

			% Small junctions
			\fill[black!70] (split) circle (1.25pt);
			\fill[black!70] (merge) circle (1.25pt);

		\end{tikzpicture}%
	}

	\caption{
		The central limit theorem and the long-time behaviour of the common-noise
		two-point motion provide two complementary inputs to Richardson--Romberg
		extrapolation. Weak synchronisation implies generalised strong mixing
		towards the diagonal coupling $\mu^\infty_\Delta$; combined with the
		central limit theorem, this yields efficient extrapolation without an
		increase in asymptotic variance.
	}
	\label{fig:weak-synchronisation-clt-rr}
\end{figure}

Our contribution to this picture is therefore to identify and verify the dynamical condition required of the common-noise two-point motion. We prove that weak synchronisation of the McKean--Vlasov dynamics, together with the relevant one-point mixing assumptions, yields generalised strong mixing of the two-point McKean--Vlasov system and hence uniqueness of its invariant probability measure, which is the diagonal coupling $\pi_\Delta$. A transfer principle reduces the required weak synchronisation of the law-dependent dynamics to synchronisation of the limiting frozen SDE, and we verify this framework for the ensemble Kalman sampler. These results provide the two-point dynamical input required for Richardson--Romberg extrapolation; the corresponding McKean--Vlasov central limit theorem and numerical analysis are left to future work.

\subsection{Overview of the existing literature}
\paragraph{Literature on synchronisation by noise.}
The phenomenon of synchronisation by noise has been extensively investigated in the literature. We organise the following discussion according to the main techniques used to establish synchronisation.

Early results originated from order-preserving or monotone random dynamical systems \cite{ArnoldChueshovDSC, ChueshovMonotonelecturenotes} and were developed for one-dimensional additive SDEs and scalar-valued reaction--diffusion equations \cite{CrauelJDDE1998, CaraballoPAMS2007}, using monotonicity of the RDS together with uniqueness of the invariant probability measure. Subsequently, the theory was extended to higher-dimensional additive SDEs and stochastic partial differential equations; see, among others, \cite{FlandoliGessScheutzow.2017.AP1325, FlandoliGessScheutzow.2017.PTRF511, GessJDDE2013, ButkovskyScheutzowCMP2020, GessTsatsoulis.2024.AP1903, GessTsatsoulis.2020.SD17}.
In finite dimensions, \cite{Baxendale1991} studied synchronisation for SDEs on compact manifolds including SDEs with multiplicative noise, while \cite{FlandoliGessScheutzow.2017.PTRF511} developed a general framework on Polish spaces and applied it to additive SDEs on $\RR^d$, including systems with double-well potentials. An alternative approach was given in \cite{LPP15} based on direct constructions of Lyapunov functions for the two point motion for a class of SDEs with multiplicative noise.  Large-deviation and perturbative techniques have also been employed in \cite{Tearne2008PTRF, MartinelliScoppola1988CMP}. Relations between synchronisation, two-point contractibility, and Lyapunov exponents have been investigated in \cite{Newman2018ETDS, ScheutzowVorkastner2018Springer}. In the opposite regime, chaotic behaviour associated with positive top Lyapunov exponents has been studied in \cite{BedrossianBlumenthalSam2022Inven, BedrossianBlumenthalJEMS, BlumenthalEngelAlexandra2023PTRF}.
Beyond additive noise, synchronisation has been considered for equations with linear multiplicative noise \cite{ArnoldCrauelWihstutz1983, CaraballoRobinson2004SCL} and master--slave systems \cite{ChueshovSchmalfu2010JMP}. More recently, synchronisation mechanisms for degenerate and non-Markovian noises have attracted increasing attention; see, for example, \cite{engel2026random, liu2025synchronisation, blessing2026synchronisation, neamctu2025negativity}.
\paragraph{Literature on McKean--Vlasov SDEs.}
McKean--Vlasov SDEs arise naturally as mean-field limits of interacting particle systems; see \cite{Sznitman.1991.165, POCReview1, POCReview2}. They appear in mathematical physics \cite{meanfieldDuke, JabinWang.2018.IM523}, machine learning and sampling \cite{Garbuno-InigoHoffmannLiStuart.2020.SJoADS412, CHSV24}, and biological models such as chemotaxis \cite{S00}.
Their long-time behaviour is considerably more subtle than that of distribution-independent SDEs. In particular, the usual transition operators associated with solutions starting from deterministic points do not in general form a linear Markov semigroup, and nonlinear Markov processes may exhibit non-standard ergodic behaviour; see \cite{Butkovsky.2014.TPA661, WFYSPA2018}. A substantial literature has therefore developed quantitative ergodicity results for McKean--Vlasov SDEs using nonlinear versions of Harris' theorem, coupling methods, Lyapunov functions, and Wasserstein contraction; see, for example, \cite{Butkovsky.2014.TPA661, harrisergodicityTransection, Wang.2023.SPA265}.
Long-time behaviour has also been studied through interacting particle approximations and uniform-in-time propagation of chaos \cite{LiuWuZhang.2021.CMP179, DelgadinoGvalaniPavliotisSmith2023CMP}. For the ensemble Kalman sampler, ergodicity was first established for Gaussian initial data in \cite{Garbuno-InigoHoffmannLiStuart.2020.SJoADS412} and later extended to more general initial distributions using covariance-modulated optimal transport and gradient-flow methods \cite{BurgerEtAl2025}. Random dynamical systems for McKean--Vlasov equations were constructed in \cite{gess2025random}, providing the framework in which the synchronisation problem considered here can be formulated.
\paragraph{Literature on sampling.}
Sampling methods play a central role in statistics, machine learning, and Bayesian inference. Classical approaches include Markov chain Monte Carlo, overdamped and underdamped Langevin dynamics, Hamiltonian Monte Carlo, and Stein variational gradient descent. More recently, ensemble-based methods such as ensemble Kalman inversion, ensemble Kalman filtering, and ensemble Kalman sampling have provided particle-based alternatives which adapt to the geometry of the target distribution through empirical covariance information.
For numerical sampling of invariant probability measures, decreasing-step Euler schemes and their central limit behaviour have been studied extensively. Richardson--Romberg extrapolation provides a higher-order correction by cancelling the leading discretisation bias; see \cite{LPP15}. The present work highlights that, when two coupled discretisations are used, the efficiency of this extrapolation is intrinsically linked to the long-time behaviour of the corresponding two-point dynamics. This provides the connection between numerical sampling and weak synchronisation which motivates the present study.
\subsection{Organisation of the paper}
In \cref{sec2:preliminaries}, we introduce the notation and prerequisites from the theory of synchronisation, including white-noise random dynamical systems, function spaces, and $M$-attractors. In \cref{sec3:weaksynchroforSDE}, we first formulate sufficient conditions for weak synchronisation at the level of abstract random dynamical systems and then provide verifiable assumptions for multiplicative SDEs in terms of the top Lyapunov exponent, reachability, and H\"ormander's condition. Building on these results, weak synchronisation and generalised strong mixing for McKean--Vlasov systems \eqref{eq:MVRDE00}--\eqref{eq:MVPDE00} is investigated in \cref{sec4:weaksynchroforMVSDE} by means of the lifted semigroup on $\RR^d\times\ca{P}_2(\RR^d)$.

\section{Preliminaries}\label{sec2:preliminaries}
In this section, we introduce some preliminaries and notations which will be useful for the rest of the paper. 

We start by introducing some notations for function spaces and probability measures.

\begin{definition}
	Let $k\in \bb{Z}_+$ and  $0\leq \delta\leq1$. We denote by  $C^{k, \delta}(\RR^d; \RR^d)$  the space of functions $f: \RR^d\to \RR^d$, which are $k$ times continuously differentiable and (for $\delta>0$) whose $k$-derivative is  $\delta$-H\"{o}lder continuous (for $\delta=1$: Lipschitz continuous), moreover,  the norm 
	\begin{align}
		\|f\|_{C^{k, \delta}(\RR^d; \RR^d)}\coloneqq \sup_{y\in \RR^d}\frac{|f(y)|}{1+|y|} +\sum_{1\leq |\alpha|\leq k}\sup_{y\in \RR^d} |D^{\alpha} f(y)|+\sum_{|\alpha|=k}\sup_{x\neq y} \frac{|D^{\alpha} f(x)-D^{\alpha} f(y) |}{|x-y|^{\delta}}<\infty\,,
	\end{align}
	where $\alpha=(\alpha_1, \ldots, \alpha_d)\in (\bb{Z}_+)^d$ is a multi-index, $|\alpha|=\alpha_1+\ldots+\alpha_d$
	and 
	\begin{align}
		D^{\alpha} f(y)=D^{\alpha}_y f(y)=\frac{\partial^{|\alpha|}}{(\partial y_1)^{\alpha_1}\ldots (\partial y_d)^{\alpha_d}} f(y)\,.
	\end{align}
	Moreover, we denote by $C_b^{2}(\RR^d;\RR^d)$ the space of functions $f: \RR^d\to \RR^d$ which are twice differentiable and 
	\begin{align}
		\|f\|_{C_b^{2}(\RR^d; \RR^d)}\coloneqq \sup_{y\in \RR^d} |f(y)|+\sum_{1\leq |\alpha|\leq 2}\sup_{y\in \RR^d} |D^{\alpha} f(y)|<\infty\,.
	\end{align}
\end{definition}
We denote by $\mathcal{P}(\bb{R}^d)$ the space of all Borel probability measures on $\RR^d$ and, for any $p\geq 1$, by $\mathcal{P}_{p}(\bb{R}^d)$ the subspace of $\mathcal{P}(\R^d)$ of measures with finite $p$-th moment, i.e.\
$$\mathcal{P}_{p}(\bb{R}^d)\coloneqq\Big\{\mu\in \ca{P}(\bb{R}^d), \;M_{p}(\mu):=\int_{\bb{R}^d}|y|^{p}\,\,\d \mu(y) <\infty\Big\} \, .$$
We equip $\mathcal{P}_p(\bb{R}^d)$  with the $p$-Wasserstein distance
\begin{align} 
	\ca{W}_p(\mu,\nu)\coloneqq\inf\limits_{\pi\in\mathcal{C}(\mu,\nu)}\Big(\int_{\RR^d\times \RR^d}|x-y|^p\;\pi(\d x \d y)\Big)^{\frac{1}{p}},
\end{align}
with $\mathcal{C}(\mu,\nu)$ being the set of all couplings of $\mu$ and $\nu$.

We now introduce some ideas from the theory of random dynamical systems.

Let $(E, \dist)$ be a Polish space with Borel $\sigma$-algebra $\ca{E}$ and $(\Omega,\ca{G},\PP,(\theta_t)_{t\in\RR})$ be a metric dynamical system in the sense of \cite[Appendix A, p. 535-537]{Arnold.1998.586} over the probability space $(\Omega,\ca{G},\PP)$. We now discuss random dynamical systems (RDS) 
in the sense of \cite[Definition 1.1.1]{Arnold.1998.586}. Since our applications are to RDS generated by SDEs driven by Brownian motion, we will assume that the RDS is adapted to a filtration and is of white-noise type. Namely,  we have a family $\bb{F}=(\ca{F}_{s,t})_{-\infty<s\leq t< \infty}$ of sub-$\sigma$ algebras of $\ca{F}$, the $\mathbb{P}$-completion of $\ca{G}$, such that 

(1) $\ca{F}_{t,u}\subseteq \ca{F}_{s,v}$ whenever $s\leq t\leq u\leq v$;

(2) $\theta_r^{-1}(\ca{F}_{s,t})=\ca{F}_{s+r,t+r}$ for all $r,s,t$;

(3) $\ca{F}_{s,t}$ and $\ca{F}_{u,v}$ are independent  whenever $s\leq t\leq u\leq v$.

\noindent For each $t\in\RR$, we  denote by  $\ca{F}_{t}$ the smallest $\sigma$-algebra containing all $\ca{F}_{s, t}, s\leq t$,  and by $\ca{F}_{t,\infty}$ the smallest $\sigma$-algebra containing all $\ca{F}_{t, u}, t\leq u$. Note that for each $t\in\RR$, the $\sigma$-algebra $\ca{F}_{t}$ and $\ca{F}_{t,\infty}$ are independent. We will assume further that for each $t\geq 0$ and $y\in E$, $\varphi(t,\cdot, y)$ is $\ca{F}_{0,t}$-measurable. The collection $(\Omega, \ca{G}, \bb{F},\bb{P},(\theta_t)_{t\in\RR},\varphi)$ is then called a \emph{white-noise RDS}. As a short-hand, we will simply call $\varphi$ a white-noise RDS.

Given an  RDS $\varphi$, we define the associated \emph{skew-product flow} by 
\begin{align}
		\Theta_t: &\Omega \times E\longrightarrow \Omega \times E\\
&(\omega, y)\longmapsto (\theta_t \omega, \varphi(t, \omega, y))\,.
\end{align}
 An invariant measure for an RDS $\varphi$ is a probability measure on $\Omega \times E$ whose marginal on $\Omega$ is $\mathbb{P}$ and which is invariant under $\Theta_t$ for all $t \geq 0$. Equivalently, $\mu$ is invariant for $\varphi$ if and only if $\varphi(t,\omega,\cdot)\mu_{\omega}= \mu_{\theta_t \omega}$ for all $t \geq 0$ and $\mathbb{P}$-almost every $\omega \in \Omega$, where $\omega \mapsto \mu_\omega$ is the unique disintegration of a probability measure $\mu$ on $\Omega \times E$ with marginal $\mathbb{P}$ on $\Omega$ (cf. \cite[Proposition 3.6]{Crauel2002Randommeasures}). An invariant measure $\mu_\omega$ is said to be a Markov  measure  if $\omega \mapsto \mu_\omega$ is measurable with respect to the past $\mathcal{F}_0$. For a white-noise RDS, we can define the associated Markov semigroup by
 \begin{align}
 (P_t f) (y) \coloneqq\EE [f(\varphi(t, \cdot, y))] \,
 \end{align}
 for all bounded, measurable functions $f$. Note , if $E$ is a Polish space, there is a one-to-one correspondence between invariant measures of $P_t$ and Markov invariant measures of $\varphi$ (\cite{Crauel1991Markovmeasures}): Let $\rho$ be $P_t$-invariant, then for any sequence $t_k \to \infty $ \footnote{The weak limit does not depend on the sequence $t_k$.}, the weak limit, i.e. the limit obtained when testing against bounded, continuous functions, 
\begin{align}
\mu_\omega\coloneqq \lim_{t_k \to \infty} \varphi(t_k, \theta_{-t_k}\omega, \cdot) \rho\,,
\label{eq:Markovmeasure}
\end{align}
exists for $\mathbb{P}$-a.s. $\omega \in \Omega$ and defines a Markov invariant measure for $\varphi$. Conversely, the measure $ \EE[\mu_{\cdot}]$ defines an invariant probability measure for  $P_t$.  

We now introduce some more dynamical concepts which will play a role in our proofs and then introduce the definition of weak synchronisation.
\begin{definition}[Strong mixing]\label{def:stronglymixing}
A Markov semigroup $(P_t)_{t\geq 0}$ with  the invariant measure $\rho$ is said to be \emph{strongly mixing} if, for all bounded, continuous functions $f$ and all $y\in E$,
\begin{align}
    P_tf(y)\to \int_E f(z) \,\d \rho(z)\,,\,\,\,\,\,\,\,\,\,\,\,\,\,\,\,{\rm{as}} \,\,\,t\to\infty\,.
\end{align}
 Similarly, we say that a white-noise RDS $\varphi$ is strongly mixing with respect to  the invariant measure $\rho$ if the law of $\varphi(t, \cdot, y)$ converges weakly to $\rho$ for all $y \in E$ as $t\to\infty$.
\end{definition}

\begin{definition}[$M$-Attractors]
	\label{def:attractor}
We say that a family of non-empty subsets $\{A(\omega)\}_{\omega \in \Omega}$  of $M\subseteq E$ is
\begin{enumerate}[i)]
\item \emph{Random closed (resp. compact)} if $A(\omega)$ is $\mathbb{P}$-a.s. closed (resp. compact) and the map $\omega \mapsto {\mathrm{dist}}(y, A(\omega))\coloneqq \inf_{a\in A(\omega)} {\rm{dist}}(y,a)$, is $\mathcal{F}$-measurable for every $y \in M$.
\item \emph{$\varphi$-invariant} if for all $t \geq 0$ and $\mathbb{P}$-a.s. $\omega\in\Omega$,
\begin{equation}
\varphi(t,\omega,A(\omega)) = A(\theta_t \omega) \, .
\end{equation}
\item A \emph{point $M$-attractor} if $A$ is random compact, $\varphi$-invariant, and for every $y\in M$, 
\begin{equation}\label{def:attractionconvergence}
\lim_{t \to \infty}  {\rm{dist}}(\varphi(t,\theta_{-t}\omega, y),A(\omega)) =0\,,\,\,\,\,\,\,\,\,\,\,\,\,\,\,\,\,\,\,\,\,\PP\text{-a.s.}
\end{equation}
A point  $M$-attractor is referred to as a \emph{weak point $M$-attractor}  if the above convergence holds in probability. A weak point $M$-attractor is said to be \emph{minimal} if it is contained in every weak point $M$-attractor.
\end{enumerate}
\end{definition}
In particular, if $M = E$, we simply drop the prefix $M$  from all the above definitions. We are now in a position to introduce the definition of weak synchronisation.

\begin{definition}[Weak synchronisation and conditional weak synchronisation]\label{def:weaksynch}
Consider a white-noise RDS $\varphi$. We say that \emph{weak synchronisation} occurs if there exists a %minimal 
weak point attractor $A(\cdot)$ which is   $\mathbb{P}$-a.s. a singleton. We say that \emph{conditional weak synchronisation} occurs on a subset $M\subsetneq E$ if there exists a %minimal 
weak point $M$-attractor $A(\cdot)$ which is   $\mathbb{P}$-a.s. a singleton. 
\end{definition}

\section{Weak synchronisation for multiplicative SDEs} \label{sec3:weaksynchroforSDE}
In this section, we focus on proposing sufficient conditions for weak synchronisation of RDS generated by SDEs with multiplicative noise. We first present general conditions for weak synchronisation at the level of abstract white-noise RDS, which serve as generalisations of the results  \cite[Theorem 2.23]{FlandoliGessScheutzow.2017.PTRF511}. We then apply these results to SDEs with multiplicative noise, and provide checkable conditions in terms of the coefficients of the SDEs. 
\subsection{Weak synchronisation at the level of RDS}
 Since we will ultimately be concerned with RDS stemming from SDEs driven by Brownian motion, we focus on white-noise RDS $(\Omega,\mathcal{G},\mathbb{F},\mathbb{P},(\theta_t)_{t\in \R},\varphi)$ with continuous trajectories. We  first introduce some notation. We denote by $\Delta$ the diagonal of the product space $E \times E$, i.e. 
 \begin{align}\label{def:Delta}
 	\Delta \coloneqq \{(y,y): y \in E\}\,.
 \end{align}
For all $\varepsilon>0$, we  define 
\begin{align}
\Delta_{ \varepsilon} &\coloneqq\{(y_1, y_2)\in E \times E: \dist(y_1, y_2) \leq  \varepsilon\}\,,
\end{align}
and  for any $a\in E$, denote by 
\begin{align}
\Delta_{a, \varepsilon} &\coloneqq \{(y_1, y_2)\in E \times E: \dist(y_1, a)+\dist(y_2, a) \leq \varepsilon\}\,.
\end{align}
We also need to define the following stopping times, for any $y_1, y_2\in E$,
\begin{align}
\tau_{(y_1,y_2), \varepsilon}(\cdot)&:= \inf\{t \geq 0: (\varphi(t,\cdot,y_1),\varphi(t,\cdot,y_2)) \in \Delta_{\varepsilon} \} \, ,\\[1mm]
\tau_{ (y_1,y_2), \varepsilon}^{a}(\cdot)&\coloneqq \inf\{t \geq 0: (\varphi(t,\cdot,y_1),\varphi(t,\cdot,y_2)) \in \Delta_{a, \varepsilon}\} \, .\label{def:stoppingtimea}
\end{align}

\begin{assumption}[$\Delta_{a,\varepsilon}$ is reachable]\label{assum:reachability}
For all  $\varepsilon>0$ and   $a\in E$,   every $(y_1, y_2) \in E\times E$,
\begin{equation}\label{ass:reachability}
\mathbb{P}\big(\tau_{(y_1, y_2), \varepsilon}^{a} < \infty\big) > 0 \, . \worknote{two-point motion can reach the $\varepsilon$-ball before any given time $T$ in a positive probability}
\end{equation}
\end{assumption}
Loosely speaking, the above assumption asserts that two trajectories of $\varphi$ starting from different initial points $y_1,y_2$ enter a small neighbourhood of the diagonal with positive probability, though not necessarily with probability one.
\begin{definition}[Weak asymptotic stability]\label{def:Weak asymptotic stability}
	Let $U\subset  E$ be a deterministic non-empty open set.  We say that $\varphi$ is \emph{weakly asymptotically stable} on $U$ if there exists a deterministic sequence $t_n\uparrow \infty$ and a set $\ca{M}\subseteq \Omega$ of positive $\PP$-measure, such that for all $x,y\in U$,
	\begin{align}
		\lim_{n\to\infty}{\bf{1}}_{\ca{M}}(\cdot){\rm{dist}}\big(\varphi(t_n,\cdot,x),\varphi(t_n,\cdot,y) \big)= 0\,,\,\,\,\,\,\,\text{in\,\, probability} .
	\end{align}
\end{definition}

\begin{theorem}\label{thm:weaksynchronisationRDS}
Assume that the white-noise RDS $\varphi:$
	\begin{thmenumerate}[label=\upshape(\roman*)]
 \item \label{thmassum:weaksynchronisationRDS1} is strongly mixing with respect to the invariant probability measure $\rho$;
 \item \label{thmassum:weaksynchronisationRDS2} is weakly asymptotically stable on some open set $U$ with $\rho(U)>0$;
 \item \label{thmassum:weaksynchronisationRDS3} satisfies~\cref{assum:reachability}. \end{thmenumerate}
  Then, there is a %minimal 
  weak point attractor consisting of a single random point $a(\omega)$ and 
  \begin{equation}
  	A(\omega)={\rm{supp}}(\mu_{\omega})=\{a(\omega)\},\,\,\,\,\,\PP-a.s.,
  \end{equation}
  where 
  \begin{equation}
  	\mu_{\omega}=\lim_{k\to\infty} \varphi(t_k, \theta_{-t_k}\omega)\rho,
  \end{equation}
  i.e. $\varphi$ exhibits weak synchronisation in the sense of~\cref{def:weaksynch}. 
\end{theorem}
\begin{remark}\label{rmk:additivestrongtransitivity}
	In contrast to \cite[Theorem 2.23]{FlandoliGessScheutzow.2017.PTRF511}, we replace the condition of pointwise strong swift transitivity by the reachability condition \eqref{ass:reachability}. The key observation is that pointwise strong swift transitivity is verifiable for SDEs with additive noise, since the two-point motion is driven by the same noise, which acts identically on both trajectories. For SDEs with multiplicative noise, however, the diffusion coefficients depend on the state, so that the same noise acts differently on the two trajectories. 
In \cref{subsubsec:reachability}, we will use the H\"{o}rmander condition to verify \cref{assum:reachability}.
\end{remark}

\vspace{3mm}
For notational simplicity, throughout the proof we write $\varphi_t(\omega,y)$ for $\varphi(t,\omega,y)$.

\begin{proof}
	
	 Because $\varphi$ is strongly mixing and  weakly asymptotically stable on an open set $U$ with $\rho(U)>0$, it follows from \cite[Lemma 2.19]{FlandoliGessScheutzow.2017.PTRF511} that there are $\ca{F}_0$-measurable random variables $a_i(\cdot), i=1,...,N$, such that 
\[
A(\omega):=\{a_i(\omega):i=1,...,N\}
\]
is a minimal weak point attractor. It remains to show that, $A(\cdot)$ is a singleton \( \mathbb{P} \)-a.s. i.e. $N=1$. 

 Let $a_1(\cdot), a_2(\cdot)$ be two different  $\ca{F}_0$-measurable elements. By weak asymptotic stability on an  open set  $U$, we have a sequence  $t_n \to \infty$  and a $\delta_0>0$, such that for all $y_1,y_2\in U$ and $\eta>0$,
\begin{align}\label{es:weak-asymptotic-stability}
	\liminf_{n\to\infty}\PP\big({\rm{dist}} (\varphi_{t_n}(\cdot,y_1),\varphi_{t_n}(\cdot,y_2))\leq \eta\big)\geq \delta_0>0.
\end{align}
Without loss of generality, we may assume  
$$U = B( x_0, \varepsilon_0)\subseteq E\,,$$  for some \( x_0 \in E \), \( \varepsilon_0 > 0 \).

We show the existence of a deterministic time $t_0$ at which two trajectories are close enough. Consider $\bar{B}( x_0, \varepsilon_0/2)\times \bar{B}( x_0, \varepsilon_0/2)\subset E\times E$. Using the independence of $\ca{F}_0$ and $\ca{F}_{0,\infty}$ and tower property of conditional expectation, we know that
\begin{align}
\PP\big(\omega\in \Omega: \tau_{(a_1(\omega), a_2(\omega)), \varepsilon_0/2}^{x_0}<\infty\big)=\,&\EE\Big[\EE\big[{\bf{1}}_{ \{  \tau_{(a_1(\omega), a_2(\omega)), \varepsilon_0/2}^{x_0} <\infty \}  } | \ca{F}_0\big] \Big]\\
=\,&\EE\Big[\EE\big[{\bf{1}}_{ \{  \tau_{(y_1, y_2), \varepsilon_0/2}^{x_0}<\infty  \}  } \big]\Big|_{y_1=a_1(\omega), y_2=a_2(\omega)} \Big]\\
>\,&0\,,
\end{align}
where we use the \cref{assum:reachability} in the last step. 
Using the continuity of the trajectories, there is an \(\ell : \Omega \to \mathbb{R}_+\) such that
\begin{align}
\PP\Big(\omega\in \Omega:\big(\varphi_{\tau_{(a_1(\omega), a_2(\omega)), \varepsilon_0/2}^{x_0}(\omega)+s}(&\omega, a_1(\omega)), \varphi_{\tau_{(a_1(\omega), a_2(\omega)), \varepsilon_0/2}^{x_0}(\omega)+s}(\omega, a_2(\omega))\big) \\
&\in B(x_0, \varepsilon_0)\times B(x_0, \varepsilon_0) ,\,s \in [0, \ell(\omega)]\Big)>0\,.\label{3epsilon}
\end{align}
We consider a countable family of balls $\{B(r_m, R_m)\}_{m \in \mathbb{N}}$ covering $\mathbb{R}_+$, where $r_m \in \mathbb{Q}$ and $R_m \in \mathbb{Q}_+$. 
Considering that 
\begin{align}
	&\sum_{i=1}^m\PP\Big(\omega\in \Omega:\big(\varphi_t(\omega, a_1(\omega)), \varphi_{t}(\omega, a_2(\omega))\big) \in B(x_0, \varepsilon_0)\times B(x_0, \varepsilon_0) ,\,t \in B(r_i, R_i)\Big)\\
	\geq\,\,&  \PP\Big(\omega\in \Omega:\big(\varphi_{\tau_{(a_1(\omega), a_2(\omega)), \varepsilon_0}^{x_0}(\omega)+s}(\omega, a_1(\omega)), \varphi_{\tau_{(a_1(\omega), a_2(\omega)), \varepsilon_0}^{x_0}(\omega)+s}(\omega, a_2(\omega))\big) \\
	&\,\,\,\,\,\,\,\,\,\,\,\,\,\,\,\,\,\,\,\,\,\,\,\,\,\,\,\,\,\,\,\,\,\,\,\,\,\,\,\,\,\,\,\,\,\,\,\,\in B(x_0, \varepsilon_0)\times B(x_0, \varepsilon_0) ,\,s \in [0, \ell(\omega)]\Big)\\
	>\,\,&0\,,
\end{align}
we obtain that  there is a deterministic time \( t_0 \geq 0 \) such that
\begin{align}
\delta\coloneqq\PP\Big(\omega\in \Omega:\big(\varphi_{t_0}(& \,\omega, a_1(\omega)), \varphi_{t_0}(\omega, a_2(\omega))\big) \in B(x_0, \varepsilon_0)\times B(x_0, \varepsilon_0) \Big)>0\,.\label{es:uniformdelta}
\end{align}

Notice that $\varphi_{t_0}\big( \cdot, a_1(\cdot)\big), \varphi_{t_0}\big( \cdot, a_2(\cdot)\big)\in \ca{F}_{t_0}$, $\varphi_{t_n}(\theta_{t_0}\cdot, y)\in \ca{F}_{t_0, t_0+t_n}$ for every $y\in E$, the $\sigma$-algebras 
$\ca{F}_{t_0}$ and $\ca{F}_{t_0, +\infty}$ are  independent.
 Using the  properties of cocycle and  conditional expectation, Fatou's lemma implies that  for all $\eta>0$,
\begin{align}            
&\liminf_{n \to \infty} \mathbb{P}\Big(\dist\big(\varphi_{t_0 +  t_n}(\cdot, a_1(\cdot)), \varphi_{t_0 +  t_n}(\cdot, a_2(\cdot))\big)\leq \eta \Big)\\
\geq \,\,\,&\liminf_{n \to \infty} \mathbb{P}\Big(\big\{\big(\varphi_{t_0}(\cdot, a_1(\cdot)), \varphi_{t_0}(\cdot, a_2(\cdot))\big) \in \Delta_{x_0, \varepsilon} \big\} \\
&\,\,\,\,\,\,\,\,\,\,\,\,\,\,\,\,\,\,\,\,\,\,\,\,\,\,\,\,\,\,\,\,\,\cap \big\{\dist\big(\varphi_{t_0 +  t_n}(\cdot, a_1(\cdot)), \varphi_{t_0 +  t_n}(\cdot, a_2(\cdot))\big)\leq \eta\big\} \Big)\\
 =\,\,\,& \liminf_{n \to \infty}\EE\Big[ {\bf{1}}_{\varphi_{t_0}\big( \cdot, \{a_1(\cdot), a_2(\cdot)\}\big) \in \Delta_{x_0, \varepsilon}},\\
 &\,\,\,\,\,\,\,\,\,\,\,\,\,\,\,\,\,\,\,\,\,\,\,\,\,\,\,\,\,\,\,\,\mathbb{P}\Big(\dist\big(\varphi_{t_n}(\theta_{t_0}\cdot,\varphi_{t_0}(\cdot, a_1(\cdot))), \varphi_{t_n}(\theta_{t_0}\cdot, 
                        \varphi_{t_0}(\cdot, a_2(\cdot)))\big) \leq \eta\Big|\ca{F}_{t_0} \Big)\Big]\\
                          =\,\,\,& \liminf_{n \to \infty}\EE\Big[{\bf{1}}_{\varphi_{t_0}\big(\cdot, \{a_1(\cdot), a_2(\cdot)\}\big) \in \Delta_{x_0, \varepsilon}},\\ &\,\,\,\,\,\,\,\,\,\,\,\,\,\,\,\,\,\,\,\,\,\,\,\,\,\,\,\,\,\,\,\,\mathbb{P}\Big(\dist(\varphi_{t_n}(\theta_{t_0}\cdot, x),\varphi_{t_n}(\theta_{t_0}\cdot, y))\leq \eta\Big)\Big|_{x=\varphi_{t_0}(\cdot, a_1(\cdot)),y=\varphi_{t_0}(\cdot, a_2(\cdot))}\Big]\\
                         \geq\,\,\,&\EE\Big[{\bf{1}}_{\varphi_{t_0}\big(\cdot, \{a_1(\cdot), a_2(\cdot)\}\big) \in \Delta_{x_0, \varepsilon}},\\ &\,\,\,\,\,\,\,\,\,\,\,\,\,\,\,\,\,\,\,\,\,\,\,\,\,\,\,\,\,\,\,\, \liminf_{n \to \infty}\mathbb{P}\Big(\dist(\varphi_{t_n}(\theta_{t_0}\cdot, x),\varphi_{t_n}(\theta_{t_0}\cdot, y))\leq \eta\Big)\Big|_{x=\varphi_{t_0}(\cdot, a_1(\cdot)),y=\varphi_{t_0}(\cdot, a_2(\cdot))}\Big] \,.\label{es:t0tn}
\end{align}
Since $\theta_t$ is measure-preserving for every $t\geq0$, it follows from \eqref{es:weak-asymptotic-stability} and \eqref{es:uniformdelta} that \eqref{es:t0tn} can be rewritten and estimated as follows:
\begin{align}
\,\,\,& \EE\Big[{\bf{1}}_{\varphi_{t_0}\big(\theta_{-t_0}\cdot, \{a_1(\theta_{-t_0}\cdot), a_2(\theta_{-t_0}\cdot)\}\big) \in\Delta_{x_0, \varepsilon}},\\ &\,\,\,\,\,\,\,\,\,\,\,\,\,\,\,\,\,\,\liminf_{n \to \infty}\mathbb{P}\Big(\dist(\varphi_{t_n}(\cdot,x), \varphi_{t_n}(\cdot, y)) \leq \eta\Big)\Big|_{x=\varphi_{t_0}(\theta_{-t_0}\cdot, a_1(\theta_{-t_0}\cdot)),y=\varphi_{t_0}(\theta_{-t_0}\cdot, a_2(\theta_{-t_0}\cdot))}\Big]\\
	\geq \,\,\,& \delta_0\delta\,.\label{t0t1tn}
\end{align}

\vspace{1mm}
We now show that $A(\cdot)$ is a $\mathbb{P}$-a.s. singleton.  Define \begin{align}\label{positivityofF}
    F(\omega) \coloneqq \min_{i, j = 1, \dots, N, i \neq j} {\rm{dist}}(a_i(\omega), a_j(\omega))\,.    
    \end{align}
It follows from the invariance property   
$\varphi_t(\omega, A(\omega)) = A(\theta_t \omega)\,,$
that for all $t\geq 0$,
\begin{align}
    F(\theta_t \omega) &=\, \min_{i, j = 1, \dots, N, i \neq j} {\rm{dist}}(a_i(\theta_t \omega), a_j(\theta_t \omega)) \\
    &=\, \min_{i, j = 1, \dots, N, i \neq j} {\rm{dist}}(\varphi_t(\omega, a_i(\omega)), \varphi_t(\omega, a_j(\omega)))\\
    &\leq\, {\rm{dist}}(\varphi_t(\omega, a_1(\omega)), \varphi_t(\omega, a_2(\omega))).
\end{align}
Using the measure-preserving property of $\theta$, \eqref{es:t0tn} and  \eqref{t0t1tn}, we obtain that for all \(\eta > 0\),
\begin{align}
   \mathbb{P}(F(\cdot) \leq \eta) &= \,\, \lim_{n\to\infty}\mathbb{P}(F(\theta_{t_0 + t_n} (\cdot)) \leq \eta) \\
    &\geq \,\, \liminf_{n\to\infty}\mathbb{P}\Big({\rm{dist}}\big(\varphi_{t_0 + t_n}(\cdot, a_1(\cdot)), \varphi_{t_0 + t_n}(\cdot, a_2(\cdot))\big) \leq \eta\Big)\\
    &\geq \,\,\delta_0\delta\,,
    \end{align}
    which implies 
\begin{align}\label{es:positivityofF}
	\PP(\omega\in \Omega:F(\omega)=0)=\bigcap_{n\geq 1}\PP\Big(\omega\in \Omega:F(\omega)\leq \frac{1}{n}\Big)=\lim_{n\to\infty}\PP\Big(\omega\in \Omega:F(\omega)\leq \frac{1}{n}\Big)\geq \delta_0\delta\,.
\end{align}
It follows from  $\varphi_t(\theta_{-t}\omega,A(\theta_{-t}\omega))=A(\omega)\,$
 that ${\rm{diam}}(A(\omega))=0$ provided that ${\rm{diam}}(A(\theta_{-t}\omega))=0$\,. Then $\PP$-a.s.
\begin{align}
	\{\omega\in \Omega:F(\theta_{-t}\omega)=0\}=\,\,&\big\{\omega\in \Omega: \min_{i, j = 1, \dots, N, i \neq j} \dist(a_i(\theta_{-t}\omega), a_j(\theta_{-t}\omega))=0\big\}\\
	\subseteq\,\,&\big\{ \omega\in \Omega: \min_{i, j = 1, \dots, N, i \neq j} \dist( a_i(\omega), a_j(\omega))=0\big\}\\
	=\,\,&\{\omega\in \Omega:F(\omega)=0\}\,.
\end{align}
Since $\theta_t$ is $\PP$ invariant, these events have the same $\PP$-mass and thus coincide almost surely. Note that $\{\omega:F(\theta_{-t}\omega)=0\}$ is $\ca{F}_{-t}$-measurable, hence 
\begin{align}
	\{\omega\in \Omega: F(\omega)=0\}\in \bigcap_{t\leq 0}\ca{F}_{t}\,.
\end{align}
Therefore, Kolmogorov's  $0$-$1$ law \footnote{Given a probability space $(\Omega, \ca{F}, \PP)$ and a sequence of independent events $A_1, A_2,\ldots\in\ca{F}$ with the tail field defined as 
	$$\ca{F}_{n, \infty}\coloneqq\cap_{n=1}^{\infty}\sigma(A_n, A_{n+1}, A_{n+2}, \ldots)\,.$$
	If $A\in \ca{F}_{n, \infty}$, then $\PP(A)=\{0, 1\}.$	} and  \eqref{es:positivityofF} imply that
\begin{align}
	\PP(\omega\in \Omega:F(\omega)=0)=1\,,
\end{align}
which completes the proof.
\end{proof}

\subsection{Weak synchronisation for multiplicative SDEs}\label{sec:synchr_SDE}

\cref{thm:weaksynchronisationSDE} provides three sufficient conditions: i.e. strong mixing, weak asymptotic stability, and reachability to guarantee weak synchronisation. These three conditions are formulated at the level of an abstract RDS. 

 In this section, we focus on multiplicative SDEs and identify classes of drift and diffusion coefficients such that the generated RDS  satisfies the above three conditions.
 
 \vspace{1mm}
  To begin with, we first introduce the probability space we work with and the white-noise random dynamical system associated with the multiplicative SDEs driven by Brownian motion.

\subsubsection{Generation of RDS for multiplicative SDEs}\label{sec:generationRDSforSDEs}

Assume that $b\in C^{1, 1}(\RR^d;\RR^d)$  and $ \sigma\in C^{1, 1}(\RR^d;\RR^{d\times m})$. 	Consider the multiplicative SDE
\begin{align}\label{eq:multiSDE00}
	\d Y_t=b(Y_t)\, \d t+\sigma(Y_t)\,\d W_t\,,
\end{align}
where $(W_t)_{t\geq 0}$ is an $m$-dimensional Brownian motion.

Let $(\bar{\Omega},\bar{ \ca{G}}, \bar{\PP}, (\bar{\theta}_t)_{t\in\RR})$ be the canonical metric dynamical system describing $\RR^m$-valued Brownian motion $W_t(\omega)\coloneqq \omega(t)$ as discussed in \cite[Appendix B]{Arnold.1998.586}. We consider 
\begin{equation}
	\bar{\Omega}=C_0(\RR; \RR^m)\coloneqq \{\omega\in C(\RR; \RR^m): \omega(0)=0\}\,,
\end{equation}
equipped with the compact-open topology and the associated Borel $\sigma$-algebra $\bar{\ca{G}}$, the shift
\begin{equation}
\bar{\theta}_t: \bar{\Omega}\to \bar{\Omega}, \,\,\,\,\,\bar{\theta}_t\omega(s)\coloneqq \omega(s+t)-\omega(t),\,\,\,s, t\in\RR,
\end{equation}
and $\bar{\PP}$ is the Wiener measure. Set $\bar{\ca{F}}$ to be the $\bar{\PP}$-completion of the Borel $\sigma$-algebra $\bar{\ca{G}}$. \footnote{Note that $(t, \omega)\mapsto \theta_t\omega$ is not $(\ca{B}(\RR)\otimes \bar{\ca{F}}, \bar{\ca{F}})$-measurable, but $(\ca{B}(\RR)\otimes \bar{\ca{G}}, \bar{\ca{G}})$-measurable.} Furthermore, we define a filtration $\bar{\bb{F}}\coloneqq (\bar{\ca{F}}_{s, t})_{-\infty<s\leq t<\infty}$, where 
\begin{equation}
	\bar{\ca{F}}_{s, t}\coloneqq \sigma(W_u-W_v: s\leq v\leq u \leq t)\vee \ca{N}, \,\,\,\,\,-\infty<s\leq t<\infty\,,
\end{equation}
and $\ca{N}$ are the null sets of $\bar{\ca{F}}$. It follows from \cite[Theorem 2.3.40]{Arnold.1998.586} that there is a $C^{1, \delta}$, $\delta<1$,  white-noise RDS $(\bar{\Omega}, \bar{\ca{G}}, \bar{\bb{F}}, \bar{\PP}, (\bar{\theta}_t)_{t\in \RR}, \bar{\varphi})$ corresponding to \eqref{eq:multiSDE00}. Moreover, for each $t\in [0, \infty)$ and $\omega\in \bar{\Omega}$, $\bar{\varphi}(t, \omega)$ is a $C^1$ diffeomorphism from  $\RR^d$ to $\RR^d$ \cite[P. 98, Theorem 2.3.40 ]{Arnold.1998.586}. 

Note however that the above construction is not sufficient to construct an RDS for the McKean--Vlasov system  as was done in \cite{gess2025random}. Since the construction of the RDS for the McKean--Vlasov system relies on rough path theory,  we need to include additional second-order information in the underlying metric dynamical system. Set 
\begin{equation}
	\Omega\coloneqq \bigcap_{k=1}^\infty \Omega_k \, ,\,\,\,\,\,\ca{G}\coloneqq \bar{\ca{G}}|_{\Omega},\,\,\,\,\, \mathbb{P}\coloneqq  \bar{\mathbb{P}}|_{\Omega}\,, 
	\end{equation}
	and 
	\begin{equation}
			\theta_t: \Omega\to \Omega, \,\,\,\,\,\theta_t\omega(s)\coloneqq \omega(s+t)-\omega(t),\,\,\,s, t\in\RR,
	\end{equation}
where for all $k=1,2,\dots$,
\begin{align}
	&\Omega_k\coloneqq \bigg\{\omega\in \bar{\Omega}:(W(\omega),\mathbb{W}^{\,\rm {Strat.}}(\omega))\in \scr{C}^\alpha([-k, k];\RR^m) \text{ is well-defined}, \quad  \\ 
	&\qquad\qquad\quad\,\Big(W^n(\omega),\int_{-k}^{\cdot} W^n(\omega) \otimes \d W^n(\omega)\Big) \xrightarrow{n \to \infty } (W,\mathbb{W}^{\,\rm {Strat.}})\in \scr{C}^\alpha([-k, k];\RR^m),\\
	&\qquad\qquad\quad\text{ for all }\alpha \in (1/3,1/2)  \bigg\} \, ,
\end{align}
$W^n_t$ is the dyadic approximation for  $W_t$ and $\bar{\mathbb{P}}|_{\Omega}$ denotes the restriction of $\bar{\PP}$ to $\Omega$. 

Then $\PP(\Omega)=1$, and for any $t\geq 0$, $\theta_t: \Omega\to \Omega$ is measure-preserving and $(\ca{B}(\RR)\otimes \ca{G}, \ca{G})$-measurable  (see \cite{gess2025random} for  more details). This completes the construction of metric dynamical system $(\Omega, \ca{G}, \PP, (\theta_t)_{t\in \RR})$. In regards of the filtration, we now set $\ca{F}$ be the $\PP$-completion of the $\sigma$-algebra $\ca{G}$ and  $\bb{F}\coloneqq (\ca{F}_{s, t})_{-\infty<s\leq t<\infty}$, where 
\begin{equation}
	\ca{F}_{s, t}\coloneqq \sigma(W_u-W_v: s\leq u\leq v\leq t)\vee \ca{N}, \,\,\,\,\,-\infty<s\leq t<\infty\,,
\end{equation}
and $\ca{N}$ are the null sets of $\ca{F}$. We will use this metric dynamical system as the underlying probability space for the RDS generated by the Mckean--Vlasov system \eqref{system3} in \cref{sec4:weaksynchroforMVSDE}.  

Since we will eventually make pathwise comparisons between multiplicative SDEs and McKean--Vlasov SDEs, we would like both the RDS to be defined over the same underlying metric dynamical system. Recall that  $\bar{\varphi}$ is the white-noise RDS corresponding to \eqref{eq:multiSDE00} constructed above, we define $\varphi\coloneqq \bar{\varphi}|_{\Omega}$. Then, one can show that $(\Omega,\ca{G},\bb{F},\mathbb{P},(\theta_t)_{t\in \R},\varphi)$ is also a white-noise RDS corresponding to \eqref{eq:multiSDE00}.

\subsubsection{Main result}
Now, we state the main result of this section. Some of the concepts involved will be defined later. The proof follows directly from a combination of \cref{prop:existenceLyapunovexponent}  and \cref{thm:reachability} below.
\begin{theorem}\label{thm:weaksynchronisationSDE}
Consider the RDS $\varphi$ corresponding to \eqref{eq:multiSDE00} and assume it is strongly mixing with respect to the invariant probability measure $\mu^{\infty}$ in the sense of \cref{def:stronglymixing}. 
Assume the following conditions:
\begin{thmenumerate}[label=\upshape(\roman*)]
	\item $\varphi$ satisfies  that  $\lambda_{\rm{top}}<0;$
	%$$\int_{\RR^d}\sup_{|r|=1}\left\langle r, D b\left(y\right)r\right\rangle \d \mu^{\infty}(y)+\frac{1}{2}\int_{\RR^d}\sup_{|r|=1}\big|D\sigma(y)r\big|^2 \d \mu^{\infty}(y)<0;$$
	
\item $b,  \sigma$ are smooth with  bounded derivatives up to order $2$ and the vector fields $U_1, U_2,\ldots,U_m,$ corresponding to two-point motion (see \cref{subsubsec:reachability}), satisfy H\"{o}rmander's condition for ${\bf{y}}\in \RR^{2d}\setminus \Delta$,

    \textbf{or}
\item[(ii')]  $b,  \sigma$ are smooth with  bounded derivatives up to order $2$ and the vector fields associated with  the two-point motion satisfy that $U_0$ is recurrent and $\{U_0, \ldots, U_m\}$ (see \cref{subsubsec:reachability}) fulfill the parabolic H\"{o}rmander condition on $\RR^{2d}\setminus \Delta$.
  %     \textbf{or}
%\item[(ii'')] $\varphi$ is pointwise strongly  swift transitive in the sense that there is a time $t>0$ such that for every $y_1, y_2\in \RR^d$ and every (arrival) point $y$,
%\begin{equation}
%	\PP\left( \varphi\left(t, \cdot, \{y_1, y_2\}\right)\subset B\left(y, |y_1-y_2|\right)\right)>0\,,
%\end{equation}
%and 
%\begin{equation}
%	\liminf_{t\to\infty} |\varphi(t, \omega, y_1)-\varphi(t, \omega, y_2)|=0,\,\,\,\,\bb{P}-{\rm{a.s.}}.
%\end{equation}
\end{thmenumerate}
Then, weak synchronisation occurs for $\varphi$ in the sense of \cref{def:weaksynch}.
	\end{theorem}

In \cref{subsubsec:Lyapunov}, we show that condition (i) in~\cref{thm:weaksynchronisationSDE} implies  that  $\varphi$ is weakly asymptotically stable. We prove that \cref{assum:reachability} is implied by condition (ii) or condition (ii') of~\cref{thm:weaksynchronisationSDE} in \cref{subsubsec:reachability}. We refer to ~\cite{FlandoliGessScheutzow.2017.PTRF511} for conditions under which condition (ii") holds and why it implies the reachability of $\varphi$. We remark that the pointwise strong swift transitivity condition is usually verifiable for SDEs with additive noise (see \cref{rmk:additivestrongtransitivity}).

\subsubsection{Lyapunov exponent and weak asymptotic stability}\label{subsubsec:Lyapunov}

In this subsection, we verify weak asymptotic stability for multiplicative SDEs using a negative top Lyapunov exponent and an appropriate integrability condition.

\begin{lemma}\cite[Proposition 2.20]{FlandoliGessScheutzow.2017.PTRF511}
	If $\varphi$ is strongly mixing and weakly asymptotically stable on $U$ with $\mu^{\infty}(U)>0$, then there is an $N\in \bb{N}$ and $\ca{F}_0$-measurable random variables $a_1,\ldots,a_N$ such that
	\begin{align}
		A(\omega)={\rm{supp}}(\mu_{\omega})=\{a_i(\omega):i=1,\ldots,N\}
	\end{align}
	is a minimal weak point attractor.
\end{lemma}
\begin{proposition}\cite[Lemma 3.1]{FlandoliGessScheutzow.2017.PTRF511}\label{prop:asymptotic-stability}
	Assume that $$\varphi(t, \omega,\cdot)\in C^{1,\delta}_{\rm{loc}}$$ for some $\delta\in (0,1)$ and all $t\geq 0$. Moreover, $P_t$ has an ergodic invariant measure $\rho$ such that
	\begin{align}
		\EE\int_{\RR^d} \log^+\|D\varphi(1, \omega, x)\|\,\d\rho(y)<\infty.
	\end{align}
	Then, the Lyapunov spectrum $\lambda_{N}<\ldots<\lambda_1$ exist such that
	\begin{align}
		\lim_{n\to \infty}\frac{\log |D\varphi(n, \omega, x)v|}{n}\in \{\lambda_i\}_{i=1}^{N},
	\end{align}
	for all $v \in \RR^d\setminus \{0\}$ and $\PP\otimes \rho$-a.s. $(\omega, x)\in \Omega\times \RR^d$.
	
	Furthermore, we assume that the top Lyapunov exponent $\lambda_{\rm{top}}:= \lambda_1<0$
	and 
	\begin{align}	\EE\int_{\RR^d}\log^+(\|\varphi(1, \omega, \cdot+x)-\varphi(1,\omega, x)\|_{C^{1, \delta}(\bar{B}(0,1))})\,\d \rho(x)<\infty.
	\end{align}
	Then, $\varphi$ is  weakly asymptotically stable on a (deterministic) non-empty, open set $U$.
\end{proposition}

For our purpose of estimating the $C^{\delta}$-norm of parameter-dependent stochastic flow, we modify the classical Kolmogorov's continuity theorem \cite[Theorem 3.15]{eberle2019stochastic} as follows.

\begin{proposition}[Parameter-dependent Kolmogorov's continuity theorem]\label{prop:modifiedKolmogorov}
	Let $(E, \|\cdot\|)$ be a Banach space, and for each $y\in \RR^d$, let  $X_{u+y}:\Omega\to E$, where $u\in \bar{B}(0,1)$, be a stochastic process valued at $E$. Suppose that there exist constants $q, \varepsilon>0$ and a function $F:\RR^d\to \RR_+$ such that 
\begin{align}\label{es:qpowerexpectation}
		\EE\big[\|X_{u+y}-X_{v+y}\|^q\big]\leq F(y)|u-v|^{d+\varepsilon}\,\,\,\,\,\,\,    \text{for any}\,\, u,v\in \bar{B}(0,1).
	\end{align} 
	Then, there exists a modification $(\xi_{u+y})_{u\in \bar{B}(0,1)}$ of $(X_{u+y})_{u\in \bar{B}(0,1)}$ such that 
	\begin{align}\label{es:KolmogrovHoldernorm}
		\EE\Big[\Big(\sup_{u\neq v\in\bar{B}(0,1)}\frac{\|\xi_{u+y}-\xi_{v+y}\|}{|u-v|^{\alpha}}\Big)^q\Big]\leq C_{d, q, \varepsilon, \alpha} F(y) \,\,\,\,\,\,\, \text{for any}\,\, \alpha\in [0,\varepsilon/q)\,. 
	\end{align}
\end{proposition}
\begin{proof}
	Define $\tilde{u}=u+y$. Then \eqref{es:qpowerexpectation} implies that for any $\tilde{u},\tilde{v}\in \bar{B}(y,1)$,
	\begin{align}
		\EE\big[\|X_{\tilde{u}}-X_{\tilde{v}}\|^q\big]\leq F(y)|\tilde{u}-\tilde{v}|^{d+\varepsilon}\,,
	\end{align}
	and for any $\beta<2d+\varepsilon$,
	\begin{align}
		&\EE\Big[\int_{\bar{B}(y,1)}\int_{\bar{B}(y,1)}\frac{\|X_{\tilde{u}}-X_{\tilde{v}}\|^q}{|\tilde{u}-\tilde{v}|^{\beta}}\,\d \tilde{u}\d \tilde{v}\Big]\\
		\leq\,\,&\int_{\bar{B}(y,1)}\int_{\bar{B}(y,1)}\frac{\EE\|X_{\tilde{u}}-X_{\tilde{v}}\|^q}{|\tilde{u}-\tilde{v}|^{\beta}}\,\d \tilde{u}\d \tilde{v}\\
		\leq \,\,&F(y) \int_{\bar{B}(y,1)}\int_{\bar{B}(y,1)}|\tilde{u}-\tilde{v}|^{d+\varepsilon-\beta}\d \tilde{u}\d \tilde{v}\\
		\leq\,\, &F(y)\big|\bb{S}^{d-1}\big|\int_0^2 r^{d+\varepsilon-\beta}r^{d-1}\d r\\
		<\,\,&\infty\,.
	\end{align}
	Hence, 
	\begin{align}
		\int_{\bar{B}(y,1)}\int_{\bar{B}(y,1)}\frac{\|X_{\tilde{u}}-X_{\tilde{v}}\|^q}{|\tilde{u}-\tilde{v}|^{\beta}}\,\d \tilde{u}\d \tilde{v}<\infty\,,\,\,\,\,\,\,\text{almost surely}.
	\end{align}
	Employing  \cite[Theorem 3.14]{eberle2019stochastic}, we obtain that there exists a modification $(\xi_{\tilde{u}})_{\tilde{u}\in \bar{B}(y,1)}$ and the Besov--H\"{o}lder embedding yields that for any %$\alpha\in [0,\varepsilon/q)$, 
    $2d+q\alpha<\beta<2d+\varepsilon$,
	\begin{align}
\EE\Big[\Big(\sup_{\tilde{u}\neq \tilde{v}\in \bar{B}(y,1)}\frac{\|\xi_{\tilde{u}}-\xi_{\tilde{v}}\|}{|\tilde{u}-\tilde{v}|^{\alpha}}\Big)^q\Big]&\leq\,\Big(\frac{8}{\log 2}\frac{\beta}{\beta-2d}\Big)^qd^{\frac{\beta}{2}}\EE\Big[\int_{\bar{B}(y,1)}\int_{\bar{B}(y,1)}\frac{\|X_{\tilde{u}}-X_{\tilde{v}}\|^q}{|\tilde{u}-\tilde{v}|^{\beta}}\,\d \tilde{u}\d \tilde{v}\Big]\\
		&\leq \Big(\frac{8}{\log 2}\frac{\beta}{\beta-2d}\Big)^qd^{\frac{\beta}{2}}F(y)\big|\bb{S}^{d-1}\big|\int_0^2 r^{d+\varepsilon-\beta}r^{d-1}\d r\\
		&\leq\,C_{d, q, \varepsilon, \alpha} F(y)\,,
	\end{align}
	which implies the estimate \eqref{es:KolmogrovHoldernorm}.
\end{proof}

\begin{lemma}\label{prop:existenceLyapunovexponent}
 Consider the RDS $\varphi$ corresponding to \eqref{eq:multiSDE00} and assume it is strongly mixing with respect to the invariant probability measure $\mu^{\infty}$ in the sense of \cref{def:stronglymixing}. We then have the following results.
\begin{thmenumerate}[label=\upshape(\roman*)]
 	\item $\begin{aligned}
 		\bb{E}\int_{\RR^d} \log^+\|(D\varphi)(1, \omega, y)\|\,\d \mu^{\infty}(y)<\infty,
 	\end{aligned}$ where $\log^+(y)\coloneqq \max (\log y, 0)$. It follows from this that   there are constants $\lambda_N<\ldots <\lambda_1$ such that for all $v\in \RR^d\setminus \{0\},   \PP\otimes\mu^{\infty}-{\rm{a.s.}}\,\, (\omega, y)\in\Omega\times \RR^d$,
 	$$\lim\limits_{n\to\infty}\frac{1}{n}\log|(D\varphi)(n, \omega,y)v|\in\{\lambda_i\}_{i=1}^N;$$
 	\item Let $\lambda_{\rm{top}}\coloneqq \lambda_1$. Then, we have the following upper bound on the top Lyapunov exponent
 	\begin{align}
 		\lambda_{\rm{top}}\leq \,\int_{\RR^d}\sup_{|r|=1}\left\langle r, D b\left(y\right)r\right\rangle \d \mu^{\infty}(y)+\frac{1}{2}\int_{\RR^d}\sup_{|r|=1}\big|D\sigma(y)r\big|^2 \d \mu^{\infty}(y)\,.\label{es:Lyaformula}
 	\end{align}

 	\item It holds that  \begin{equation}\label{assu:intergrability-for-aysmptotic-stability}
 		\int_{\RR^d}\int_{\Omega}\log^+(\|\varphi(1, \omega, \cdot+y)-\varphi(1, \omega, y)\|_{C^{1, \delta}(\bar{B}(0,1))})\,\d\PP\otimes\mu^{\infty}<\infty\,.
 	\end{equation} 
 	Together with $\lambda_{\rm{top}}<0$, this implies that $\varphi$ is  weakly asymptotically stable  in the sense of \cref{def:Weak asymptotic stability}.
 \end{thmenumerate}
  
\end{lemma}

\begin{proof} For notational simplicity, throughout the proof we write $\varphi_t(\omega,y)$ for $\varphi(t,\omega,y)$. We divide the proof into three steps corresponding to the three statements in the theorem:

	\emph{Proof of} (i) Differentiating both sides of  \eqref{eq:multiSDE00}, we obtain for any $y\in\RR^d$,
	\begin{equation}
		D \varphi_{t}(y)={\bf{I}}_{d\times d}+\int_0^t  (D b)(\varphi_s(y)) (D \varphi_{s}(y)) \,\d s+\int_0^t  (D\sigma)(\varphi_s(y)) (D \varphi_{s}(y))\,\d W_{s}\,.
	\end{equation}
It follows from the boundedness of $D b$ and $D \sigma$, It\^{o}'s isometry,  H\"{o}lder's inequality, and Gr\"{o}nwall's inequality that for any $t\geq 0$,
	\begin{align}
		\sup_{y\in\RR^d}\EE|D\varphi_t(y)|^2\leq C_t\,.\label{es:firstderivativeexpectation}
	\end{align}
Then, Fubini's theorem and Jensen's inequality yield that
\begin{align}
	\EE\int_{\RR^d} \log^+\|D\varphi_1(\omega, y)\|\,\d\mu^{\infty}(y)
	\leq \,\,\int_{\RR^d} \log^+\EE(\|D\varphi_1(\omega, y)\|^2+1)\,\d\mu^{\infty}(y)<\,\,\infty\,.
\end{align}
Hence, using \cref{prop:asymptotic-stability}, we conclude the proof of (i).

\emph{Proof of} (ii) The directional derivative $D \varphi_{t}(y)\cdot v$, where $y\in\RR^d$ and $v\in \RR^d \setminus \{0\}$, satisfies 
\begin{equation}\label{eq:Dvarphiv}
	D \varphi_{t}(y)\cdot v=v+\int_0^t  (D b)(\varphi_s(y)) (D \varphi_{s}(y))\cdot v \,\d  s+\int_0^t  (D \sigma)(\varphi_s(y)) (D \varphi_{s}(y))\cdot v\,\d W_{s}\,.
\end{equation}
Applying It\^o's formula, we have that
\begin{align}\label{eq:v_tsquare}
\left|D \varphi_{t}(y)\cdot v\right|^{2}= & |v|^2+ 2 \int_{0}^{t}\left\langle D \varphi_{s}(y)\cdot v, (D b)\left(\varphi_{s}(y)\right)D \varphi_{s}(y)\cdot v\right\rangle + \big|(D\sigma)(\varphi_s(y))D \varphi_{s}(y)\cdot v\big|^2 \d s \\
& +2\int_{0}^{t}\left\langle D \varphi_{s}(y)\cdot v, (D\sigma)(\varphi_s(y)) D \varphi_{s}(y)\cdot v  \,\d W_{s}\right\rangle\,.
\end{align}
Again, It\^{o}'s formula  \footnote{Since $\varphi_t(\omega, \cdot)$ is a $C^1$-diffeomorphism, $D\varphi_t(\omega, \cdot)$ is invertible for every $t\geq 0$. Hence, for every $v\neq 0$, $D\varphi_t(\omega, y)\cdot v\neq 0$, $t\geq 0$. Therefore, although $x\mapsto \log |x|$ is singular at origin, It\^{o}'s formula can be applied to $\log |D\varphi_t(\omega, y)\cdot v|$, since the process $D\varphi_t(\omega, y)\cdot v$ never reach the origin.} leads us to 
\begin{align}
\log \left(\left|D \varphi_{t}(y)\cdot v\right|^{2}\right)= \,\,\,&\log (\left|v\right|^{2})+ \int_{0}^{t} \frac{2}{\left|D \varphi_{s}(y)v\right|^{2}}\left\langle D \varphi_{s}(y)\cdot v, (D b)\left(\varphi_{s}(y)\right)D \varphi_{s}(y)\cdot v\right\rangle\,\d s\\
& +\int_0^t \frac{1}{\left|D \varphi_{s}(y)\cdot v\right|^{2}}\big|(D \sigma)(\varphi_s(y))D \varphi_{s}(y)\cdot v\big|^2 \d s \\
& -2 \int_{0}^{t} \frac{1}{\left(\left|D \varphi_{s}(y)\cdot v\right|^{2}\right)^{2}} |(D \sigma(\varphi_s(y))D \varphi_{s}(y)\cdot v)^TD \varphi_{s}(y)\cdot v|^2 \,\d s \\ 
&+2\int_{0}^{t} \frac{1}{\left|D \varphi_{s}(y)\cdot v\right|^{2}} \left\langle D \varphi_{s}(y)\cdot v, (D \sigma)(\varphi_s(y)) D \varphi_{s}(y)\cdot v  \,\d W_{s}\right\rangle\,.
\end{align}
Then, we have
\begin{align}
	\frac{1}{t}\log |D\varphi_t(y)\cdot v|%\leq \,\,\,&\frac{1}{2t}\log \left(\left|\nabla \varphi_{t}(y)v\right|^{2}\right)\\[1mm]
	=\,\,\, &\frac{1}{2t}\log (\left|v\right|^{2})+\frac{1}{t}\int_{0}^{t} \frac{1}{\left|D \varphi_{s}(y)\cdot v\right|^{2}}\left\langle D \varphi_{s}(y)\cdot v, (D b)\left(\varphi_{s}(y)\right)D \varphi_{s}(y)\cdot v\right\rangle\,\d s\\
& +\frac{1}{2t}\int_0^t \frac{1 }{\left|D \varphi_{s}(y)\cdot v\right|^{2}}\big|(D \sigma)(\varphi_s(y))D \varphi_{s}(y)\cdot v\big|^2 \d s \\
& -\frac{1}{t} \int_{0}^{t} \frac{1}{\left(\left|D \varphi_{s}(y)\cdot v\right|^{2}\right)^{2}} |(D \sigma(\varphi_s(y))D \varphi_{s}(y)\cdot v)^TD \varphi_{s}(y)\cdot v|^2 \,\d s \\ 
&+\frac{1}{t}\int_{0}^{t} \frac{1}{\left|D \varphi_{s}(y)\cdot v\right|^{2}} \left\langle D \varphi_{s}(y)\cdot v, (D \sigma)(\varphi_s(y)) D \varphi_{s}(y)\cdot v  \,\d W_{s}\right\rangle\\[1mm]
%&\leq \int_{\RR^d\times S^{d-1}} \left\langle z, D b\left(y\right)z\right\rangle \d \mu(y,z)+ \frac{1}{2t}\int_{0}^{t} \frac{1}{\left|v_{s}\right|^{2}}\big|\\ Here I still don't want to apply ergodic theorem to coupled system $(\varphi_s(y), \frac{v_s}{|v_s|}), which remains to be a Markov process, because otherwise I need to apply Ito formula to $f(x)=\frac{x}{|x|}$$
\leq \,\,\, &\frac{1}{2t}\log (\left|v\right|^{2})+\frac{1}{t} \int_{0}^{t} \left\langle \frac{D \varphi_s(y)\cdot v}{|D \varphi_s(y)\cdot v|}, D b\left(\varphi_{s}(y)\right)\frac{D \varphi_s(y)\cdot v}{|D \varphi_s(y)\cdot v|}\right\rangle \d s\\
&+ \frac{1}{2t}\int_{0}^{t} \left|(D \sigma)(\varphi_s(y))\frac{D\varphi_s(y)\cdot v}{|D \varphi_s(y)\cdot v|}\right|^2 \d s \\
&+\frac{1}{t}\int_{0}^{t}  \left\langle \frac{D \varphi_s(y)\cdot v}{|D \varphi_s(y)\cdot v|}, D \sigma(\varphi_s(y)) \frac{D \varphi_s(y)\cdot v}{|D \varphi_s(y)\cdot v|}  \,\d W_{s}\right\rangle\,.
\end{align} 
Passing to the limit $t\to\infty$ and applying Birkhoff's ergodic theorem leads us to
	\begin{align}
	\lambda_{\rm{top}}&\leq \,\int_{\RR^d}\sup_{|r|=1}\left\langle r, D b(y)r\right\rangle \d \mu^{\infty}(y)+\frac{1}{2}\int_{\RR^d}\sup_{|r|=1}\big|D \sigma(y)r\big|^2 \d \mu^{\infty}(y)\,,\label{es:Lyaformula1}
\end{align}
where in the last inequality, we use the fact that $\PP$-a.s.
\begin{align}\label{es:martingalelimsup}
	\lim_{t\to\infty}\frac{1}{t}\int_{0}^{t}  \left\langle \frac{D \varphi_s(y)\cdot v}{|D \varphi_s(y)\cdot v|}, D \sigma(\varphi_s(y))\frac{D \varphi_s(y)\cdot v}{|D \varphi_s(y)\cdot v|}  \,\d W_{s}\right\rangle= 0\,.
\end{align}
Indeed, let $r_s\coloneqq \frac{D \varphi_s(y)\cdot v}{|D \varphi_s(y)\cdot v|}$, and note that the quadratic variation of the stochastic integral $M_t$ in \eqref{es:martingalelimsup} is given by
\begin{equation}
\langle M\rangle_t=\int_{0}^{t} \big|\big(D \sigma(\varphi_s(y))r_s\big)^{T}r_s\big|^2\,\d s 
\leq C_{\sigma} t \, ,		
\end{equation}
where we use the boundedness of $D \sigma$. By Chebyshev's inequality,  for any $\delta>0$,
\begin{equation}
	\PP\left(|M_{k^2}|\geq k^2 \delta \right) \leq C_{\sigma}\frac{1}{k^2\delta^2}\, .
\end{equation}
The right-hand side is summable in $k$, and thus, by the Borel--Cantelli lemma, we have that $\PP$-a.s.  for any $\delta>0$, there exists $n_0$ such that for all $k\geq n_0$, $M_{k^2}/k^2<\delta$, meaning that $M_{k^2}/k^2\to 0$ as $k\to\infty$. Moreover,
\begin{align}
\PP \left(\sup_{t\in [k^2, (k+1)^2)}\frac{|M_t|}{t}\geq \delta\right) \leq \,\,&\, \PP \left(\frac{1}{k^2}\sup_{t\in [k^2, (k+1)^2)}|M_t|\geq \delta\right) \\
\leq \,\,&\, \PP \left(\frac{1}{k^2}\sup_{t\in [k^2, (k+1)^2)}|M_t-M_{k^2}|+\frac{|M_{k^2}|}{k^2}\geq \delta\right) \\
\leq \,\,& \, \PP \left(\frac{1}{k^2}\sup_{t\in [k^2, (k+1)^2)}|M_t-M_{k^2}|\geq \delta/2\right)+\PP\left(\frac{|M_{k^2}|}{k^2}\geq \delta/2\right) \\
\leq \,\,& \, \frac{4\, \EE\left(\sup_{t\in [k^2, (k+1)^2)}|M_t-M_{k^2}|^2\right) }{k^4\delta^2}+C_{\sigma}\frac{4}{k^2\delta^2}\,.\label{es:averagemartingale}
\end{align}
By the Burkholder--Davis--Gundy inequality,
\begin{align}
 \EE\left(\sup_{t\in [k^2, (k+1)^2)}|M_t-M_{k^2}|^2\right)=\,\,& \EE\left(\sup_{t\in [k^2, (k+1)^2)}\Big|\int_{k^2}^{t}  \left\langle r_s, D \sigma(\varphi_s(y))\,r_s  \,\d W_{s}\right\rangle\Big|^2\right)\\
 \leq\,\,&4\,\EE\int_{k^2}^{(k+1)^2}   \big|\big(D \sigma(\varphi_s(y))r_s\big)^{T}r_s\big|^2\,\d s\\
 \leq\,\,&4\, C_{\sigma}\big((k+1)^2-k^2\big)\,.
\end{align}
Substituting this estimate into \eqref{es:averagemartingale}, we deduce that 
\begin{align}
\PP \left(\sup_{t\in [k^2, (k+1)^2)}\frac{|M_t|}{t}\geq \delta\right) \leq \,\,&\frac{16\, C_{\sigma}(2k+1) }{k^4\delta^2}+C_{\sigma}\frac{4}{k^2\delta^2}\,,
\end{align}
which is summable in $k$ as well. Therefore,   the Borel--Cantelli lemma yields  \eqref{es:martingalelimsup}.

\emph{Proof of} (iii) Now we turn to the estimate \eqref{assu:intergrability-for-aysmptotic-stability}. By the definition of the $C^{1,\delta}$-norm, we have, for any $y\in\mathbb{R}^d$, 
\begin{align}
	&\EE\big\|\varphi_1(\omega, \cdot+y)-\varphi_1(\omega, y)\big\|_{C^{1, \delta}(\bar{B}(0,1))}\\[1mm]
	=\,\,\,&\EE\big\|\varphi_1(\omega, \cdot+y)-\varphi_1(\omega, y)\big\|_{C^{0}(\bar{B}(0,1))}+\EE\big\|D\big(\varphi_1(\omega, \cdot+y)-\varphi_1(\omega, y)\big)\big\|_{C^{0}(\bar{B}(0,1))}\\
	&+\EE\Big[D\big(\varphi_1(\omega, \cdot+y)-\varphi_1(\omega, y)\big)\Big]_{C^{\delta}(\bar{B}(0,1))}\\[1mm]
	\leq\,\,\,&2\EE\big\|D\varphi_1(\omega, \cdot+y)\big\|_{C^{0}(\bar{B}(0,1))}+\EE\Big[D\varphi_1(\omega, \cdot+y)\Big]_{C^{\delta}(\bar{B}(0,1))}\\[1mm]
	\leq\,\,\,&2\EE\big\|D\varphi_1(\omega, \cdot+y)\big\|_{C^{0}(\bar{B}(0,1))}\\
    &+\bigg(\EE\Big(\sup_{u, v\in \bar{B}(0,1)}\frac{|D\varphi_1(\omega, u+y)-D\varphi_1(\omega, v+y)|}{|u-v|^{\delta}}\Big)^{2p}\bigg)^{\frac{1}{2p}}\,,\label{Cdeltanorm}
\end{align}
where we used Jensen's inequality in the third inequality, and $p\geq 1$ is a constant to be chosen later. We proceed to estimate the second term  of \eqref{Cdeltanorm}. Since for any $t\geq 0$, $y\in\RR^d$,
\begin{align}
	&(D\varphi_t)(u+y)-(D\varphi_t)(v+y)\\
	%=\,\,&\int_0^t (D b)(\varphi_s(u+y))D\varphi_s(u+y)\,\d s+\sum_{k=1}^d\int_0^t(D \sigma^k)(\varphi_s(u+y))D\varphi_s(u+y)\,\d W_s^k\\
	%&-\int_0^t (D b)(\varphi_s(v+y))D\varphi_s(v+y)\,\d s-\sum_{k=1}^d\int_0^t(D \sigma^k)(\varphi_s(v+y))D\varphi_s(v+y)\,\d W_s^k\\
	=\,\,\,&\int_0^t \Big((D b)(\varphi_s(u+y))-(D b)(\varphi_s(v+y))\Big)D\varphi_s(u+y)\,\d s\\
	&+\int_0^t (D b)(\varphi_s(v+y))\Big(D\varphi_s(u+y)-D\varphi_s(v+y)\Big)\,\d s\\
	&+\int_0^t\Big((D\sigma)(\varphi_s(u+y))-(D\sigma)(\varphi_s(v+y))\Big)D\varphi_s(u+y)\,\d W_s\\
	&+\int_0^t(D\sigma)(\varphi_s(v+y))\Big(D\varphi_s(u+y)-D\varphi_s(v+y)\Big)\,\d W_s\,.
\end{align}
Hence, for any $p\geq 1$, by the Burkholder–Davis–Gundy inequality, Jensen’s inequality, and H\"older’s inequality, we obtain
\begin{align}
	&\EE\big|D\varphi_1(u+y)-D\varphi_1(v+y)\big|^{2p}\\
	\lesssim_{p}\,\,& \int_0^1 \EE\Big(\big|\varphi_s(u+y)-\varphi_s(v+y)\big|^{2p}\big|(D\varphi_s)(u+y)\big|^{2p}\Big)\,\d s\\
	&+\int_0^1 \EE\big|(D\varphi_s)(u+y)-(D\varphi_s)(v+y)\big|^{2p}\,\d s\\
	\lesssim_{p}\,\,& \int_0^1 \Big(\EE\big|\varphi_s(u+y)-\varphi_s(v+y)\big|^{4p}\Big)^{\frac{1}{2}}\Big(\EE\big|(D\varphi_s)(u+y)\big|^{4p}\Big)^{\frac{1}{2}}\,\d s\\
	&+\int_0^1 \EE\big|(D\varphi_s)(u+y)-(D\varphi_s)(v+y)\big|^{2p}\,\d s\\
	\lesssim_{p}\,\,&\int_0^1 \Big(\sup_{z\in\RR^d}\EE|D \varphi_s(z)|^{4p}|u-v|^{4p}\Big)^{\frac{1}{2}}\Big(\EE\big|(D\varphi_s)(u+y)\big|^{4p}\Big)^{\frac{1}{2}}\,\d s\\
	&+\int_0^1 \EE\big|(D\varphi_s)(u+y)-(D\varphi_s)(v+y)\big|^{2p}\,\d s\\
	\lesssim_{p}\,\,&|u-v|^{2p}+\int_0^1 \EE\big|(D\varphi_s)(u+y)-(D\varphi_s)(v+y)\big|^{2p}\,\d s\,,
\end{align}
where we use a similar estimate as \eqref{es:firstderivativeexpectation}  in the last inequality. Hence, Gr\"{o}nwall's inequality leads to  
\begin{align}
	\sup_{y\in\RR^d}\EE\big|(D\varphi_1)(\omega, u+y)-(D\varphi_1)(\omega, v+y)\big|^{2p}\leq C_{p}|u-v|^{2p}\,.
\end{align}
According to the parameter-dependent Kolmogorov's continuity theorem established in  \cref{prop:modifiedKolmogorov}, we conclude that for any  $p>\frac{d}{2(1-\delta)}$, %$2p>d$, and $\delta\in[0,(2p-d)/2p$),
\begin{align}
	\EE\Big(\sup_{u,v\in \bar{B}(0,1)}\frac{|(D\varphi_1)(\omega, u+y)-(D\varphi_1)(\omega, v+y)|}{|u-v|^{\delta}}\Big)^{2p}\leq C_{p}\,.
\end{align}
For every $u\in\bar B(0,1)$, taking $v=0$ yields
\begin{align}
	|D\varphi_1(\omega,u+y)|
	\leq
	|D\varphi_1(\omega,y)|
	+
	\sup_{u,v\in\bar B(0,1)}
	\frac{|D\varphi_1(\omega,u+y)-D\varphi_1(\omega,v+y)|}
	{|u-v|^\delta}.
\end{align}
Therefore, by \eqref{es:firstderivativeexpectation} and H\"older's inequality,
\begin{align}
	\EE\|D\varphi_1(\omega,\cdot+y)\|_{C^0(\bar B(0,1))}
	\leq C,
\end{align}
uniformly in $y\in\RR^d$. In particular,
\begin{align}
	\EE\|D\varphi_1(\omega,\cdot)\|_{C^0(\bar B(0,1))}
	+
	\EE\|D\varphi_1(\omega,\cdot+y)\|_{C^0(\bar B(0,1))}
	\leq C.
\end{align}
Together with \eqref{Cdeltanorm}, we complete  the proof of the estimate \eqref{assu:intergrability-for-aysmptotic-stability}.

In combination with $\lambda_{\rm{top}}<0$,  by \cite[Theorem 5.1 (a)]{Ruelle.1979.IHESPM27} (see also \cite[Lemma 3.1]{FlandoliGessScheutzow.2017.PTRF511}),  we derive that  for every $\gamma\in (0,-\lambda_{\rm{top}})$, there are  measurable maps $\beta_2>\beta_1:\Omega\times \RR^d\to\RR_+$ such that for $\PP\otimes\mu^{\infty}$-a.e. $(\omega, y)\in \Omega\times\RR^d$,
\begin{align}
	\ca{O}(\omega,y)\coloneqq \{z\in B(y, \beta_1(\omega, y)):|\varphi_n(\omega,z)-\varphi_n(\omega,y)|\leq \beta_2(\omega,y){\rm{e}}^{-\gamma n}, \,\,\text{for all}\,\,n\geq 0\}
\end{align}
is an open neighborhood of $y$, $\PP$-a.s. Consider a countable family of balls of the form $B(y_m, r_m)$, where $(y_m, r_m)\in \bb{Q}^d\times \bb{Q}^+$. Then,  for $\PP\otimes\mu^{\infty}$-a.e. $(\omega, y)\in\Omega\times\RR^d$, it holds that $\ca{O}(\omega,y)\subseteq \cup_{m=1}^{\infty} B(y_m, r_m)$ and consequently, 
\begin{align}
	\Big\{\omega:\lim_{n\to\infty}{\text{diam}}\big(\varphi_{n}(\omega,B(y_m, r_m))\big)=0,& \,\,\text{for some}\,\,m\in\bb{N}\Big\}\\
	&\supseteq\Big\{\omega:\lim_{n\to\infty}{\text{diam}}\big(\varphi_{n}(\omega,\ca{O}(\omega,y))\big)=0\Big\}\,.
\end{align}
Thus, there exists $m_0\in\bb{N}$ such that  
\begin{align}
	\PP\Big(\omega:\lim_{n\to\infty}{\text{diam}}\big(\varphi_{n}(\omega,B(y_{m_0}, r_{m_0}))\big)=0\Big)>0\,.
\end{align}
The ball $B(y_{m_0},r_{m_0})$ is the set $U$ on which $\varphi$ is weakly asymptotically stable, which completes the proof.
\end{proof}
\begin{remark} 
	Consider the additive SDE on $\RR^d$,
	\begin{align}
		\d Y_t=-\nabla V(Y_t)\d t+\sigma \d W_t\,,
	\end{align}
	with $\sigma>0$ and satisfying one-sided Lipschitz condition
	\begin{equation}
		\langle \nabla V(x)-\nabla V(y), x-y\rangle \geq -\lambda |x-y|^2\,, \,\,\,\,\,\,\,x, y\in \RR^d
	\end{equation}
	for some $\lambda>0$.
	
	(i) \cite[Example 3.7]{FlandoliGessScheutzow.2017.PTRF511} Assume that  $V\in C^{2, \delta}_{\rm{loc}}(\RR^d)$ for some $\delta>0$ such that for some  $C_0>0, R_0>0, N\geq 0$,
	\begin{align}
		V(x)\geq &C_0 \log |x|, \,\,\,|D^2 V(x)|\leq C_0 |x|^N,  \,\,\,\,\,\,\,\text{for all} \,\,|x|\geq R_0,\\[1mm]
		 &\inf_{y: \,{\rm{minimizer\, of}} \,V}\Big\{\min_{|r|=1} \big\langle D^2 V(y)r, r\big\rangle\Big\}>0\,.
	\end{align}
Then, $\lambda_{\rm{top}}<0$ for sufficiently small $\sigma$.

(ii) \cite[Example 3.8]{FlandoliGessScheutzow.2017.PTRF511} Assume that  $V$ is radially symmetric, i.e.  there exists a convex function $g\in C_{\rm{loc}}^{2, \delta} $ such that $V(x)=g(|x|^2)$. In addition,  if
\begin{equation}
\int_{\RR^d}{\rm{e}}^{-\frac{2}{\sigma^2} V(x)} \d x<\infty\,,
\end{equation}
then, $\lambda_{\rm{top}}<0$ for any $\sigma>0$.
\end{remark}
\begin{remark}
	Consider the multiplicative SDE on $\RR^d$,
	\begin{equation}
		\d Y_t=b(Y_t)\d t+\sqrt{\varepsilon}\sigma (Y_t)\d W_t\,,
	\end{equation}
	and assume that it is strongly mixing with respect to the invariant probability measure $\mu^{\infty}_{\varepsilon}$ in the sense of \cref{def:stronglymixing}.
	Assume that $b\in C^{1, 1}(\RR^d;\RR^d)$  and $ \sigma\in C^{1, 1}(\RR^d;\RR^{d\times m})$. 	If \begin{equation}
		\lim_{\varepsilon\to 0}\int_{\RR^d}\sup_{|r|=1}\left\langle r, D b\left(y\right)r\right\rangle \d \mu_{\varepsilon}^{\infty}(y)<0\,,\label{eq:hypobolic}
	\end{equation}
	then  the top Lyapunov exponent $\lambda_{\rm{top}}^{\varepsilon}$ is negative for all  sufficiently small $\varepsilon$. Indeed, it follows from \eqref{es:Lyaformula} that 
	\begin{equation}
		\lambda_{\rm{top}}^{\varepsilon}\leq \int_{\RR^d} \sup_{|r|=1}\langle r, Db(y) r\rangle \d \mu_{\varepsilon}^{\infty}+\frac{\|\sigma\|_{C^{1, 1}(\RR^d)}^2\varepsilon}{2}\,.
	\end{equation}
	By the assumption above, the right-hand side is strictly negative for all sufficiently small $\varepsilon>0$, which yields the desired conclusion.
\end{remark}

\subsubsection{Two-point motion and reachability}\label{subsubsec:reachability}
In this subsection, we discuss how to establish \cref{assum:reachability} using condition (ii) of \cref{thm:weaksynchronisationSDE}. The proof is based on the two-point motion and combines geometric controllability under the H\"{o}rmander condition with the Stroock--Varadhan support theorem. Before proceeding to the proof, we introduce some useful notation.

\begin{definition}[Two-point Motion]
 We refer to the process $(Y_t^{y_1},Y_t^{y_2})_{t\geq 0}$ associated with 
\begin{align}
	\begin{cases}
		\d Y_t^{y_1}=b(Y_t^{y_1})\,\d t+\sigma(Y_t^{y_1})\,\d W_t\,,\,\,\,\,\,\,\,\,\,\,Y_0^{y_1}=y_1\in\RR^d\,,\\
		\d Y_t^{y_2}=b(Y_t^{y_2})\,\d t+\sigma(Y_t^{y_2})\,\d W_t\,,\,\,\,\,\,\,\,\,\,\,Y_0^{y_2}=y_2\in\RR^d\,,
	\end{cases}
\end{align}
 as the \emph{two-point motion}.
\end{definition}
  This  terminology  originates from the  general theory of stochastic flows, see \cite{Kunita.1997.346,BaxendaleStroock.1988.PTRF169,Carverhill.1985.S273}. Notice that the two processes are driven by the same Brownian motion on $\RR^m$. Equivalently, we may interpret it as a stochastic differential equation on $\RR^{2d}$ taking the form of  
\begin{align}\label{eq:two-pointmotionR2d}
\d\begin{pmatrix}
Y_t^{y_1}\\
Y_t^{y_2}
\end{pmatrix}
=
\begin{pmatrix}
		b(Y_t^{y_1})\\
		b(Y_t^{y_2})
\end{pmatrix}	
\,\d t
+
\begin{pmatrix}
		\sigma(Y_t^{y_1})\\
		\sigma(Y_t^{y_2})
\end{pmatrix}	
\,\d W_t\,,
\,\,\,\,\,\,\,\,\,
\begin{pmatrix}
		Y_0^{y_1}\\
		Y_0^{y_2}
\end{pmatrix}=\begin{pmatrix}
		y_1\\
		y_2
\end{pmatrix}\,.
\end{align}
Since we are working with the two-point motion, we use $\mathbf{y}=(y_1, y_2)^T, y_1, y_2\in\RR^d$ to denote a generic element in $\RR^{2d}$. Then the system \eqref{eq:two-pointmotionR2d} reduces to 

\vspace{-4mm}
\begin{equation}\label{eq:compact2ddimensionSDE}
	\,\mathrm{d} {\mathbf{Y}}_t^{\,\mathbf{y}} \, 
=
{\bm{b}}({\bf{Y}}_t^{\,\bf{y}})
\,\mathrm{d}t+
{\bm{\sigma}}({\bf{Y}}_t^{\,\mathbf{y}})
\,\mathrm{d}W_t\,,
\end{equation}
where \begin{align}\label{eq:BSigma}
{\bm{b}}:\RR^{2d}\ni\begin{pmatrix}
		y_1\\
		y_2
\end{pmatrix}\mapsto\begin{pmatrix}
	b(y_1)\\
	b(y_2)
\end{pmatrix}\in \RR^{2d}\,,\,\,\,\,\,\,\,\, {\bm{\sigma}}:\RR^{2d}\ni\begin{pmatrix}
		y_1\\
		y_2
\end{pmatrix}=\begin{pmatrix}
	\sigma(y_1)\\
	\sigma(y_2)
\end{pmatrix}\in \RR^{2d\times m}.
\end{align}

In order to apply geometric control theory  and the Stroock--Varadhan support theorem, it is convenient to reformulate \eqref{eq:two-pointmotionR2d} in Stratonovich form. For any ${\mathbf{y}}=({\mathbf{y}}^{(1)},\ldots, {\mathbf{y}}^{(i)},\ldots,{\mathbf{y}}^{(2d)})^T\in \RR^{2d}$, we define a family of vector fields on $\RR^{2d}$ as 

\vspace{-4mm}
{\small{
\begin{align}\label{eq:Sigma0}
	{U}_0^{(i)}({\mathbf{y}})\coloneqq
	\begin{cases}
		b^{(i)}({\mathbf{y}}^{(1)},\ldots, {\mathbf{y}}^{(d)})-\frac{1}{2}
				\displaystyle\sum_{k=1}^d\sum_{\ell=1}^m \big((\partial_{k}\sigma^{(i\ell)})\sigma^{(k\ell)}\big)({\mathbf{y}}^{(1)},\ldots, {\mathbf{y}}^{(d)}),\,\,\,\, \,\,\,\, \,\,\,\, \,\,\,\, \,\,\,\, \,\,\,\, \,\,\,\, \,\,\,\, \,\,\,\, 1 \leq i \leq d\,,\\[1mm]
		b^{(i-d)}({\mathbf{y}}^{(d+1)},\ldots, {\mathbf{y}}^{(2d)})-\frac{1}{2}\displaystyle\sum_{k=1}^d\sum_{\ell=1}^m \big((\partial_{k}\sigma^{((i-d)\ell)})\sigma^{(k\ell)}\big)({\mathbf{y}}^{(d+1)},\ldots, {\mathbf{y}}^{(2d)}), \,\,\,\,d+1 \leq i \leq 2d\,,
	\end{cases}
	% \begin{pmatrix}
	% 	b^{(1)}({\mathsf{y}}^{(1)},\ldots, {\mathsf{y}}^{(d)})\\
	% 	\vdots\\
	% 	b^{(d)}({\mathsf{y}}^{(1)},\ldots, {\mathsf{y}}^{(d)})\\
	% 	b^{(1)}({\mathsf{y}}^{(d+1)},\ldots, {\mathsf{y}}^{(2d)})\\
	% 	\vdots\\
	% 	b^{(d)}({\mathsf{y}}^{(d+1)},\ldots, {\mathsf{y}}^{(2d)})\\
	% \end{pmatrix}-\frac{1}{2}\begin{pmatrix}
	% 	\sum_{k=1}^d\sum_{\ell=1}^m \partial_{k}\sigma^{(1\ell)}({\mathsf{y}}^{(1)},\ldots, {\mathsf{y}}^{(d)})\sigma^{(k\ell)}({\mathsf{y}}^{(1)},\ldots, {\mathsf{y}}^{(d)})\\
	% 	\vdots\\
	% 		\sum_{k=1}^d\sum_{\ell=1}^m \partial_{k}\sigma^{(d\ell)}({\mathsf{y}}^{(1)},\ldots, {\mathsf{y}}^{(d)})\sigma^{(k\ell)}({\mathsf{y}}^{(1)},\ldots, {\mathsf{y}}^{(d)})\\
	% 		\sum_{k=1}^d\sum_{\ell=1}^m \partial_{k}\sigma^{(1\ell)}({\mathsf{y}}^{(d+1)},\ldots, {\mathsf{y}}^{(2d)})\sigma^{(k\ell)}({\mathsf{y}}^{(d+1)},\ldots, {\mathsf{y}}^{(2d)})\\
	% 	\vdots\\
	% 		\sum_{k=1}^d\sum_{\ell=1}^m \partial_{k}\sigma^{(d\ell)}({\mathsf{y}}^{(d+1)},\ldots, {\mathsf{y}}^{(2d)})\sigma^{(k\ell)}({\mathsf{y}}^{(d+1)},\ldots, {\mathsf{y}}^{(2d)})
	% \end{pmatrix}
\end{align}
}}
and for each $j=1,\ldots,m$,

\vspace{-4mm}
{\small{
		\begin{align}\label{eq:Sigma1d}
			{U}_j^{(i)}({\mathbf{y}})\coloneqq
			\begin{cases}
				\sigma^{(ij)}({\mathbf{y}}^{(1)},\ldots, {\mathbf{y}}^{(d)}),\,\,\,\, \,\,\,\, \,\,\,\, \,\,\,\, 1 \leq i \leq d\,,\\[1mm]
				\sigma^{((i-d)j)}({\mathbf{y}}^{(d+1)},\ldots, {\mathbf{y}}^{(2d)}), \,\,\,\,d+1 \leq i \leq 2d\,,
			\end{cases}
		\end{align}
%\begin{align}
%	{U}_j({\mathsf{y}})\coloneqq\begin{pmatrix}
%		\sigma^{(1j)}({\mathbf{y}}^{(1)},\ldots, {\mathbf{y}}^{(d)})\\
%		\vdots\\
%		\sigma^{(dj)}({\mathbf{y}}^{(1)},\ldots, {\mathbf{y}}^{(d)})\\
%		\sigma^{(1j)}({\mathbf{y}}^{(d+1)},\ldots, {\mathbf{y}}^{(2d)})\\
%		\vdots\\
%		\sigma^{(dj)}({\mathbf{y}}^{(d+1)},\ldots, {\mathbf{y}}^{(2d)})\\
%	\end{pmatrix}\,,
%\end{align}
}}
i.e. $U_j,j=1,\ldots,m,$ is the $j$-th column of the matrix $\bm{\sigma}$ of \eqref{eq:BSigma}. 

\vspace{4mm}
We now proceed to state the geometric controllability result which follows from a variant of the Chow--Rashevsky theorem.
\begin{definition}[H\"{o}rmander's condition]
	Let $\ca{M}$ be an $n$-dimensional connected differentiable manifold. Given a family of vector fields  on $\RR^{n}$,
	$\ca{U}=\{U_1,\ldots, U_m\}\,.$ We say that $\ca{U}$ satisfies \emph{H\"{o}rmander's condition} on $\ca{M}$ if for all $x\in\ca{M}$,
$${\mathrm{Lie}} _x\{f_1, \ldots, f_m\}=T_x\ca{M},$$
where $T_x\ca{M}$ denotes the tangent space of $\ca{M}$ at $x$, ${\mathrm{Lie}} _x\{f_1, \ldots, f_m\}$ is the smallest Lie algebra containing $f_1(x), \ldots, f_m(x)$, i.e. the linear span of all iterated Lie brackets generated by these vector fields.   Recall that for  vector fields $U,V$ on $\ca{M}$,  the Lie bracket $[U,V]$ is defined as 
$[U,V]=UV-VU.$
	%$[U,V](x)=\nabla V(x) U(x)-\nabla U(x) V(x)\,. $
\end{definition}

\begin{lemma}[Geometric controllability]\label{lem:geometricalcontrollability}	
 Given smooth vector fields $f_0, \ldots, f_m$ on $\ca{M}$,   consider a control affine system 
	\begin{equation}\label{eq:affinecontrolsystem}
		\dot{x}=f_0(x)+\sum_{i=1}^m \ell_i f_i(x)\,,
	\end{equation}
where $\ell = (\ell_1, \ell_2, \ldots, \ell_m)\in \mathbb{R}^m$ is a control. We denote by $x(t; x_0, \ell)$ the corresponding solution at time $t$ with initial condition $x_0$ driven by the control $\ell$. This control affine system defines a family of vector fields 
\begin{equation}
	\scr{F}=\Big\{f_0+\sum_{i=1}^m \ell_i f_i: \ell=\big(\ell_1, \ell_2, \ldots \ell_m\big)\in\RR^m\Big\}\,,
\end{equation}
and   the reachable set  from $x_0$ at time $t>0$ is given by  	 \footnote{Here, for any $t\geq 0$ and vector field $f$, ${\rm{e}}^{t f}$ is the flow generated by the vector field defined by ${\rm{e}}^{t f}(x_0)\coloneqq x(t; x_0, 1)$. }
\begin{align}\label{eq:trajectories}
	\ca{R}(t, x_0)\coloneqq \Big\{{\rm{e}}^{t_n \tilde{f}_{i_n}}\circ \cdots \circ {\rm{e}}^{t_1 \tilde{f}_{i_1}}x_0, \,\,\,\,\,\tilde{f}_{i_1},\ldots, \tilde{f}_{i_n}\in\scr{F}, t_1,  \ldots t_n\geq 0, {\text{such that}}~\sum_{k=1}^{n} t_k=t\,\Big\},
\end{align}
see \cite[Chapter 1, P.28, Definition 6]{bookGeometriccontroltheory}.
\begin{thmenumerate}[label=\upshape(\roman*)]
	\item \cite[Chapter 4, p.~106, Theorem 2]{bookGeometriccontroltheory} Assume that for all $x$ in $\ca{M}$, the vector fields $\{ f_1, \ldots, f_m\}$ satisfy
	    H\"{o}rmander's condition.
	  % $${\mathrm{Lie}} _x\{f_1, \ldots, f_m\}=T_x\ca{M},$$
	  %where $T_x\ca{M}$ denotes the tangent space of $\ca{M}$ at $x$.  \footnote{${\mathrm{Lie}} _x\{f_1, \ldots, f_m\}$ is the smallest Lie algebra containing $f_1(x), \ldots, f_m(x)$, i.e. the linear span of all iterated Lie brackets generated by these vector fields.   For any vector fields $U,V$ on $\RR^{d}$,  the Lie bracket $[U,V]$ is defined as 
	  	%	$[U,V](x)=\nabla V(x) U(x)-\nabla U(x) V(x)\,. $}
Then the control affine system \eqref{eq:affinecontrolsystem} 
satisfies 
\begin{equation}
\ca{R}(T, x_0)=\ca{M}\,,
\end{equation}
 for all $x_0\in \ca{M}$ and $T>0$.
%is strongly controllable in the sense that for any $T>0$,  any point of $\ca{M}$ is reachable from any other point by $\scr{F}$ in $T$.
	\item \cite[Chapter 4, p.~114, Theorem 5]{bookGeometriccontroltheory} Assume that the drift $f_0$ is recurrent \footnote{We call a vector field $f$ recurrent if it is complete in the sense that the integral curves through each point $x\in \ca{M}$ are defined for all $t\in \RR$ and  it has a dense subset of recurrent points. Recall a point $x\in\ca{M} $ is called recurrent if $x$ is contained in the positive limit set $\omega(x)\coloneqq \big\{y=\lim_{n\to\infty} {\text{e}}^{t_n f}(x): \lim_{n\to\infty}t_n=\infty\big\}$. },  and for all $x$ in $\ca{M}$, the vector fields $\{f_0, f_1, \ldots, f_m\}$ satisfy 
 H\"{o}rmander's condition.
	 Then  \eqref{eq:affinecontrolsystem}  
	 satisfies that for all $x_0, x_1\in \ca{M}$ , there exists $T>0$ such that 
	 \begin{equation}
	 x_1\in \ca{R}(T, x_0).
	 \end{equation}
	\end{thmenumerate}

\end{lemma}

\vspace{3mm}
We now introduce the parabolic H\"{o}rmander condition which is widely used in the literature to ensure the existence of a smooth density of the solution (see~\cite[Theorem 1.3]{Hairer.2011}). We will later show that it can be used to verify the controllability condition (ii) of ~\cref{lem:geometricalcontrollability}.
\begin{definition}[Parabolic H\"{o}rmander's condition]
%	\item (Lie bracket) Given a manifold $\ca{M}$ and $x\in \ca{M}$, for any vector fields $U,V\in T_x\ca{M}$, where $T_x\ca{M}$ is the tangent space of $\ca{M}$ at $x$,  we define the Lie bracket $[U,V]:C^{\infty}(\ca{M})\to C^{\infty}(\ca{M})$ as
	%\begin{equation}
	%	[U,V]f=UVf-VUf\,.
	%\end{equation}
%In particular, 
Given a family of vector fields  on $\ca{M}$,
$\ca{U}=\{U_0, U_1,\ldots, U_m\}\,.$
We  inductively define a collection of vector fields
\begin{align} \label{eq:Liespace}
	{\scr{U}}_0\coloneqq\,\{U_j:\,\,j=1,\ldots,m\}\,,\,\,\,\,\,{\scr{U}}_{k+1}\coloneqq\,{\scr{U}}_{k}\cup \big\{[U,U_j], \,U\in {\scr{U}}_{k}, \,\,j=0,\ldots,m\big\}\,,
\end{align}
and vector spaces 
$
	{\scr{U}}_k(x)={\rm{span}}\{U(x): U\in \scr{U}_k\}\,.
$
We say that $\ca{U}$ satisfies the \emph{parabolic H\"{o}rmander condition}  if for all $x\in\ca{M}$,
\begin{equation}\label{eq:Hormander}
		\bigcup_{k\geq 1} {\scr{U}}_{k}(x)=T_x\ca{M}\,.
	\end{equation}
%and $\ca{U}$ satisfies H\"{o}rmander's condition if for all $y\in\RR^d$,
%\begin{equation}
%		\widetilde{{\scr{U}}}_0(y) \cup \Big(\cup_{k\geq 1} {\scr{U}}_{k}(y)\Big)=\RR^{d}\,,
%	\end{equation}
%where $\widetilde{{\scr{U}}}_0=\,\,\{U_j:\,\,0, 1,\ldots,n\}\,.$
	\worknote{It only proves the existence of a smooth density of the the solutions with respect to Lebesgue measure and has no information of the positivity or Harnack inequalities of the density under H\"{o}rmander's condition. We instead look for the help from support theorem. The support theorem requires the existence of a deterministic path, which is guaranteed by the Chow--Rashevsky theorem.}

\end{definition}

\vspace{3mm}

%Propagation path: \begin{definition}[Propagation path] Given a family of vector fields $$\ca{U}=\{U_0(y), U_1(y),\ldots, U_n(y)\}$$ on $\RR^{d},$ a propagation path associated with $\ca{U}$ (or $\ca{U}$-trajectory) is an absolutely continuous path $\gamma:[0,T]\to\RR^d$ such that
%\begin{align}
%\begin{cases}
%	\dot{\gamma}(t)=\sum_{j=0}^n\lambda_j(t)U_j(\gamma(t))\,,\,\,\,\,\,\,[0,T]\,,\\
%	\gamma(0)=y_0\,,
%\end{cases}
%\end{align}
%for piecewise constant real functions $\lambda_0,\lambda_1,\ldots, \lambda_n$. Based on propagation paths, we define the set of all points $y$ reachable from $y_0$:
%\begin{align}
%	\text{PathSet}(y_0,\RR^d)\coloneqq \{y\in\RR^d|\,\text{there exists propagation path}\,\gamma, %\gamma(0)=y_0, \gamma(T)=y \}\,.
%\end{align}
%\end{definition}

%\begin{lemma}[Chow--Rashevskii Theorem]\label{lem:Chow--Rashasky's theorem} Let $\ca{M} $ be a connected $d$-dimensional manifold. Given a family of vector fields $\ca{U}=\{U_0, U_1,\ldots, U_n\}$ on $\ca{M}$. If at each $y\in\ca{M}$, all these vector fields together with their iterated Lie brackets span the tangent space $T_y\ca{M}$ , then any two points of $\ca{M}$ can be connected by a piecewise smooth $\ca{U}$-path. Precisely speaking, for any $y_0, y_1 \in \ca{M}$, there is piecewise constant real functions $\lambda_0,\lambda_1,\ldots, \lambda_n$ such that
%\begin{align}
%\begin{cases}
%	\dot{\gamma}(t)=\sum_{j=0}^n\lambda_j(t)U_j(\gamma(t))\,,\,\,\,\,\,\,[0,T]\,,\\
%	\gamma(0)=y_0\,,\,\,\gamma(T)=y_1\,.
%\end{cases}
%\end{align}
%\end{lemma}

\vspace{3mm}
We now show how condition (ii) of \cref{thm:weaksynchronisationSDE} allows us to verify the reachability  \cref{assum:reachability}.

\begin{theorem}[Reachability]\label{thm:reachability}
	Consider the RDS $\varphi$ corresponding to multiplicative SDE \eqref{eq:multiSDE00}.
We further assume  that $b,  \sigma$ are smooth with  bounded derivatives up to order $2$ and

\begin{thmenumerate}[label=\upshape(\roman*)]
\item The vector fields $U_1, U_2,\ldots,U_m,$ corresponding to two-point motion defined in \eqref{eq:Sigma1d}, satisfy H\"{o}rmander's condition for ${\bf{y}}\in \RR^{2d}\setminus \Delta$,

\textbf {or}

\item  $U_0$ is recurrent and  $U_0, U_1, U_2,\ldots,U_m,$ corresponding to two-point motion defined in \eqref{eq:Sigma0} and \eqref{eq:Sigma1d}, satisfy  the parabolic H\"{o}rmander condition on $\RR^{2d}\setminus \Delta$.
\end{thmenumerate}
Then, \cref{assum:reachability} is satisfied.

\end{theorem}
\begin{remark}

It is worthwhile to emphasise that $\bm{\sigma}$ is degenerate on the diagonal $\Delta\coloneqq \{(y, y)\in \RR^d\times \RR^d\}$ as for any  $\mathbf{y}\in \Delta$, the determinant
\begin{equation}
	\mathrm{det}\,(\bm{\sigma}\bm{\sigma}^T)(\mathbf{y})=0\,.
\end{equation}
Moreover, it is impossible for  $\ca{U}=\{U_0, U_1, \ldots, U_m\}$ to satisfy H\"{o}rmander condition  on  $\Delta$ since the largest dimension vector space that the Lie algebra of $\ca{U}$ can generate is $d$ (as opposed to the required $2d$).
\end{remark}

\begin{proof}[{Proof of \cref{thm:reachability}}]
 	We prove the theorem  by combining geometric controllability (cf.  \cref{lem:geometricalcontrollability})	 with probabilistic controllability, namely, the Stroock--Varadhan support theorem for the two-point motion. Since $b$ and $\sigma$ have bounded derivatives up to order 2, it follows from  the Stroock--Varadhan support theorem (see \cite[Theorem 3.5]{MilletSanz1994}) that, for any $\alpha\in [0, 1/2)$, the support of  the law
 	\worknote{Here, the support theorem requires the boundedness of the derivatives up to order 2, but still we can prove the reachability; if we just apply Hormander condition, which only implies the existence of the heat kernel. If we want the positivity of the heat kernel, we need to impose boundedness of all the derivatives for all orders. Then, we have the infinite differentiability of the density, which implies that the heat kernel would be positive anywhere, otherwise, it's not continuous anymore. Here, even though we use Hormander condition, but we don't require the boundedness of the derivatives greater than 3} $\PP\circ ({\bf{Y}}_{\cdot}^{\,\mathbf{y}})^{-1}$ of the two-point motion (see \eqref{eq:two-pointmotionR2d}) 
 	\begin{align}
 		\,\mathrm{d} {\bf{Y}}_t^{\,\mathbf{y}}=
 		\bm{b}({\bf{Y}}_t^{\,\mathbf{y}})\,\mathrm{d}t+
 		\bm{\sigma}({\bf{Y}}_t^{\,\mathbf{y}})\,\mathrm{d}W_t\,,\,\,\,\,\,\,{\bf{Y}}_0^{\,\mathbf{y}}=\mathbf{y}\,,
 	\end{align}	
 	is characterized by 
 	$$\overline{\{\ca{S}(h): h\in W_0^{1,2}([0,T]; \RR^{m})\}}^{ C^{\alpha}([0,T];\RR^{2d})},$$ where the closure is taken in $C^{\alpha}([0,T];\RR^{2d})$, and 
 $W_0^{1,2}([0, T]; \RR^m)$  denotes the space of absolutely continuous functions  $h: [0, T]\to\RR^m$ such that  $h_0=0$, and the  derivative $\dot{h}\in L^2([0,T];\RR^{m})$. Here,  ${\ca{S}} (h)$ is the solution to the controlled equation 
 	\begin{align}\label{eq:support}
 		\ca{S}(h)_t={\mathbf{y}}+\int_0^t \Big[\bm{b}(\ca{S}(h)_s)-\frac{1}{2}\big((\nabla{\bm{\sigma}}){\bm{\sigma}}\big)(\ca{S}(h)_s)\Big]\d s+\int_0^t {\bm{\sigma}}(\ca{S}(h)_s)\dot{h}_s\,\d s\,.
 	\end{align}
 	
 	Consider the following family of vector fields
 	\begin{align}
 			\scr{F}\coloneqq\Big\{U_0+\sum_{i=1}^m \ell_i U_i: \ell=\big(\ell_1, \ell_2, \ldots \ell_m\big)\in \RR^m\Big\}\,, \label{newF}
 	\end{align}
 	where $ U_0, U_1,\ldots,U_m,$ are  defined in \eqref{eq:Sigma0},  \eqref{eq:Sigma1d}. 
 	
 	(i) Let $d\geq 2$. Fix
 	$\bar{\mathbf{y}}=(\bar{y}_1, \bar{y}_2)^T\in \RR^{2d}\setminus\Delta,$ and choose $\hat{\mathbf{y}}=(\hat{y}_1, \hat{y}_2)^T\in \Delta_{a, \varepsilon/2}\setminus\Delta\subset\RR^{2d}\setminus\Delta \,.$ 
 	 Since the vector fields $ U_1, U_2,\ldots,U_m,$  satisfy  $	{\text{Lie}}_\mathbf{y} \{ U_1, \ldots, U_m\}=\RR^{2d}$ for  ${\bf{y}}\in \RR^{2d}\setminus \Delta$ and $ \RR^{2d}\setminus \Delta$ is connected, 
 	geometric controllability  yields that
 	the control-affine system
 	\begin{equation}
 		\dot{\mathbf{y}}=U_0(\mathbf{y})+\sum_{i=1}^m \ell_i U_i(\mathbf{y})
 	\end{equation}
 	satisfies
 	\begin{equation}
 		\ca{R}(T,\bar{\mathbf{y}})=
 	\RR^{2d}\setminus \Delta\,.
 	\end{equation}
 In particular, for any $T>0$, there exists  $\tilde{U}_1, \ldots, \tilde{U}_n\in \scr{F}$ (see \eqref{newF}) and $t_1, \ldots, t_n>0$ with $\sum_{k=1}^n t_k=T$ such that 
 	\begin{equation}
 		\hat{\mathbf{y}}={\rm{e}}^{t_n \tilde{U}_n}\circ \dots \circ {\rm{e}}^{t_1 \tilde{U}_1}	\bar{\mathbf{y}}\,.\label{eq:control} 
 	\end{equation}
 	 We next formulate   \eqref{eq:control} in terms of  a controlled ODE. Write   
 	 $$\tilde{U}_k=U_0+\sum_{i=1}^m \ell_{i, k} U_i,\,\,\,\,\,k=1, \ldots, n.$$ Then the path   $\mathbf{y}(\cdot)$ connecting $\bar{\mathbf{y}}$ to $\hat{\mathbf{y}}$ satisfies 
 	\begin{align}
 		\hat{\mathbf{y}}=\,\,&\bar{\mathbf{y}}+\int_0^{t_1} \Big[ U_0({\mathbf{y}}(s))+\sum_{i=1}^m\ell_{i, 1} U_i({\mathbf{y}}(s))\Big] \d s+\int_{t_1}^{t_1+t_2} \Big[U_0({\mathbf{y}}(s))+\sum_{i=1}^m\ell_{i, 2} U_i({\mathbf{y}}(s)) \Big]\d s\\
 		&+\dots+\int_{\sum_{i=1}^{n-1} t_i}^{\sum_{i=1}^{n} t_i} \Big[ U_0({\mathbf{y}}(s))+\sum_{i=1}^m\ell_{i, n} U_i({\mathbf{y}}(s))\Big] \d s\\
 		\eqqcolon\,\,& \bar{\mathbf{y}}+\int_0^T U_0({\mathbf{y}}(s)) \d s+\int_0^T \ell_1(s)U_1({\mathbf{y}}(s)) \d s+\dots+\int_0^T \ell_m(s)U_m({\mathbf{y}}(s)) \d s\,,\label{eq:support2}
 	\end{align}
 	where $\ell_j(\cdot), j=1, \ldots, m,$ is the piecewise constant function defined by 
 	\begin{equation}
 		\ell_j(t)={\bf{1}}_{[0, t_1]}(t)\ell_{j, 1} +{\bf{1}}_{[t_1, t_1+t_2]}(t)\ell_{j, 2}+\cdots+ {\bf{1}}_{[\sum_{i=1}^{n-1} t_i, \sum_{i=1}^{n} t_i]}(t) \ell_{j, n}\,.
 	\end{equation}
 	Define a function $h_{\cdot}\in W_0^{1, 2}([0, T]; \RR^{m})$ by $h_0=0$ and  $\dot{h}_\cdot=(\ell_1(\cdot),\ldots, \ell_m(\cdot))^T$.  It follows from \eqref{eq:support}  that the controlled path to \eqref{eq:support2} satisfies
 	$$({\mathbf{y}}(t))_{t\in [0, T]}\in \rm{supp}\left(\bb{P}\circ \big({\bf{Y}}_{t}^{\bar{\mathbf{y}}}\big)_{t\in [0, T]}^{-1}\right).$$ 	
 	Since $\mathbf{y}(T)=\hat{\mathbf{y}}\in
 	\Delta_{a,\varepsilon/2}\setminus\Delta$, consider the open neighbourhood
 	of $\mathbf{y}(\cdot)$ given by
 	\begin{align}
 		O\coloneqq\left\{
 		\mathbf{z}(\cdot)\in C^\alpha([0,T];\RR^{2d}):\sup_{t\in[0,T]}
 		|\mathbf{z}(t)-\mathbf{y}(t)|<\frac{\varepsilon}{2}
 		\right\}.
 	\end{align}
 	For every $\mathbf{z}(\cdot)\in O$, we have
 	\begin{align}
 		|\mathbf{z}(T)-(a,a)|\leq|\mathbf{z}(T)-\mathbf{y}(T)|+
 		|\hat{\mathbf{y}}-(a,a)|<\varepsilon,
 	\end{align}
 	and hence $\mathbf{z}(T)\in\Delta_{a,\varepsilon}$.
 	Therefore,
 	\begin{align}
 		\PP\big(
 		\tau_{(\bar y_1,\bar y_2),\varepsilon}^{a}<\infty
 		\big)\geq\,\,&\PP\left(({\bf Y}_{t}^{\,\bar{\mathbf y}})_{t\in[0,T]}\in O\right)
 		=\,\,\big(\PP\circ({\bf Y}_{\cdot}^{\,\bar{\mathbf y}})^{-1}
 		\big)(O)>0,
 	\end{align}
 	where the last inequality follows from
 	$\mathbf{y}(\cdot)\in
 	\operatorname{supp}\big(
 	\PP\circ({\bf Y}_{\cdot}^{\,\bar{\mathbf y}})^{-1}
 	\big)$
 	and the definition of the support of a probability measure.
 	
 	 When $d=1$, the set $\RR^2\setminus\Delta$ consists of two connected
 	components,
 	\begin{align}
 		\{(x,y)\in\RR^2:x<y\}
 		\qquad\text{and}\qquad
 		\{(x,y)\in\RR^2:x>y\}\,.
 	\end{align}
The geometric controllability argument is applied separately to
 	each component. Hence, 
 	\begin{align}
 		\ca R(T,\bar{\mathbf y})=
 		\begin{cases}
 			{(x,y)\in\RR^2:x<y},
 			&\text{if }\bar y_1<\bar y_2,\\ 
 			{(x,y)\in\RR^2:x>y},
 			&\text{if }\bar y_1>\bar y_2.
 		\end{cases}
 	\end{align}
 	For either case, one can choose
 	$\hat{\mathbf y}\in\Delta_{a,\varepsilon/2}\setminus\Delta$
 	in the same connected component as $\bar{\mathbf y}$, and the preceding support-theorem argument applies without change. This again yields
 	\begin{align}
 		\PP\big(
 		\tau_{(\bar y_1,\bar y_2),\varepsilon}^{a}<\infty
 		\big)>0.
 	\end{align}
 	
(ii) The parabolic H\"{o}rmander condition on $\RR^{2d}\setminus \Delta$ implies that 	\begin{align}
	{\text{Lie}} _{\mathbf{y}}\{U_0, U_1, \ldots, U_m\}=\RR^{2d}\,,
\end{align}
since the  vector fields involved in the parabolic H\"{o}rmander condition already generate the entire tangent space.
Using  \cref{lem:geometricalcontrollability} (ii), there exist a time $T_0$,  and 
vector fields $Z_1, \ldots, Z_n\in \scr{F}$ such that 
\begin{align}
	\tilde{\mathbf{y}}={\rm{e}}^{t_n Z_n}\circ \cdots\circ {\rm{e}}^{t_1 Z_1}\bar{\bf{y}},
\end{align}
where $t_1, \ldots, t_n\geq 0$ and $\sum_{i=1}^n t_i=T_0$.  From this point onwards, the proof proceeds exactly as in case (i) following \eqref{eq:control}.

\end{proof}

\subsection{Optimised multiplicative noise and synchronising synchronous realisations}
\label{sec:optimized-multiplicative-synchronisation}

Let $\mathbb{T}^d=\mathbb{R}^d/\mathbb{Z}^d$ be the $d$-dimensional
flat torus, and let
\begin{equation}
   \d  \mu( x)\coloneqq Z^{-1}{\rm{e}}^{-V(x)}\,\d x,
    \qquad
    Z\coloneqq\int_{\mathbb{T}^d}{\rm{e}}^{-V(x)}\,\d x,
    \label{eq:gibbs-measure-general}
\end{equation}
where $V\in C^\infty(\mathbb{T}^d)$. A standard approach to sampling
$\mu$ is to simulate an overdamped Langevin diffusion and approximate
expectations with respect to $\mu$ by long-time averages. The usual
overdamped Langevin equation corresponds to a constant diffusion
matrix. More generally, for a smooth map
\begin{equation}
     H:\mathbb{T}^d\longrightarrow \mathbb{S}^{++}_d,
\end{equation}
one may consider
\begin{equation}
    \d X_t
    =
    \bigl[-H(X_t)\nabla V(X_t)+\operatorname{div}H(X_t)\bigr]\,\d t
    +
    \sigma(X_t)\,\d W_t,
    \qquad
    \sigma(x)\sigma(x)^{\mathsf T}=2H(x),
    \label{eq:general-overdamped-langevin}
\end{equation}
where $W$ is an $m$-dimensional Brownian motion, with $m$ not
necessarily equal to $d$, and
\begin{equation}
    (\operatorname{div}H)_i=
    \sum_{j=1}^d \partial_j H_{ij}.
\end{equation}
The generator associated with \eqref{eq:general-overdamped-langevin}
is
\begin{equation}
    L_H f=
    \bigl[-H\nabla V+\operatorname{div}H\bigr]\cdot\nabla f
    +
    H:\nabla^2 f,
    \label{eq:general-generator}
\end{equation}
and the corresponding Fokker--Planck equation is
\begin{equation}
    \partial_t \rho_t
    =
    \nabla\cdot
    \left[
        H\bigl(\nabla\rho_t+\rho_t\nabla V\bigr)
    \right].
    \label{eq:general-fokker-planck}
\end{equation}
In particular, the Gibbs measure \eqref{eq:gibbs-measure-general} is
invariant, independently of the particular factorization
$\sigma\sigma^{\mathsf T}=2H$.

The work \cite{LelievrePavliotisRobinSantetStoltz2025} emphasizes that,
for a fixed target measure, there are infinitely many reversible
overdamped Langevin dynamics of the form
\eqref{eq:general-overdamped-langevin}, and proposes choosing $H$ so
as to maximize the spectral gap of $L_H$. The underlying heuristic is
to increase the diffusivity in low-probability barrier regions, where
rapid motion facilitates transitions between modes, and to decrease it
in regions where the target measure is concentrated. Under an
appropriate normalization of $H$, nonconstant diffusion coefficients
can substantially increase the spectral gap and hence the
$L^2(\mu)$ convergence rate.

Once $H$, and therefore the one-point generator
\eqref{eq:general-generator} and Fokker--Planck equation
\eqref{eq:general-fokker-planck}, have been fixed, there nevertheless
remains a second degree of freedom. Namely, the covariance
$2H=\sigma\sigma^{\mathsf T}$ does not determine the noise frame
$\sigma$. Distinct, possibly rectangular, square roots of the same
diffusion matrix define the same one-point Markov semigroup but
generally generate different stochastic flows. Consequently, they give
rise to different synchronous two-point motions.

This distinction is relevant for numerical sampling with
Richardson--Romberg extrapolation. The duplicated-diffusion analysis of
Lemaire, Pag\`es and Panloup \cite{LPP15} shows that,
when two discretisations are driven by the same noise, the asymptotic
variance of the extrapolated estimator depends on the long-time joint
law of the corresponding two-point process. In particular, if this
joint law converges to the diagonal coupling of the invariant measure,
then the leading discretization bias can be removed without increasing
the asymptotic variance. Strong mixing of the one-point motion alone
does not guarantee such a property, since the two copies may retain
nontrivial long-time correlations.

The weak-synchronisation framework developed in the present work provides a mechanism for proving the required
two-point mixing for multiplicative SDEs. Roughly speaking, a negative
top Lyapunov exponent yields local asymptotic stability, while a
H\"ormander condition for the two-point motion away from the diagonal
provides the required reachability. This raises the following natural
question: given a diffusion matrix $H$ chosen for favorable one-point
sampling properties, can one choose a noise frame
$\sigma\sigma^{\mathsf T}=2H$ whose synchronous stochastic flow is
synchronizing?

We illustrate this question on an explicit double-well example. The
same multiplicative-noise sampler admits a natural principal
realization whose synchronous two-point motion does not synchronize,
as well as an overcomplete realization whose two-point motion satisfies
H\"ormander's condition away from the diagonal and whose top Lyapunov
exponent is strictly negative.

	\subsubsection{Application: A double-well target and the homogenized diffusion coefficient}\label{subsec:double-well-homogenized}

	In this subsection, we expand on the model problem developed in \cite{LelievrePavliotisRobinSantetStoltz2025}, Section 4.1, Figure 3 (p. 16–17), and extend it by the aspect of choosing a synchronous, synchronizing coupling. We start by recalling the construction and the results of \cite{LelievrePavliotisRobinSantetStoltz2025} regarding the spectral gap of the one-point motion. % of Optimizing the diffusion coefficient of overdamped Langevin dynamics for the example \(V(q)=\cos(4\pi q)\), with \(D_{\mathrm{hom}}=e^V\) and the corresponding spectral-gap comparison.
	Consider the double-well potential
	\begin{equation}
		V(q)=\cos(4\pi q),
		\qquad q\in\mathbb{T},
		\label{eq:double-well-potential}
	\end{equation}
	with corresponding Gibbs measure
	\begin{equation}
		\d\mu(q)=Z^{-1}{\rm{e}}^{-V(q)}\,\d q,
		\qquad Z=\int_{\mathbb{T}}{\rm{e}}^{-V(q)}\,\d q.
		\label{eq:double-well-measure}
	\end{equation}	
	For a scalar diffusion coefficient
	$H:\mathbb{T}\to(0,\infty)$, the reversible Langevin equation reads
	\begin{equation}
		\d X_t=
		\bigl[-H(X_t)V'(X_t)+H'(X_t)\bigr]\,\d t
		+\sqrt{2H(X_t)}\,\d W_t,
		\label{eq:scalar-variable-diffusion}
	\end{equation}
	%with generator
	%\begin{equation}
	%    L_D f
	%    =
	%    Df''+(D'-DV')f'
	%    =
	%    e^V\frac{d}{dq}
	%    \left(
	%        De^{-V}f'
	%    \right).
	%    \label{eq:scalar-generator}
	%\end{equation}
	%Its Dirichlet form is
	%\begin{equation}
	%    \mathcal{E}_D(f,f)
	%    =
	%    -\langle f,L_Df\rangle_{L^2(\mu)}
	%    =
	%    \int_{\mathbb{T}}
	%        D(q)|f'(q)|^2\,\mu(dq),
	%    \label{eq:scalar-dirichlet-form}
	%\end{equation}
	with spectral gap $\Lambda(H)$.
	%\begin{equation}
	%    \Lambda(D)
	%    =
	%    \inf_{\substack{
			%        f\in H^1(\mu),\ \mu(f)=0\\
			%        f\not\equiv0
			%    }}
	%    \frac{
		%        \displaystyle
		%        \int_{\mathbb{T}}D(q)|f'(q)|^2\,\mu(dq)
		%    }{
		%        \displaystyle
		%        \int_{\mathbb{T}}f(q)^2\,\mu(dq)
		%    }.
	%    \label{eq:scalar-spectral-gap}
	%\end{equation}
	%
	%To prevent a trivial acceleration obtained by multiplying $D$ by a
	%constant, we impose the normalization
	%\begin{equation}
	%    \int_{\mathbb{T}}
	%        D(q)^2e^{-2V(q)}\,dq
	%    =
	%    1.
	%    \label{eq:diffusion-normalization}
	%\end{equation}
	%The corresponding normalized constant coefficient is
	%\begin{equation}
	%    D_{\mathrm{cst}}
	%    =
	%    \left(
	%        \int_{\mathbb{T}}e^{-2V(q)}\,dq
	%    \right)^{-1/2}.
	%    \label{eq:constant-diffusion}
	%\end{equation}
	The explicit optimiser of the one-dimensional homogenized problem
	studied in
	\cite{LelievrePavliotisRobinSantetStoltz2025} is
$
		H_{\mathrm{homo}}(q)={\rm{e}}^{V(q)}.
$
The homogenised sampler thus takes the form
	\begin{equation}
		\d X_t=\sqrt{2}\,{\rm{e}}^{V(X_t)/2}\,\d W_t,
		\label{eq:homogenized-principal-sde}
	\end{equation}
	with generator and Fokker--Planck equation
	\begin{equation}
		\ca{A}_{\mathrm{homo}}f={\rm{e}}^Vf'',\qquad
		\partial_t\rho_t=\partial_q^2({\rm{e}}^V\rho_t).
		\label{eq:homogenized-generator-fp}
	\end{equation}
	Since ${\rm{e}}^V\mu$ is a constant multiple of Lebesgue measure, $\mu$ is
	stationary for \eqref{eq:homogenized-generator-fp}.
	For the potential \eqref{eq:double-well-potential}, the numerical
	values reported in
	\cite{LelievrePavliotisRobinSantetStoltz2025} are
	\begin{equation}
		%\Lambda(D^\star)\simeq22.84,
		%\qquad
		\Lambda(H_{\mathrm{homo}})\simeq21.18,
		\qquad
		\Lambda(H_{\mathrm{cst}})\simeq8.46.
		\label{eq:double-well-spectral-gaps}
	\end{equation}
	%where $H^\star$ denotes the numerically optimised diffusion
	%coefficient. 
	%Thus
	%\begin{equation}
	%    \frac{
		%        \Lambda(D_{\mathrm{hom}})
		%    }{
		%        \Lambda(D_{\mathrm{cst}})
		%    }
	%    \simeq 2.50.
	%   \label{eq:spectral-gap-ratio}
	%\end{equation}
	The explicit multiplicative-noise coefficient $H_{\mathrm{homo}}$
	therefore yields a convergence rate approximately two and a half times
	that of the normalized additive-noise dynamics. The mechanism is transparent: since $H_{\mathrm{hom}}(q)
	\propto
	\frac{1}{\mu(q)},$
	the diffusivity is largest in the barrier regions, where $\mu$ is
	small, and smallest near the two modes.

	Let $m\geq1$, and let $
	\sigma
	=
	(\sigma_1,\ldots,\sigma_m):
	\mathbb{T}\longrightarrow\mathbb{R}^m$ 
	be a smooth noise frame satisfying
$
		|\sigma(q)|^2=
		\sum_{k=1}^m\sigma_k(q)^2=
		2{\rm{e}}^{V(q)}.
		\label{eq:noise-frame-constraint}
$
	Then the It\^o SDE
	\begin{equation}
		\d X_t=
		\sum_{k=1}^m
		\sigma_k(X_t)\,\d W_t^k
		\label{eq:general-noise-frame-sde}
	\end{equation}
	has generator $\ca{A}_{\mathrm{hom}}$ and therefore the same one-point
	Markov semigroup as \eqref{eq:homogenized-principal-sde}. For a fixed choice of $\sigma$, the associated
	synchronous two-point motion is
	\begin{equation}
		\begin{aligned}
			\d X_t=\,&
			\sum_{k=1}^m
			\sigma_k(X_t)\,\d W_t^k,
			\\
			\d Y_t
			=\,&
			\sum_{k=1}^m
			\sigma_k(Y_t)\,\d W_t^k,
		\end{aligned}
		\label{eq:synchronous-two-point-general}
	\end{equation}
	where both components are driven by the same $m$-dimensional Brownian
	motion. The corresponding noise vector fields on $\mathbb{T}^2$ are
	\begin{equation}
		U_k(x,y)
		=
		\sigma_k(x)\,\partial_x
		+
		\sigma_k(y)\,\partial_y,
		\qquad
		k=1,\ldots,m.
		\label{eq:two-point-vector-fields}
	\end{equation}
	
	Let $    \Delta
	=
	\{(q,q):q\in\mathbb{T}\}
	\label{eq:diagonal}
	$
	be the diagonal and let $
	\mu_\Delta
	=
	(\operatorname{Id},\operatorname{Id})_\#\mu,
	$
	be the diagonal coupling.

	\textbf{Failure of synchronisation for the principal noise frame}
	The principal square root of the scalar covariance
	$2H_{\mathrm{hom}}$ is
	\begin{equation}
		\sigma_{\mathrm{prin}}(q)=
		\sqrt{2}\,{\rm{e}}^{V(q)/2}.
		\label{eq:principal-noise-frame}
	\end{equation}
	It gives the canonical realization
	\eqref{eq:homogenized-principal-sde} considered in  \cite{LelievrePavliotisRobinSantetStoltz2025}, and the synchronous coupling
	\begin{equation}
		\begin{aligned}
		\d X_t=\,&
			\sigma_{\mathrm{prin}}(X_t)\,\d W_t,
			\\
			\d Y_t
			=\,&
			\sigma_{\mathrm{prin}}(Y_t)\,\d W_t.
		\end{aligned}
		\label{eq:principal-synchronous-coupling}
	\end{equation}
	
	The potential has the half-period symmetry
$
		V(q+\tfrac12)=V(q),
		\label{eq:half-period-potential}
$
	and thus,
	\begin{equation}
		\sigma_{\mathrm{prin}}(q+\tfrac12)=
		\sigma_{\mathrm{prin}}(q).
		\label{eq:half-period-sigma}
	\end{equation}
	\begin{proposition}
		\label{prop:principal-not-synchronizing}
		The principal synchronous coupling
		\eqref{eq:principal-synchronous-coupling} is not weakly synchronizing.
		Its two-point motion has at least two distinct invariant probability
		measures.
	\end{proposition}
	
	\begin{proof}
		Let $\tau_{1/2}(q)=q+\tfrac12$ denote translation by one half on $\mathbb{T}$, and define
		\begin{equation}
			\mu_{\Delta_{1/2}}=
			(\operatorname{Id},\tau_{1/2})_\#\mu,
			\qquad
			\mu_{\Delta_{1/2}}(\d x, \d y)
			=\mu(\d x)\delta_{x+1/2}(\d y).
			\label{eq:half-period-invariant-measure}
		\end{equation}
		Suppose that $Y_0=X_0+1/2$. If $X$ solves the first equation in
		\eqref{eq:principal-synchronous-coupling}, then, by
		\eqref{eq:half-period-sigma}, the process $X+1/2$ solves the same SDE as $Y$, with the same
		initial condition. Pathwise uniqueness yields $
		Y_t=X_t+\tfrac12
		\text{ for all }t\geq0
		\text{ a.s..}$ Hence, if $X_0\sim\mu$, then the joint
		law of $(X_t,Y_t)$ is $\mu_{\Delta_{1/2}}$ for every $t\geq0$ and
		$\mu_{1/2}$ is invariant.
		
	Since $\operatorname{supp}(\mu_{\Delta_{1/2}})
		=\big\{(q,q+\tfrac12):q\in\mathbb{T}\big\},$ which is disjoint from $\Delta$, then $\mu_{\Delta_{1/2}}\neq\mu_\Delta.$
		The two-point motion therefore has nonunique invariant probability
		measures, and in particular cannot be weakly synchronising.
	\end{proof}
	
	In fact,  the failure of synchronisation is not caused
	by a lack of local contraction because the principal realisation
	itself has a negative top Lyapunov exponent.  Taking the spatial derivative and applying It\^{o}'s formula, Birkhoff's ergodic theorem yields 
	\begin{align}
		\lambda_{\mathrm{top}}^{\rm{prin}}
		&=
		-\frac{1}{2}
		\int_{\mathbb{T}}
		|\sigma_{\mathrm{prin}}'(q)|^2\,\d\mu(q)=
		-\frac{1}{4Z}
		\int_{\mathbb{T}}
		|V'(q)|^2\,\d q
		=
		-\frac{2\pi^2}{Z}
		<0.
		\label{eq:principal-top-lyapunov}
	\end{align}
	The essential reason is that  the invariant graph
	$y=x+1/2$ prevents the two-point dynamics from reaching a neighborhood
	of the diagonal. This illustrates why negativity of the top Lyapunov
	exponent must be supplemented by a reachability or irreducibility
	condition for the two-point motion.
	
	\textbf{Synchronisation by an overcomplete H\"{o}rmander noise frame}
	
	The following construction is motivated by the general observation that the Lagrangian representation of a
	second-order equation is not unique: one may enlarge the driving noise and choose a representation whose two-point vector fields satisfy a H\"ormander condition.
	
Define the three-component noise frame
	\begin{equation}	\label{eq:overcomplete-frame}
		\sigma(q)=
		{\rm{e}}^{V(q)/2}
		\begin{pmatrix}
			1, \cos(2\pi q), 
			\sin(2\pi q)
		\end{pmatrix},
		\qquad
		q\in\mathbb{T}.
	\end{equation}
	Then
$
		\sum_{k=0}^2\sigma_k(q)^2
	=
		{\rm{e}}^{V(q)}
		\left[
		1+\cos^2(2\pi q)+\sin^2(2\pi q)
		\right]
		=
		2{\rm{e}}^{V(q)}.
$
Therefore, the SDE
	\begin{equation}
		\d X_t
		=
		{\rm{e}}^{V(X_t)/2}
		\left[
		\d W_t^0
		+
		\cos(2\pi X_t)\,\d W_t^1
		+
		\sin(2\pi X_t)\,\d W_t^2
		\right]
		\label{eq:overcomplete-sde}
	\end{equation}
	has exactly the same generator $\ca{A}_{\mathrm{homo}}$, invariant measure
	$\mu$, Fokker--Planck equation, and spectral gap as
	\eqref{eq:homogenized-principal-sde}. We next consider the two-point motion of \eqref{eq:overcomplete-sde}, 
	\begin{equation}
		\begin{aligned}
			\d X_t
			=\,&
			\sum_{k=0}^2
			\sigma_k(X_t)\,\d W_t^k,
			\\
			\d Y_t
			=\,&
			\sum_{k=0}^2
			\sigma_k(Y_t)\,\d W_t^k.
		\end{aligned}
		\label{eq:overcomplete-two-point}
	\end{equation}
	The associated two-point noise vector fields are
	\begin{equation}
		U_k(x,y)
		=
		\sigma_k(x)\,\partial_x
		+
		\sigma_k(y)\,\partial_y,
		\qquad
		k=0,1,2.
		\label{eq:overcomplete-two-point-fields}
	\end{equation}
	
	In the following we validate the assumptions of  \cref{thm:weaksynchronisationSDE}  for this particular example, thereby deducing synchronisation.

	\textbf{Negativity of the top Lyapunov exponent}
	\label{subsec:negative-lyapunov}
	
	Let $\varphi_t(\omega,q)$ denote the RDS generated by \eqref{eq:overcomplete-sde}. %Since the coefficients are smooth on the compact torus,
	%$\varphi_t(\omega,\cdot)$ is a $C^1$ stochastic flow of
	%diffeomorphisms. 
	Taking the spatial derivative and applying It\^{o}'s formula, Birkhoff's ergodic theorem yields
	\begin{equation}
		\lambda_{\mathrm{top}}\leq 
		-\frac{1}{2}
		\int_{\mathbb{T}}
		|\sigma'(q)|^2\,\d\mu(q).
		\label{eq:top-lyapunov-general}
	\end{equation}
Rewrite
$a(q)={\rm{e}}^{V(q)/2},
	g(q)=\bigl(
	1,\cos(2\pi q),\sin(2\pi q)\bigr).
$ Then, $\sigma=ag$. Since $
	|g(q)|^2=2,
	\,\,
	g(q)\cdot g'(q)=0,
	\,\,
	|g'(q)|^2=4\pi^2,$
	we obtain
	\begin{equation}
		|\sigma'(q)|^2=a(q)^2
		\left|\frac{V'(q)}{2}g(q)+g'(q)\right|^2
		={\rm{e}}^{V(q)}
		\left[\frac{1}{2}|V'(q)|^2+4\pi^2
		\right].
		\label{eq:sigma-prime-norm}
	\end{equation}
	Using
	$\d \mu(q)=Z^{-1}{\rm{e}}^{-V(q)}\,\d q$  and $V(q)=\cos(4\pi q)$, we have
	\begin{equation}	\label{eq:top-lyapunov-before-evaluation}
		\lambda_{\mathrm{top}}\leq -\frac{1}{2Z}
		\int_{\mathbb{T}}
		\left[\frac12|V'(q)|^2+4\pi^2
		\right]\d q\leq 	-\frac{4\pi^2}{Z}<0.
	\end{equation}	
	\textbf{H\"{o}rmander condition}
	
	\begin{proposition}
		\label{prop:hormander-off-diagonal}
		The vector fields $U_0,U_1,U_2$ span
		$T_{(x,y)}\mathbb{T}^2$ at every
		$(x,y)\in\mathbb{T}^2\setminus\Delta$. In particular, the synchronous
		two-point motion \eqref{eq:overcomplete-two-point}  satisfies H\"{o}rmander's condition, away from the diagonal.
	\end{proposition}
	
	\begin{proof}
		The values of the three vector fields are the columns of the matrix
		\begin{equation}
		U(x,y)=
			\begin{pmatrix}
				{\rm{e}}^{V(x)/2}
				&
				{\rm{e}}^{V(x)/2}\cos(2\pi x)
				&
				{\rm{e}}^{V(x)/2}\sin(2\pi x)
				\\
				{\rm{e}}^{V(y)/2}
				&
				{\rm{e}}^{V(y)/2}\cos(2\pi y)
				&
				{\rm{e}}^{V(y)/2}\sin(2\pi y)
			\end{pmatrix}.
			\label{eq:two-point-noise-matrix}
		\end{equation}
		Since the exponential factors are strictly positive,
		$U(x,y)$ has rank less than $2$ if and only if
	$g(x)=
		\bigl(1,\cos(2\pi x),\sin(2\pi x)\bigr)
$
		and $
		g(y)=\bigl(1,\cos(2\pi y),\sin(2\pi y)\bigr)
$
		are linearly dependent. Suppose that
	$g(y)=c\,g(x)
	$
		for some $c\in\mathbb{R}$. Comparing the first components gives
		$c=1$. Hence
$
	\cos(2\pi x)=\cos(2\pi y),\,
\sin(2\pi x)=\sin(2\pi y),
$
which implies that $x=y$ in $\mathbb{T}$. Therefore,
$
			\operatorname{rank}U(x,y)=2\,,
			\text{whenever }x\neq y,
$
and 
	\begin{equation}
		\operatorname{span}
		\{U_0(x,y),U_1(x,y),U_2(x,y)\}
		=T_{(x,y)}\mathbb{T}^2
	\end{equation}
		off the diagonal. 
	\end{proof}

	\textbf{Weak synchronisation and uniqueness of the two-point invariant measure}
	We now verify the hypotheses of the weak-synchronisation criterion developed in this work. The theorem is stated there for
	multiplicative SDEs on $\mathbb{R}^d$. We use its compact-manifold
	counterpart on $\mathbb{T}$; the proof is unchanged after replacing
	the Euclidean geometric-control and support arguments by their
	standard manifold versions.
	
	The required hypotheses are as follows.
	
	\begin{enumerate}
		\item
		\emph{Regularity.}
		The coefficients $\sigma_0,\sigma_1,\sigma_2$ in
		\eqref{eq:overcomplete-frame} are $C^\infty$ and, since the
		state space is compact, have bounded derivatives of every order.
		
		\item
		\emph{Strong mixing of the one-point motion.}
		Its generator is
	$
		\ca{A}_{\mathrm{hom}}=
		{\rm{e}}^V\partial_q^2$.
		Since the diffusion is uniformly elliptic on the connected compact
		manifold $\mathbb{T}$,  it  has the unique invariant
		probability measure $\mu$ and its transition semigroup is strongly
		mixing.
		
		\item
		\emph{Negative top Lyapunov exponent.}
		 \eqref{eq:top-lyapunov-before-evaluation} gives
$
		\lambda_{\mathrm{top}}<0.
$
		
		\item
		\emph{H\"ormander condition away from the diagonal.}
		Proposition~\ref{prop:hormander-off-diagonal} gives
		\[
		\operatorname{span}
		\{
		U_0(x,y),U_1(x,y),U_2(x,y)
		\}
		=
		T_{(x,y)}\mathbb{T}^2
		\]
		for every $(x,y)\notin\Delta$.
	\end{enumerate}
	
	The synchronisation criterion therefore implies that the RDS generated by \eqref{eq:overcomplete-sde} satisfies
	\begin{equation}
		d_{\mathbb{T}}
		\bigl(
		\varphi_t(x),\varphi_t(y)
		\bigr)
		\longrightarrow0
		\qquad
		\text{in probability}
		\label{eq:overcomplete-weak-synchronisation}
	\end{equation}
	for every $x,y\in\mathbb{T}$.
	
	Let $P_t^{(2)}$ denote the transition semigroup of the two-point
	motion,
	\begin{equation}
		P_t^{(2)}F(x,y)=
		\mathbb{E}
		\left[F\bigl(
		\varphi_t(x),\varphi_t(y)
		\bigr)\right]\,,
		\label{eq:two-point-semigroup}
	\end{equation}
	where $F$ is any bounded measurable function on $\bb{T}\times \bb{T}$.
	\begin{corollary}
		\label{cor:two-point-unique-invariant}
		For every $x,y\in\mathbb{T}$,
		\begin{equation}
			P_t^{(2)}
			\bigl((x,y),\cdot
			\bigr)
			\xrightarrow[t\to\infty]{w}
			\mu_\Delta\,.
			\label{eq:two-point-convergence}
		\end{equation}
		In particular, $\mu_\Delta$ is the unique invariant probability measure of the two-point motion.
	\end{corollary}
	
	\begin{proof}
		Strong mixing of the one-point process yields as $t\to\infty$,
		\begin{equation}
			\operatorname{Law}(\varphi_t(x))
			\xrightarrow[]{w}\mu,
			\qquad
			\operatorname{Law}(\varphi_t(y))
			\xrightarrow[]{w}\mu.
			\label{eq:marginal-convergence}
		\end{equation}
		Since $\mathbb{T}^2$ is compact, the family of joint laws of
		$(\varphi_t(x),\varphi_t(y))$ is tight. Let $\nu$ be any subsequential
		weak limit. By
		\eqref{eq:overcomplete-weak-synchronisation},
$
		\nu(\Delta)=1.
$
	By 
	\eqref{eq:marginal-convergence}, the first marginal of $\nu$ is $\mu$, and thus,
	$$
		\d\nu(x, y)= \d\mu_\Delta(x, y)
		=
		\d\mu(x)\d\delta_x(y).
	$$
		Hence every subsequential limit equals $\mu_\Delta$, which proves
		\eqref{eq:two-point-convergence}. Let $\rho$ be an invariant probability measure for
		$P_t^{(2)}$. For every $F\in C(\mathbb{T}^2)$,
		\[
		\int_{\mathbb{T}^2}F\,\d\rho
		=
		\int_{\mathbb{T}^2}
		P_t^{(2)}F(x,y)\,\d\rho(x, y).
		\]
		By \eqref{eq:two-point-convergence} and bounded convergence theorem,
		\[
		\lim_{t\to\infty}
		\int_{\mathbb{T}^2}
		P_t^{(2)}F(x,y)\,\d\rho(x, y)
		=
		\int_{\mathbb{T}^2}F\,\d\mu_\Delta.
		\]
		Therefore $\rho=\mu_\Delta$.
	\end{proof}
	
\section{Weak synchronisation for  multiplicative McKean--Vlasov systems}\label{sec4:weaksynchroforMVSDE}

In this section, we provide sufficient conditions for weak synchronisation of the following multiplicative McKean--Vlasov system
\begin{subequations}\label{system3}
	\begin{align}
		\d Y_t = & \,\,b(Y_t, \mu_t) \, \d t + \sigma(Y_t, \mu_t) \,\d W_t\,,\,\,\,\,\,\,\,\,\,\,\,\,\,\,\,\,\,\,\,\,\,\,\,\,\,\,\,\,\,\,\,\,\,\,\,\,\,\,\,\,\,\,\,\,\,\,\,\,\,\,\,\,\,Y_0=y\in \RR^d,\label{eq: MVRDE3}\\
		\partial_t \mu_t =&-\nabla\cdot(b(\cdot,\mu_t)\mu_t)+\frac{1}{2}D^2:\left((\sigma\sigma^T)(\cdot,\mu_t)\mu_t\right)\,, \,\,\,\,\,\,\,\,\,\,\,\,\,\,\,\mu_0=\mu\in\ca{P}_2(\RR^d)\label{eq:MVPDE3}\,.
	\end{align}
\end{subequations}
 It turns out that weak synchronisation of the coupled McKean--Vlasov system \eqref{system3} can be reduced to that of the corresponding frozen  SDE,  provided that \eqref{eq:MVPDE3} exhibits exponential ergodicity and certain additional conditions are satisfied.

\subsection{Generation of RDS for coupled McKean--Vlasov systems}
Recall that in \cite[Theorem 3.10]{gess2025random}, we  proved that, under suitable assumptions (see Assumptions 3.1, 3.7, 3.9) on the coefficients, there exists an  RDS $\varphi^{\rm{MV}}$ associated with the coupled system \eqref{system3}. More precisely,  
 \begin{align}
	&\varphi^{\rm{MV}}: [0,\infty) \times \Omega \times \RR^d\times \ca{P}_2(\RR^d) \longrightarrow \RR^d\times \ca{P}_2(\RR^d)\,,\\
	&\,\,\,\,\,\,\,\,\,\,\,\,\,\,\,(t,\omega,y,\mu)
	\longmapsto  
	( \phi_{t}(\omega,y,\mu),  \label{definedRDS}
	S_{t}(\mu))\,,
\end{align}
is an RDS over the metric dynamical system $(\Omega,\mathcal{G},\mathbb{P},(\theta_t)_{t\in \R})$ constructed in \cref{sec:generationRDSforSDEs}, 
where $S_t:\cP_2(\R^d)\to \cP_2(\R^d)$ is the solution map of~\eqref{eq:MVPDE3} and  $\phi_t(\omega,\cdot,\cdot): \R^d\times \cP_2(\R^d)\to \R^d$ is the solution map of~\eqref{eq: MVRDE3}.
Furthermore, we let $\ca{F}$ be the $\PP$-completion of the $\sigma$-algebra $\ca{G}$ and  $\bb{F}\coloneqq (\ca{F}_{s, t})_{-\infty<s\leq t<\infty}$, where 
$
	\ca{F}_{s, t}\coloneqq \sigma(W_u-W_v: s\leq u\leq v\leq t)\vee \ca{N}, \,-\infty<s\leq t<\infty\,,
$
and $\ca{N}$ denotes the collection of $\ca{F}$-null sets.
Then the tuple 
$$\big(\Omega, \mathcal{G}, \bb{F}, \mathbb{P}, (\theta_t)_{t\in \R}, \varphi^{\rm{MV}}\big)$$ is a white-noise RDS. 

\begin{assumption}\label{assum:existence_of_mu_infty}
	Assume that there exist a constant $\lambda>0$ and  a probability measure $\mu^{\infty}\in\ca{P}_2(\RR^d)$ such that, for any $\mu\in \ca{P}_2(\RR^d)$, the solution to \eqref{eq:MVPDE3} satisfies
	\begin{align}
		\ca{W}_2(S_t\mu, \mu^{\infty})\leq C_{\mu, \mu^{\infty}} {\rm{e}}^{-\lambda t}\,,
	\end{align}
	for some constant depending on $\mu, \mu^{\infty}$.
\end{assumption}
Under \cref{assum:existence_of_mu_infty}, we define 
$$b^{\infty}(\cdot)\coloneqq b(\cdot,\mu^{\i})\,,\;\;\;\;\;\sigma^{\infty}(\cdot)\coloneqq\sigma(\cdot,\mu^{\infty})\,,$$ 
and introduce the following frozen SDE
\begin{equation}\label{eq:inftySDE}
	\d Y_t^{\infty} = \,b^{\infty}(Y_t^{\infty}) \, \d t + \sigma^{\infty}(Y_t^{\infty}) \,\d W_t\,,\,\,\,\,\,\,\,\,\,Y_0^{\infty}=y,
\end{equation}
which gives rise to a Markov semigroup $(Q_t)_{t\geq 0}$. Since \eqref{definedRDS}
defines a white-noise RDS $$(\Omega, \mathcal{G}, \bb{F}, \mathbb{P}, (\theta_t)_{t\in \R}, \varphi^{\rm{MV}})$$ associated with the coupled system \eqref{system3}, fixing the measure component at $\mu^{\infty}$ yields a white-noise RDS 
$$\varphi^{{\rm{Fro}}}:[0,\i)\times \Omega \times \RR^d\to \RR^d$$ 
on the same metric dynamical system $\big(\Omega, \mathcal{G}, \bb{F}, \mathbb{P}, (\theta_t)_{t\in \R}\big)$
corresponding to the frozen SDE \eqref{eq:inftySDE}.

\subsection{Lifted semigroup and weak synchronisation for  coupled McKean--Vlasov systems}\label{subsec:liftsemigroup}
In this subsection, we introduce  the lifted semigroup  associated with the coupled McKean--Vlasov system and study weak synchronisation within this framework.

\begin{definition}[Lifted semigroup]
	For any $(y, \mu)\in \RR^d\times \ca{P}_2(\RR^d)$, and any bounded measurable function $f$ on $\RR^d\times \ca{P}_2(\RR^d)$, define
\begin{align}\label{liftedsemigroup}
		{\bf{P}}_t f(y, \mu)\coloneqq\EE f(Y_t, \mu_t)\,,
	\end{align}
where $Y_t$ and $\mu_t$ solve  \eqref{eq: MVRDE3} and \eqref{eq:MVPDE3},  
respectively. It follows from \cite[Proposition 4.8]{RenRocknerWang.2022.JDE1} that 	$({\bf{P}}_t)_{t\geq 0}$ \footnote{We emphasise the use of the semigroups $(S_t)_{t\geq 0}$, $(Q_t)_{t\geq 0}$, and $({\bf P}_t)_{t\geq 0}$, which are associated with the MVPDE \eqref{eq:MVPDE3}, the frozen SDE \eqref{eq:inftySDE}, and the coupled McKean–Vlasov system \eqref{eq: MVRDE3}–\eqref{eq:MVPDE3}, respectively.} is a time-homogeneous Markov semigroup on $\mathbb{R}^d \times \mathcal{P}_2(\mathbb{R}^d)$.
\end{definition}

\begin{lemma}
	The probability measures $\mu^{\infty}$ and $\mu^{\infty}\otimes \delta_{\mu^{\infty}}$ are  invariant  for    $(Y_t^{\infty})_{t\geq 0}$ and  $({\bf{P}}_t)_{t\geq 0}$, respectively.  
\end{lemma}
Let 
	\begin{equation}
		{\bf{E}}_0\coloneqq \,\,\Big\{(y, \mu)\in \RR^d\times \ca{P}_2(\RR^d): {\bf{P}}_t\big((y, \mu), \cdot\big)\xrightarrow[t\to\infty]{w}\footnotemark\mu^{\infty}\otimes\delta_{\mu^{\infty}}\Big\}\,.
\end{equation}
\footnotetext{Here, weak convergence is understood with respect to the topology induced by the metric
$d_2\big((x,\mu),(y,\nu)\big)
\coloneqq |x-y|+\mathcal W_2(\mu,\nu)$ on $\RR^d\times\ca P_2(\RR^d)$. More precisely, for every
$f\in C_b\big(\RR^d\times\ca P_2(\RR^d), d_2\big),$
$
\lim_{t\to\infty}{\bf{P}}_t f(y,\mu)=
\int_{\RR^d\times\ca P_2(\RR^d)}
f(x,\nu),
\big(\mu^\infty\otimes\delta_{\mu^\infty}\big)(\d x,\d\nu).
$}
\begin{theorem}\label{thm:weak_synchronisation_MVSDE}
	Assume that \cref{assum:existence_of_mu_infty} holds. Let 
	\begin{equation}
\mu^{\rm{Fro}}_{\omega}\coloneqq\,\,\lim_{t_k\to\infty}\varphi^{{\rm{Fro}}}(t_k,\theta_{-t_k}\omega, \cdot)\mu^{\infty}\,.
	\end{equation}
If ${\rm{supp}}\big(\mu^{{\rm{Fro}}}_{\omega}\big)$ is almost surely compact, then $${\rm{supp}}\big(\mu^{{\rm{Fro}}}_{\omega}\big)\times \{\mu^{\infty}\}$$
is a weak point attractor for ${\bf{E}}_0$. In particular, if ${\rm{supp}}\big(\mu^{{\rm{Fro}}}_{\omega}\big)$ is almost surely a singleton, then conditional weak synchronisation occurs for $\varphi^{\rm{MV}}$ on $\bf{E}_0$.
\end{theorem}
\begin{proof}
	One can  check that, for any $t\geq 0$, for $\omega\in \Omega$,
	\begin{align}
		\varphi^{\rm{MV}}(t,\omega, (\cdot,\cdot))(\mu^{\infty}\otimes  \delta_{\mu^{\infty}})&=\varphi^{\rm{MV}}(t,\omega,(\cdot,\mu^{\infty}))\mu^{\infty}\otimes  \delta_{\mu^{\infty}}\\
		&=\varphi^{\rm{Fro}}(t,\omega, \cdot)\mu^{\infty}\otimes  \delta_{\mu^{\infty}}\,.
	\end{align}
Hence, the statistical equilibrium of $\varphi^{\rm{MV}}$ associated with the invariant probability measure $\mu^{\infty}\otimes  \delta_{\mu^{\infty}}$ exists and satisfies
	\begin{align}
\lim_{t_{k}\to\infty}\varphi^{\rm{MV}}(t_k,\theta_{-t_k}\omega, (\cdot,\cdot))(\mu^{\infty}\otimes  \delta_{\mu^{\infty}})=\mu^{\rm{Fro}}_{\omega}\otimes  \delta_{\mu^{\infty}}.\label{muinfinity}
	\end{align}
Since ${\rm{supp}}\big(\mu^{\rm{Fro}}_{\omega}\big)$ is almost surely compact, to prove that ${\rm{supp}}\big(\mu^{\rm{Fro}}_{\omega}\big)\times \{\mu^{\infty}\}$ is a weak point attractor for $\bf{E}_0$, it remains to show that this random set is invariant and  attracts every point in $\bf{E}_0$ in probability.

	\vspace{2mm}
	\noindent {\bf{Step (i):}} Fix $\varepsilon>0$. Let $(y_1, \mu_1), (y_2, \mu_2),\ldots,$ be a countable dense subset of  $\RR^d\times \ca{P}_2(\RR^d)$. Denote by 
	\begin{align}
		B_i\coloneqq \,\,&\Big\{(y, \mu): d_2\big((y, \mu), (y_i, \mu_i)\big) < \varepsilon/3\Big\}\,,\\[1mm]
		I(\omega)\coloneqq \,\,&\min \left\{i: {\rm{supp}}\big(\mu^{\rm{Fro}}_{\omega}\big)\times \{\mu^{\infty}\}\subseteq \cup_{j=1}^i B_j\right\}\,.
	\end{align}
	Since 	${\rm{supp}}\big(\mu^{\rm{Fro}}_{\omega}\big)\times \{\mu^{\infty}\}$ is compact almost surely, $I(\omega)<\infty$ almost surely.  Moreover, let 
	\begin{align}
		\tilde{A}(\omega)\coloneqq\,\,& \cup_{j\in J(\omega)} B_j=\cup_{j\in \big\{j\leq I(\omega): B_j\cap \big({\rm{supp}}(\mu^{\rm{Fro}}_{\omega})\times \{\mu^{\infty}\}\big)\neq \emptyset \big\}} B_j\,,\\[1mm]
			\beta(\omega)\coloneqq\,\,& \min \big\{\big( \mu^{\rm{Fro}}_{\omega}\otimes \delta_{\mu^{\infty}}\big)(B_j), \,j\in J(\omega)\big\}.
	\end{align}
	Then $\tilde{A}(\omega)$ is a  bounded open set and $\beta(\omega)>0$ almost surely.
	
	\vspace{2mm}
	\noindent {\bf{Step (ii):}}  For any $\varepsilon>0$, we define 
	\begin{align}
		 \big({\rm{supp}}\big(\mu^{\rm{Fro}}_{\omega}\big)\times \{\mu^{\infty}\}\big)^{\varepsilon}\coloneqq \,\,&\left\{(y, \mu): d_2\big((y, \mu), {\rm{supp}}\big(\mu^{\rm{Fro}}_{\omega}\big)\times \{\mu^{\infty}\}\big) < \varepsilon\right\}\\[1mm]
		=\,\,&\Big\{(y, \mu):\inf_{(x, \nu)\in {\rm{supp}}\big(\mu^{\rm{Fro}}_{\omega}\big)\times \{\mu^{\infty}\} } d_2\big((y, \mu), (x, \nu)\big)<\varepsilon\Big\}\,.
	\end{align}
 Given $\gamma>0$, we construct  a random set $\bar{A}_{n, b}(\cdot)$ such that 
	\begin{align}
	\PP\Big( \omega: \tilde{A}(\omega)\subseteq\bar{A}_{n, b}(\omega)\subseteq \big({\rm{supp}}\big(\mu^{\rm{Fro}}_{\omega}\big)\times \{\mu^{\infty}\}\big)^{\varepsilon}\Big)\geq 1-\frac{\gamma}{4}\,.\label{eq:three_sets}
	\end{align}
Indeed, since $\lim_{n\to\infty}\varphi^{\rm{MV}}(n, \theta_{-n}\omega)(\mu^{\infty}\otimes \delta_{\mu^{\infty}})=\mu^{\rm{Fro}}_{\omega}\otimes  \delta_{\mu^{\infty}}$ and $\mu^{\rm{Fro}}_{\omega}\otimes  \delta_{\mu^{\infty}}\big(\big({\rm{supp}}\big(\mu^{\rm{Fro}}_{\omega}\big)\times \{\mu^{\infty}\}\big)^{\varepsilon/3}\big)=1$, it follows that, $\PP$-a.s. for any $b>0$, there exists $n_0$, such that for $n\geq n_0$,
\begin{align}\label{eq:mass_on_third_varepsilon}
\varphi^{\rm{MV}}(n, \theta_{-n}\omega)(\mu^{\infty}\otimes \delta_{\mu^{\infty}})\big(\big({\rm{supp}}\big(\mu^{\rm{Fro}}_{\omega}\big)\times \{\mu^{\infty}\}\big)^{\varepsilon/3}\big)\geq 1-\frac{b}{2}\,.
\end{align}
%Choose  $(y, \mu)\in \RR^d\times \ca{P}_2(\RR^d)$ such that the mass of push-forward measure on its ball is greater than $b$ and 
Let 
$$D(n, b, \omega)\coloneqq \Big\{(y, \mu): \Big(\varphi^{\rm{MV}}(n, \theta_{-n}\omega)(\mu^{\infty}\otimes \delta_{\mu^{\infty}})\Big) B\big((y, \mu), \varepsilon/3\big)\geq b\Big\}\,.$$
Define 
$$\bar{A}_{n, b}(\omega)\coloneqq \cup_{(y, \mu)\in D(n, b, \omega)} B\big((y, \mu), \varepsilon/3\big)\,.$$
	
By the construction of $\bar{A}_{n, b}$ and \eqref{eq:mass_on_third_varepsilon}, for   $n\geq n_0$ and  $(y, \mu)\in D(n, b, \omega)$, we have $\big({\rm{supp}}\big(\mu^{\rm{Fro}}_{\omega}\big)\times \{\mu^{\infty}\}\big)^{\varepsilon/3} \cap B\big((y, \mu), \varepsilon/3\big)\neq \emptyset$, $B\big((y, \mu), \varepsilon/3\big)\subseteq \big({\rm{supp}}\big(\mu^{\rm{Fro}}_{\omega}\big)\times \{\mu^{\infty}\}\big)^{\varepsilon}$ and $\bar{A}_{n, b}(\omega)\subseteq \big({\rm{supp}}\big(\mu^{\rm{Fro}}_{\omega}\big)\times \{\mu^{\infty}\}\big)^{\varepsilon}$.
	
Furthermore, by the definition of $\beta(\omega)$, for $\omega\in \Omega$ such that $\beta(\omega)>b$, it holds that 
$\big( \mu^{\rm{Fro}}_{\omega}\otimes \delta_{\mu^{\infty}}\big)(B_j)>b$, $j\in J(\omega)$, and for $n$ sufficiently large,  $\tilde{A}(\omega)\subseteq \bar{A}_{n, b}(\omega)$. Then, 
\begin{align}
	\liminf_{n\to\infty}\PP \big(\omega: \tilde{A}(\omega)\subseteq \bar{A}_{n, b}(\omega)\big)\geq \PP (\omega: \beta(\omega)>b).
\end{align} 
Therefore, for  given $\gamma>0$, we may find $b$ sufficiently small and $n_0$ sufficiently large  such that for all $n\geq n_0$,
\begin{align}
	\PP \Big(\omega: \tilde{A}(\omega)\subseteq \bar{A}_{n, b}(\omega)\subseteq \big({\rm{supp}}\big(\mu^{\rm{Fro}}_{\omega}\big)\times \{\mu^{\infty}\}\big)^{\varepsilon}\Big)\geq 1-\frac{\gamma}{4}\,,
\end{align}
which completes the construction of $\bar{A}_{n, b}$ in \eqref{eq:three_sets}.

Moreover, since ${\rm{supp}}\big(\mu^{\rm{Fro}}_{\omega}\big)\times \{\mu^{\infty}\}\subseteq \tilde{A}(\omega)$,   we have $\mu^{\rm{Fro}}_{\omega}\otimes \delta_{\mu^{\infty}}\big(\tilde{A}(\omega)\big)=1$. Since $\tilde{A}(\omega)$ is open,  the weak convergence of $\lim_{n\to\infty}\varphi^{\rm{MV}}(n, \theta_{-n}\omega)(\mu^{\infty}\otimes \delta_{\mu^{\infty}})=\mu^{\rm{Fro}}_{\omega}\otimes \delta_{\mu^{\infty}}$ and Portmanteau theorem yield that  for $\PP$-almost every $\omega$, there exists $n_1$,
$
	\varphi^{\rm{MV}}(n, \theta_{-n}\omega)(\mu^{\infty}\otimes \delta_{\mu^{\infty}})\big(\tilde{A}(\omega)\big)\geq 1-\frac{\gamma}{3}\,.$ Hence, \eqref{eq:three_sets} implies that there exists $n_1>n_0$ such that 
	\begin{align}
		&\PP\Big(\omega:\varphi^{\rm{MV}}(n, \theta_{-n}\omega)(\mu^{\infty}\otimes \delta_{\mu^{\infty}})\big( \bar{A}_{n, b}(\omega)\big)\geq 1-\frac{\gamma}{3}\Big)\geq 1-\frac{\gamma}{3}\,.\label{eq:square_probability}
		\end{align}

	\noindent {\bf{Step (iii):}}  We prove that for every $(y, \mu)\in {\bf{E}_0}$, ${\rm{supp}}\big(\mu^{\rm{Fro}}_{\omega}\big)\times \{\mu^{\infty}\}$ attracts $(y, \mu)$ in probability. Indeed,
	\begin{align}
		&\liminf_{s\to \infty}\PP \Big(\omega: d_2\big(\varphi^{\rm{MV}}(s+n, \theta_{-(s+n)}\omega, (y, \mu)), {\rm{supp}}\big(\mu^{\rm{Fro}}_{\omega}\big)\times \{\mu^{\infty}\}\big)<\varepsilon \Big)\\[1mm]
		=\,\,& \liminf_{s\to \infty}\PP \Big(\omega:\varphi^{\rm{MV}}(s+n, \theta_{-(s+n)}\omega, (y, \mu))\in \big( {\rm{supp}}\big(\mu^{\rm{Fro}}_{\omega}\big)\times \{\mu^{\infty}\}\big)^\varepsilon \Big)\\[1mm]
		\geq \,\,& \liminf_{s\to \infty}\PP \Big(\omega:\varphi^{\rm{MV}}(s+n, \theta_{-(s+n)}\omega, (y, \mu))\in \bar{A}_{n, b}(\omega)\Big)-\frac{\gamma}{4}\\[1mm]
		= \,\,& \liminf_{s\to \infty}\PP \Big(\omega:\varphi^{\rm{MV}}(s, \theta_{-(s+n)}\omega, (y, \mu))\in \varphi^{\rm{MV}}(n, \theta_{-n}\omega)^{-1}\bar{A}_{n, b}(\omega)\Big)-\frac{\gamma}{4}\,,\label{eq:delt34}
	\end{align}
	where in the second step, we use the fact that 
	\begin{align}
		&\PP \Big(\omega:\varphi^{\rm{MV}}(s+n, \theta_{-(s+n)}\omega, (y, \mu))\in \bar{A}_{n, b}(\omega)\Big)\\[1mm]
		\leq\,\,&\PP \Big(\omega:\varphi^{\rm{MV}}(s+n, \theta_{-(s+n)}\omega, (y, \mu))\in \big({\rm{supp}}\big(\mu^{\rm{Fro}}_{\omega}\big)\times \{\mu^{\infty}\}\big)^{\varepsilon}\Big)\\[1mm]
		&+\PP \Big(\omega:\varphi^{\rm{MV}}(s+n, \theta_{-(s+n)}\omega, (y, \mu))\in \bar{A}_{n, b}(\omega)\setminus \big({\rm{supp}}\big(\mu^{\rm{Fro}}_{\omega}\big)\times \{\mu^{\infty}\}\big)^{\varepsilon}  \Big)\\[1mm]
		\leq \,\,&\PP \Big(\omega:\varphi^{\rm{MV}}(s+n, \theta_{-(s+n)}\omega, (y, \mu))\in \big({\rm{supp}}\big(\mu^{\rm{Fro}}_{\omega}\big)\times \{\mu^{\infty}\}\big)^{\varepsilon} \Big)\\[1mm]
		&+\PP \Big(\omega: \bar{A}_{n, b}(\omega)\setminus \big({\rm{supp}}\big(\mu^{\rm{Fro}}_{\omega}\big)\times \{\mu^{\infty}\}\big)^{\varepsilon} \Big)\\[1mm]
		\leq \,\,&\PP \Big(\omega:\varphi^{\rm{MV}}(s+n, \theta_{-(s+n)}\omega, (y, \mu))\in \big({\rm{supp}}\big(\mu^{\rm{Fro}}_{\omega}\big)\times \{\mu^{\infty}\}\big)^{\varepsilon} \Big)+\frac{\gamma}{4}\,.
	\end{align}
	
	To estimate \eqref{eq:delt34}, we observe that for any $a, b, c, e\in \RR$, $d\in \RR_+$, and $y\in \RR^d$,
	\begin{align}
		\theta_a\ca{F}_{b, c}=\ca{F}_{b-a, c-a}\,,\,\,\,\,\,\,\,\varphi^{\rm{MV}}(d, \theta_e\omega)y\,\,\,\text{is}\,\,\,\ca{F}_{e, d+e}-\ca{B}(\RR^d)\,\,\text{measurable}\,.
	\end{align}
	Then  $\bar{A}_{n, b}(\cdot)$ is $\ca{F}_{-n, 0}$-measurable,  $\varphi^{\rm{MV}}(n, \theta_{-n}\cdot)^{-1}\big( \bar{A}_{n, b}\big)\,\,\,\,\,\text{is}\,\, \ca{F}_{-n, 0}- \ca{B}(\RR^d)-\text{measurable}\,,$ and  $\varphi^{\rm{MV}}(s, \theta_{-(s+n)}\cdot, (y, \mu))$  is $\ca{F}_{-\infty, -n}-\ca{B}(\RR^d)$-measurable. Therefore, by the independence of $\ca{F}_{-\infty, -n}$ and  $\ca{F}_{-n, 0}$, and Fatou's lemma,
	\begin{align}
		&\liminf_{s\to\infty}\PP \Big(\omega:\varphi^{\rm{MV}}(s, \theta_{-(s+n)}\omega, (y, \mu))\in \varphi^{\rm{MV}}(n, \theta_{-n}\omega)^{-1}\bar{A}_{n, b}(\omega)\Big)\\[1mm]
		=\,\,&\liminf_{s\to\infty}\EE_{\omega'}\PP \Big(\omega:\varphi^{\rm{MV}}(s, \theta_{-(s+n)}\omega, (y, \mu))\in \varphi^{\rm{MV}}(n, \theta_{-n}\omega')^{-1}\bar{A}_{n, b}(\omega')\Big)\\[1mm]
	\geq\,\,&\EE_{\omega'}\liminf_{s\to\infty}\PP \Big(\omega:\varphi^{\rm{MV}}(s, \theta_{-(s+n)}\omega, (y, \mu))\in \varphi^{\rm{MV}}(n, \theta_{-n}\omega')^{-1}\bar{A}_{n, b}(\omega')\Big)\\[1mm]
\geq\,\,&\EE_{\omega'}\big(\mu^{\infty}\otimes \delta_{\mu^{\infty}}\big) \Big(\varphi^{\rm{MV}}(n, \theta_{-n}\omega')^{-1}\bar{A}_{n, b}(\omega')\Big)\\[1mm]
		\geq \,\,&\int_{\big\{\omega': (\mu^{\infty}\otimes \delta_{\mu^{\infty}}) (\varphi^{\rm{MV}}(n, \theta_{-n}\omega')^{-1}\bar{A}_{n, b}(\omega'))\geq 1-\gamma/3\big\}}\big(\mu^{\infty}\otimes \delta_{\mu^{\infty}}\big) \big(\varphi^{\rm{MV}}(n, \theta_{-n}\omega')^{-1}\bar{A}_{n, b}(\omega')\big) \d \PP\\[1mm]
		\geq \,\,& (1-\gamma/3)^2\,,
	\end{align}
	where we use  the definition of $\bf{E}_0$ in the third step and \eqref{eq:square_probability} in the last step.
	Plugging this estimate into \eqref{eq:delt34}, we obtain 
	\begin{align}
		\liminf_{s\to \infty}\PP \Big(\omega: d_2\big(\varphi^{\rm{MV}}(s+n, \theta_{-(s+n)}\omega, (y, \mu)), {\rm{supp}}\big(\mu^{\rm{Fro}}_{\omega}\big)\times \{\mu^{\infty}\}\big)<\varepsilon \Big)\geq 1-\gamma\,,
	\end{align}
	which implies that ${\rm{supp}}\big(\mu^{\rm{Fro}}_{\omega}\big)\times \{\mu^{\infty}\}$ attracts every $(y, \mu)\in{\bf{E}_0}$ in probability by the arbitrariness of $\varepsilon$ and $\gamma$.
	
	\noindent {\bf{Step (iv):}} It remains to prove the 
	invariance of  ${\rm{supp}}\big(\mu^{\rm{Fro}}_{\omega}\big)\times \{\mu^{\infty}\}$  under $\varphi^{\rm{MV}}$. 
Since $\varphi^{\rm{Fro}}(t,\omega,\cdot)$ is a homeomorphism, we have
\begin{equation}
	{\rm{supp}}\Big(
	\varphi^{\rm{Fro}}(t,\omega,\cdot)\mu_{\omega}^{\rm{Fro}}\Big)=	\varphi^{\rm{Fro}}\big(
	t,\omega,{\rm{supp}}(\mu_{\omega}^{\rm{Fro}})
	\big)\,.
\end{equation}
On the other hand, by the invariance of the statistical equilibrium,
$\varphi^{\rm{Fro}}(t,\omega,\cdot)\mu_{\omega}^{\rm{Fro}}=\mu_{\theta_t\omega}^{\rm{Fro}}\,.
$
Therefore,
\begin{align}
	\varphi^{\rm{Fro}}\big(
	t,\omega,{\rm{supp}}(\mu_{\omega}^{\rm{Fro}})
	\big)={\rm{supp}}\big(\mu_{\theta_t\omega}^{\rm{Fro}}\big)\,,
\end{align}
which proves the desired invariance.

\end{proof}

\begin{proposition}[Characterisation of $\bf{E}_0$]\label{prop:characterisation_of_E_0}
	Assume the following three conditions.
	\begin{enumerate}[(i)]
		\item The McKean--Vlasov PDE \eqref{eq:MVPDE3} satisfies \cref{assum:existence_of_mu_infty};
		\item The frozen SDE \eqref{eq:inftySDE} is strongly mixing with respect to $\mu^{\infty}$ and there exists a constant $C>0$  such that for  any  $\mu, \nu\in\ca{P}(\RR^d)$,   the solution to  \eqref{eq:inftySDE} satisfies 
		\begin{equation}
			\sup_{t\geq 0}d_w\big(Q_t^*\mu, Q_t^*\nu\big)\leq C d_w\big(\mu, \nu)\,.\label{es:exponential_ergodicity_frozen}
		\end{equation}
		\item For any $x, y\in \RR^d$, $\mu, \nu\in \ca{P}_2(\RR^d)$, 
		\begin{align}
        %|b(x, \mu)-b(y, \mu)|+|\sigma(x, \mu)-\sigma(y, \mu)|\leq\,\,& F(M_2(\mu), 0)|x-y|\,,\\
        %|b(0, \mu)|+|\sigma(0, \mu)|\leq\,\,&F(M_2(\mu), 0)\,,\\
%2\langle x, b(x,\mu_t)\rangle
%+\Tr(\sigma\sigma^T(x,\mu_t))
%\le -c|x|^2+C,\qquad t\ge T_0\,,\\
			|b(y, \mu)-b(y, \nu)|+|\sigma\sigma^T(y, \mu)-\sigma\sigma^T(y, \nu)|\leq\,\,& F(M_2(\mu), M_2(\nu))(1+|y|)\ca{W}_2(\mu, \nu)\,,\label{eq:E_0_local_Lipschitz}
		\end{align}
		where $F: \RR_+^2\to\RR_+$ is non-decreasing in each variable. Moreover, 
		for any $(y, \mu)\in \RR^d\times \ca{P}_2(\RR^d)$, the solution to the McKean--Vlasov system \eqref{eq: MVRDE3} satisfies 
	\begin{equation}\label{es:uniform_in_time_estimate}
	    \sup_{t\geq 0}M_1\big(\Law(Y_t)\big)<\infty\,.
	\end{equation}

	\end{enumerate}
	Then 
	\begin{equation}
		{\bf{E}_0}=\RR^d\times \ca{P}_2(\RR^d)\,.
	\end{equation}
	%$({\bf{P}}_t)_{t\geq 0}$ is strongly mixing with respect to $\mu^{\infty}\otimes \delta_{\mu_{\infty}}.$
	
\end{proposition}
\begin{proof}
	We first show that  $\mu^{\infty}\otimes \delta_{\mu^{\infty}}$ is the unique invariant probability measure of $({\bf{P}}_t)_{t\geq 0}$, and then conclude by a  uniform-on-compacts strong mixing argument.
	
	Suppose that $\lambda\in \ca{P}\big(\RR^d\times \ca{P}_2(\RR^d)\big)$ is an invariant probability measure of $({\bf{P}}_t)_{t\geq 0}$, i.e. if $\Law(Y_0, \mu_0)=\lambda$
, then $\Law(Y_t, \mu_t)=\lambda$ for any $t\geq 0$. Then, taking marginals, we have
	\begin{align}
		\Law(Y_t)=\pi_1 \lambda\,,\,\,\,\,\,\,\,\,\,\Law(\mu_t)=\pi_2\lambda\,,\,\,\,\,\,\,\,\,\,\,t\geq 0\,.
	\end{align}
Using the exponential ergodicity in \cref{assum:existence_of_mu_infty}, we know that $\mu^{\infty}$ is the unique stationary solution to \eqref{eq:MVPDE3} and thus $\pi_2\lambda=\delta_{\mu^{\infty}}$.  Then the dynamics of the coupled McKean--Vlasov system \eqref{eq: MVRDE3}-\eqref{eq:MVPDE3} with initial distribution $\lambda$ coincides with the system $(Y_t^{\infty}, \mu^{\infty})$  with initial distribution $(\pi_1\lambda, \mu^{\infty})$. Moreover, by  condition (ii), it holds that $\mu^{\infty}$ is also the unique invariant probability measure of  \eqref{eq:inftySDE} and thus $\pi_1\lambda=\mu^{\infty}$. Consequently, every invariant probability measure of $({\bf{P}}_t)_{t\geq 0}$ takes the form of $\mu^{\infty}\otimes \delta_{\mu^{\infty}}$.

	By \eqref{es:uniform_in_time_estimate} and (i), for every $(y,\mu)\in\RR^d\times\ca P_2(\RR^d)$, the family $\big(\Law(Y_t,\mu_t)\big)_{t\geq0}$ is tight.  Prokhorov's theorem implies that every sequence $t_n\to\infty$ admits a  subsequence $(t_{n})_{n\geq1}$, which we do not relabel  for notational convenience,   such  that 
	\begin{align}
		\Law(Y_{t_{n}}, \mu_{t_{n}})\xrightarrow[n\to\infty]{w} \tilde{\lambda}\,,
	\end{align}
	for some $\tilde{\lambda}\in \ca{P}(\RR^d\times \ca{P}_2(\RR^d))$. We will now show that $\tilde{\lambda}=\mu^{\infty}\otimes \delta_{\mu^{\infty}}\,.$ In fact,  due to \cref{assum:existence_of_mu_infty}, we know that $\pi_2\tilde{\lambda}=\delta_{\mu^{\infty}}$. Since any probability measure on
	$\RR^d\times\ca{P}_2(\RR^d)$ whose second marginal is a Dirac measure is necessarily a product measure, there exists $\tilde{\mu}\in\ca{P}(\RR^d)$ such that
	\begin{align}
		\tilde{\lambda}=\tilde{\mu}\otimes\delta_{\mu^{\infty}}.
	\end{align}
	We are left to show that
	$\tilde{\mu}=\mu^{\infty}$. 
	Let $\nu_t\coloneqq\Law(Y_t)$. Recall that $Y_0=y$ and that the measure argument in the equation for $Y_t$ in \eqref{eq: MVRDE3} is given by the deterministic curve of probability measures $(\mu_t)_{t\geq0}$ defined by \eqref{eq:MVPDE3}.
	Then, the family $(\nu_t)_{t\geq0}$ satisfies the linear time-inhomogeneous
	Fokker--Planck equation
	\begin{equation}
		\partial_t\nu_t
		=
		-\nabla\cdot\big(b(\cdot,\mu_t)\nu_t\big)
		+
		\frac{1}{2}D^2:
		\big(
		\sigma(\cdot,\mu_t)\sigma(\cdot,\mu_t)^T\nu_t
		\big)\,,\,\,\,\,\,\,\,\nu_0=\delta_y.
	\end{equation}
By condition (iii), we know that  $\sup_{t\in [0, T]}M_2(\nu_t)<\infty$.
    
Let $K\coloneqq \overline{ \{\nu_t: t\geq 0\}}^{d_w}$. 	Fix $T>0$ and for $n$ large enough, set
	$
	s_{n}\coloneqq t_{n}-T
	$.	 Since $K$ is compact, there	exists a subsequence  $(s_{n})_{n\geq1}$, which we again do not relabel for notational simplicity, and $\rho_T\in K$ such that
	\begin{equation}
		\nu_{s_{n}}
		\xrightarrow[n\to\infty]{w}
		\rho_T.
	\end{equation}
	Let $(q_t^{\rho_T})_{t\in[0,T]}$ be the solution to the frozen
	Fokker--Planck equation
	\begin{equation}\label{eq:q_rho_T}
		\partial_tq_t^{\rho_T}=
		-\nabla\cdot\big(
		b(\cdot,\mu^{\infty})q_t^{\rho_T}\big)
		+\frac{1}{2}D^2:
		\big(\sigma(\cdot,\mu^{\infty})
		\sigma(\cdot,\mu^{\infty})^T
		q_t^{\rho_T}\big),
		\qquad
		q_0^{\rho_T}=\rho_T.
	\end{equation}
and let $u$ solve the backward Kolmogorov equation
	\begin{align}
		\partial_tu(t,x)
		={}&
		-\langle b(x,\mu^{\infty}),\nabla u(t,x)\rangle
		-\frac{1}{2}
		D^2u(t,x):
		\big(
		\sigma(x,\mu^{\infty})
		\sigma(x,\mu^{\infty})^T
		\big),
		\\
		u(T,\cdot)
		={}&f,
		\qquad
		f\in C_b^2(\RR^d).
	\end{align}
	Then
	\begin{align}
		&\int_{\RR^d}
		u(T,x)\,\d\big(\nu_{s_{n}+T}-q_T^{\rho_T}
		\big)\\
		={}&\int_{\RR^d}	u(0,x)\,
		\d\big(\nu_{s_{n}}-\rho_T\big)+\int_0^T\frac{\d}{\d t}
		\int_{\RR^d}u(t,x)\,\d\big(\nu_{s_{n}+t}-q_t^{\rho_T}
		\big)
		\,\d t.
		\label{eq:fundamental_calculus}
	\end{align}
	Using the equations satisfied by $\nu_{s_{n}+t}$,
	$q_t^{\rho_T}$, and $u$, we obtain \footnote{The proof is analogous to the uniqueness part of the proof of \cite[Theorem 3.4]{gess2025random}.}
	\begin{align}
		&\frac{\d}{\d t}
		\int_{\RR^d}u(t,x)\,\d\big(\nu_{s_{n}+t}-q_t^{\rho_T}\big)\\
		=\,&
		\int_{\RR^d}
		\langle\nabla u(t,x), b(x,\mu_{s_{n}+t})-b(x,\mu^{\infty})\rangle
		\d\nu_{s_{n}+t}\\
		&+\frac{1}{2}\int_{\RR^d}D^2u(t,x):
		\big(\sigma\sigma^T(x,\mu_{s_{n}+t})-\sigma\sigma^T(x,\mu^{\infty})\big)
		\d\nu_{s_{n}+t}.
		\label{eq:derivative_asymptotic_autonomy}
	\end{align}
	By the convergence of $\mu_t$ to $\mu^\infty$ in \cref{assum:existence_of_mu_infty}, together with \eqref{eq:E_0_local_Lipschitz} and the bound $\sup_{t\in[0,T]}M_2(\nu_t)<\infty$, the right-hand side of \eqref{eq:derivative_asymptotic_autonomy} converges to $0$ uniformly for $t\in[0,T]$. Hence,
\begin{equation}
   \lim_{n\to\infty}\int_{\RR^d}f(x),\d\big(
\nu_{s_n+T}-q_T^{\rho_T}\big)=0. 
\end{equation}
Therefore, we obtain the following asymptotic autonomy:
\begin{equation}
\nu_{s_n+T}\xrightarrow[n\to\infty]{w}q_T^{\rho_T}.
\end{equation}

Since $\nu_{t_n}\xrightarrow[n\to\infty]{w}\tilde{\mu}$, by the uniqueness of the weak limit,  we have $q_T^{\rho_T}=Q_T^*{\rho_T}=\tilde{\mu}$.
	Using \eqref{es:exponential_ergodicity_frozen}, we know that for any $\varepsilon>0$, there exists $\delta$, such that for $d_w(\mu, \nu)\leq \delta$, $d_w(Q_t^*\mu, Q_t^*\nu)\leq \varepsilon$.  Since $K$ is a compact set in $(\ca{P}(\RR^{d}), d_w)$, it is totally bounded and we can find finitely many $\lambda_1, \ldots, \lambda_N\in \ca{P}(\RR^{d})$ such that 
	\begin{equation}
		K\subset \cup_{i=1}^N B_{d_w}(\lambda_i, \delta)\,.
	\end{equation}
	On the other hand, for each fixed $\lambda_i$, 
	\begin{equation}
		d_w\big(Q_T^* \lambda_i, \mu^{\infty}\big)\xrightarrow[T\to\infty]{}0\,,
	\end{equation}
	and there exists $T_0$ such that for all $T\geq T_0$, 
	\begin{equation}
		\max_{1\leq i\leq N} d_w\big(Q_T^* \lambda_i, \mu^{\infty}\big)\leq \varepsilon\,.
	\end{equation}
	For arbitrary $\lambda\in K$, choose $i_0$ such that $d_w(\lambda, \lambda_{i_0})\leq \delta$. Then, for all $T\geq T_0$, 
	\begin{align}
		d_w\big(Q_T^*\lambda, \mu^{\infty}\big)\leq\,\,& d_w\big(Q_T^* \lambda, Q_T^* \lambda_{i_0}\big)+d_w\big(Q_T^* \lambda_{i_0}, \mu^{\infty}\big)\\[1mm]
		\leq\,\,& \varepsilon+	\max_{1\leq i\leq N} d_w\big(Q_T^*\lambda_i, \mu^{\infty}\big)\\[1mm]
		\leq\,\,& 2\varepsilon\,,
	\end{align}
	and  
	\begin{equation}
		d_w(\tilde{\mu},\mu^{\infty})=
		d_w\big(Q_T^* \rho_T, Q_T^*\mu^{\infty} 
		\big)\leq \sup_{\lambda\in K}d_w\big(Q_T^*\lambda, \mu^{\infty}\big)\leq 2\varepsilon\,.
	\end{equation}
	By the arbitrariness of $
	\varepsilon$, $\tilde{\mu}=\mu^{\infty}$. Thus, every subsequential weak limit of $\Law(Y_t, \mu_t)$ is equal to $\mu^{\infty}\otimes \delta_{\mu^{\infty}}$ and 
	\begin{equation}
		{\bf E}_0=\RR^d\times\ca{P}_2(\RR^d).
	\end{equation}
\end{proof}

\begin{remark}\label{rem:global_dissipativity_E_0}
	Assume that $b$ and $\sigma$ are continuous on $\RR^d\times \ca{P}_2(\RR^d)$ and there exist constants $K>0$ and  $\lambda> \kappa\geq 0$ such that for any $(x, \mu), (y, \nu)\in \RR^d\times \ca{P}_2(\RR^d)$, 
	\begin{align}
		&|b(x, \mu)|+|\sigma (x, \mu)|\leq\,\, K(1+|x|+M_2(\mu)),\\[1mm]
		&2\langle b(x, \mu)-b(y, \nu), x-y\rangle+|\sigma(x, \mu)-\sigma(y, \nu)|^2\leq\, \kappa \ca{W}_2(\mu, \nu)^2-\lambda |x-y|^2\,\,.
	\end{align}
	Then ${\bf{E_0}}=\RR^d\times \ca{P}_2(\RR^d)$. Indeed, it follows from \cite[Theorem 6.1]{RenRocknerWang.2022.JDE1} that for all $ y\in\RR^d, \mu\in \ca{P}_2(\RR^d)$,
	\begin{align}
		&\ca{W}_2\big({\rm{Law}}(Y_t), \mu^{\infty})^2+\ca{W}_2\big(\mu_t, \mu^{\infty}) ^2\\
        \leq\,\,& \ca{W}_2 (\mu, \mu^{\infty})^2\big( 2{\rm{e}}^{-(\lambda-\kappa) t}-{\rm{e}}^{-\lambda t} \big)+\ca{W}_2(\delta_y, \mu^{\infty})^2{\rm{e}}^{-\lambda t} \,.\label{eq:exponential_ergodicity_P_t}
	\end{align}
	Consequently, 
	$$\Law(Y_t, \mu_t)=\Law(Y_t)\otimes \delta_{\mu_t}\xrightarrow[t\to\infty]{\ca{W}_2}\mu^{\infty}\otimes \delta_{\mu^{\infty}}\,.$$
	Then ${\bf{E_0}}=\RR^d\times \ca{P}_2(\RR^d)$,  since convergence in $\ca{W}_2$ implies weak convergence.
\end{remark}

\subsection{Weak synchronisation and generalised strong mixing for two-point McKean--Vlasov SDEs}
\begin{definition}[Generalised strong mixing]\label{matierial_def:strong_mixing}
	Let $\tilde{\ca{P}}(E)\subseteq \ca{P}(E)$. We say that a stochastic process $(X_t)_{t\geq 0}$ is well-posed in $\tilde{\ca{P}}(E)$, whenever $\Law(X_t)\in \tilde{\ca{P}}(E), t\geq 0,$ provided that $\Law(X_0)\in \tilde{\ca{P}}(E)$. Building on it, we say that it is \emph{generalised strongly mixing} in $\tilde{\ca{P}}(E)$ with respect to an invariant probability measure $\mu^{\infty}\in \tilde{\ca{P}}(E)$ if, for every random variable $X_0$ with $\Law(X_0)\in \tilde{\ca{P}}(E)$ and every bounded continuous function $f$, 
	\begin{equation}
		\int_{E}f (x)\, \d\Law(X_t)\to \int_{E} f(x)\,\d \mu^{\infty}(x)\,,\,\,\,\,\,\,\,\text{as}\,\,t\to\infty\,.
	\end{equation}
\end{definition}

\begin{remark}
	For a measure-independent SDE that is well-posed in
	$\widetilde{\ca P}(\RR^d)$, with
	$\delta_x\in\widetilde{\ca P}(\RR^d)$ for every $x\in\RR^d$,
	\cref{matierial_def:strong_mixing} is equivalent to convergence
	from every deterministic initial condition, since the associated
	Markov semigroup is linear in the initial distribution. In contrast,
	the law of a McKean--Vlasov SDE evolves nonlinearly and, in general,
	\begin{align}
		\Law(X_t^\mu)
		\neq
		\int_{\RR^d}\Law(X_t^{\delta_x})\,\d\mu( x).
	\end{align}
	Hence, convergence from every deterministic initial condition does
	not in general imply generalised strong mixing, which motivates the
	latter notion for McKean--Vlasov SDEs.
\end{remark}

We now investigate the relationship between weak synchronisation of coupled McKean–Vlasov SDEs and generalised strong mixing of two-point McKean–Vlasov SDEs.

Consider  the two-point 	McKean--Vlasov equation
\begin{subequations}
	\begin{align}
		\d Y_t^1=\,\,&b(Y_t^1,\Law(Y_t^1))\,\d t
		+
		\sigma(Y_t^1,\Law(Y_t^1))\,\d W_t,
		\label{eq:uniqueness_two_point_MVSDE_1}\\
		\d Y_t^2=\,\,&b(Y_t^2,\Law(Y_t^2))\,\d t
		+
		\sigma(Y_t^2,\Law(Y_t^2))\,\d W_t
		\label{eq:uniqueness_two_point_MVSDE_2}\,.
	\end{align}	
\end{subequations}

\begin{theorem}[Weak synchronisation and generalised strong mixing for McKean--Vlasov SDEs]\label{thm:weak_synchro_generalised_strong_mixing}
	Consider the coupled
	McKean--Vlasov system
	\begin{subequations}
		\begin{align}
			\d Y_t=\,\,&b(Y_t,\mu_t)\,\d t+\sigma(Y_t,\mu_t)\,\d W_t,\qquad Y_0=y\in \RR^d,
			\label{eq:prove_uniqueness_coupled_MVSDE}\\
			\partial_t\mu_t=\,\,&-\nabla\cdot\bigl(b(\cdot,\mu_t)\mu_t\bigr)
			+
			\frac{1}{2}
			D^2:\bigl(\sigma\sigma^T(\cdot,\mu_t)\mu_t\bigr),
			\qquad
			\mu_0=\mu\in\ca{P}_2(\RR^d).\label{eq:prove_uniqueness_coupled_MVSDE1}
		\end{align}
	\end{subequations}
Assume the following three conditions.
	\begin{enumerate}[(i)]
		\item The McKean--Vlasov PDE \eqref{eq:MVPDE3} satisfies \cref{assum:existence_of_mu_infty};
		\item The frozen SDE \eqref{eq:inftySDE} is strongly mixing with respect to $\mu^{\infty}$ and there exists a constant $C>0$  such that for  any  $\mu, \nu\in\ca{P}(\RR^d)$,   the solution to  \eqref{eq:inftySDE} satisfies 
		\begin{equation}
			\sup_{t\geq 0}d_w\big(Q_t^*\mu, Q_t^*\nu\big)\leq C d_w\big(\mu, \nu)\,.
		\end{equation}
		\item For any $x, y\in \RR^d$, $\mu, \nu\in \ca{P}_2(\RR^d)$, 
		\begin{align}
        %|b(x, \mu)-b(y, \mu)|+|\sigma(x, \mu)-\sigma(y, \mu)|\leq\,\,& F(M_2(\mu), 0)|x-y|\,,\\
        %|b(0, \mu)|+|\sigma(0, \mu)|\leq\,\,&F(M_2(\mu), 0)\,,\\
%2\langle x, b(x,\mu_t)\rangle
%+\Tr(\sigma\sigma^T(x,\mu_t))
%\le -c|x|^2+C,\qquad t\ge T_0\,,\\
			|b(y, \mu)-b(y, \nu)|+|\sigma\sigma^T(y, \mu)-\sigma\sigma^T(y, \nu)|\leq\,\,& F(M_2(\mu), M_2(\nu))(1+|y|)\ca{W}_2(\mu, \nu)\,,
		\end{align}
		where $F: \RR_+^2\to\RR_+$ is non-decreasing in each variable. Moreover, 
		for any $(y, \mu)\in \RR^d\times \ca{P}_2(\RR^d)$, the solution to the McKean--Vlasov system \eqref{eq: MVRDE3} satisfies 
	\begin{equation}
	    \sup_{t\geq 0}M_1\big(\Law(Y_t)\big)<\infty\,.
	\end{equation}  
    %\begin{equation}\label{eq:joint_tightness}
			%\big(\Law(Y_t, \mu_t)\big)_{t\geq 0}\,\,\,\,\,\,\,\,\text{is tight on}\,\,\RR^d\times\ca{P}_2(\RR^d)\,.
		%\end{equation}
	\end{enumerate}
If weak synchronisation occurs for \eqref{eq:prove_uniqueness_coupled_MVSDE}-\eqref{eq:prove_uniqueness_coupled_MVSDE1}, then the two-point McKean--Vlasov system \eqref{eq:uniqueness_two_point_MVSDE_1}-\eqref{eq:uniqueness_two_point_MVSDE_2}
	is generalised strongly mixing with respect to
	$\mu^{\infty}_{\Delta} 
	$ in $\ca{P}_2(\RR^d\times\RR^d)$. 
\end{theorem}

\begin{proof}
	Let $\eta\in\ca P_2(\RR^{2d})$ be an arbitrary initial joint
	distribution and set
	\begin{equation}
		\mu\coloneqq\pi_1\eta,\qquad \nu\coloneqq\pi_2\eta.
	\end{equation}
	Consider \eqref{eq:prove_uniqueness_coupled_MVSDE}--\eqref{eq:prove_uniqueness_coupled_MVSDE1}
	with initial conditions $(y_1,\mu)$ and $(y_2,\nu)$, respectively, where
	$(y_1,y_2)\in\RR^d\times\RR^d$.
	Let
	\begin{equation}
	\lambda_t\coloneqq
	\Law\big(\phi_t(\cdot,y_1,\mu),\phi_t(\cdot,y_2,\nu)\big),
	\end{equation}
	where $\phi$ is defined in \eqref{definedRDS}.
	It follows from \cref{prop:characterisation_of_E_0} that both marginals
	of $\lambda_t$ converge weakly to $\mu^\infty$. Thus, the family
	$(\lambda_t)_{t\geq0}$ is tight. Hence, by Prokhorov's theorem, for every
	sequence $t_n\to\infty$, there exist a subsequence $(t_{n_k})_{k\geq1}$
	and $\lambda\in\ca P(\RR^{2d})$ such that
	\begin{equation}
	\lambda_{t_{n_k}}\xrightarrow[k\to\infty]{w}\lambda.
	\end{equation}

	By weak synchronisation of the coupled McKean--Vlasov system
	\eqref{eq:prove_uniqueness_coupled_MVSDE}--\eqref{eq:prove_uniqueness_coupled_MVSDE1},
	\begin{equation}
		|\phi_t(\cdot,y_1,\mu)-\phi_t(\cdot,y_2,\nu)|
		\longrightarrow0
		\qquad\text{in probability}.
	\end{equation}
	Consider the bounded continuous function
	$
	h(x,y)=|x-y|\wedge1.
	$
	By the dominated convergence theorem,
	\begin{align}
		\EE h\big(
		\phi_{t_{n_k}}(\cdot,y_1,\mu),
		\phi_{t_{n_k}}(\cdot,y_2,\nu)
		\big)
		=
		\int_{\RR^{2d}}h(z_1,z_2)\,\d\lambda_{t_{n_k}}(z_1,z_2)
		\xlongrightarrow[k\to\infty]{}0.
	\end{align}
	Using the weak convergence of $\lambda_{t_{n_k}}$, we obtain
	\begin{equation}
		\int_{\RR^{2d}}h(x,y)\,\d\lambda(x,y)=0,
	\end{equation}
	which implies that $h=0$ $\lambda$-almost surely.
	Since $\{h=0\}=\Delta$, we have $\lambda(\Delta)=1$.
	Since both marginals of $\lambda$ are equal to $\mu^\infty$, it follows that
	\begin{equation}
	\lambda=\mu^\infty_\Delta.
	\end{equation}
	Since every subsequential limit is equal to $\mu^\infty_\Delta$, we conclude that
	\begin{align}
		\Law\big(
		\phi_t(\cdot,y_1,\mu),
		\phi_t(\cdot,y_2,\nu)
		\big)
		\xrightarrow[t\to\infty]{w}
		\mu^\infty_\Delta.
		\label{eq:joint_law_fixed_initial_data}
	\end{align}

	Conditioning on the initial data gives
	\begin{align}
		\Law(Y_t^1,Y_t^2)
		=
		\int_{\RR^{2d}}
		\Law\big(
		\phi_t(\cdot,y_1,\mu),
		\phi_t(\cdot,y_2,\nu)
		\big)
		\,\d\eta(y_1,y_2).
		\label{eq:linearity_of_joint_law}
	\end{align}
	For every $f\in C_b(\RR^{2d})$, by the dominated convergence theorem,
	\begin{align}
		\lim_{t\to\infty}\EE f(Y_t^1,Y_t^2)
		=\,\,&
		\int_{\RR^{2d}}
		\lim_{t\to\infty}
		\EE f\big(
		\phi_t(\cdot,y_1,\mu),
		\phi_t(\cdot,y_2,\nu)
		\big)
		\,\d\eta(y_1,y_2)\\
		=\,\,&
		\int_{\RR^{2d}}
		f(x,y)\,\d\mu^\infty_\Delta(x,y),
	\end{align}
	where \eqref{eq:joint_law_fixed_initial_data} is used in the last equality.
	Therefore,
	\begin{equation}
		\Law(Y_t^1,Y_t^2)
		\xrightarrow[t\to\infty]{w}
		\mu^\infty_\Delta.
		\label{eq:weak_limit_joint_law}
	\end{equation}
	%It remains to verify $\mu^{\infty}_{\Delta}$ is an invariant probability measure. Start
	%\eqref{eq:uniqueness_two_point_MVSDE_1}%--
	%\eqref{eq:uniqueness_two_point_MVSDE_2}
	%from
	%$(Y_0^1,Y_0^2)\sim\mu^{\infty}_{\Delta}.
	%$
	%Then
	%\begin{align}
	%	\Law(Y_0^1)=\Law(Y_0^2)=\mu^\infty.
	%\end{align}
	%Since $\mu^\infty$ is invariant for the one-point McKean--Vlasov SDE,
	%both marginals remain equal to $\mu^\infty$ for all $t\geq0$, and the two equations have identical coefficients.
	%Moreover, $Y_0^1=Y_0^2$ almost surely, together with pathwise uniqueness implies 
	%$
	%		Y_t^1=Y_t^2, \text{a.s. for all }t\geq0.
	%$
	%	Thus,
	%\begin{align}
	%\Law(Y_t^1,Y_t^2)=\mu^{\infty}_{\Delta},
	%	\qquad t\geq0,
	%	\end{align}
%and therefore $\mu^{\infty}_{\Delta}$ is an invariant
%probability measure of the two-point McKean--Vlasov system.
\end{proof}

\begin{remark}
Generalised strong mixing with respect to $ \mu^{\infty}_{\Delta}$ implies uniqueness of the invariant probability measure of the two-point motion in $\ca{P}_2(\RR^{2d})$.
Indeed, let $\Pi\in\ca P_2(\RR^{2d})$ be an invariant probability
measure and initiate the two-point motion from  $\Pi$. Invariance gives
\begin{equation}
	\Law(Y_t^1,Y_t^2)
	=
	\Pi,
	\qquad t\geq0.
\end{equation}
On the other hand, generalised strong mixing  in $\ca{P}_2(\RR^{2d})$ yields for any $(\Law(Y_0^1, Y_0^2))\in \ca P_2(\RR^{d}\times \RR^d)$,
\begin{align}
	\Law(Y_t^1,Y_t^2)
	\xrightarrow[t\to\infty]{w}
	\mu^{\infty}_{\Delta}.
\end{align}
Therefore,
$
\Pi=\mu^{\infty}_{\Delta},
$
which proves uniqueness.
\end{remark}

	\subsection{Application: Ensemble Kalman sampler}	
	In the following, we apply our main result \cref{thm:weak_synchronisation_MVSDE} to the coupled ensemble Kalman sampler. We first introduce the notation.   Let ${\sf{N}}_{\sf{m}, \sf{C}}$ be the normal distribution with mean ${\sf{m}}\in \RR^d$ and covariance ${\sf{C}}\in \bb{S}^d_{\succ 0}$, i.e.
	\begin{align}
		{\sf{N}}_{\sf{m}, \sf{C}}(y)=\frac{1}{(2\pi)^{d/2}({\rm{det}}\, \sf{C})^{1/2}}\exp\Big(-\frac{1}{2}|y-\sf{m}|_{\sf{C}}^2\Big)\,,
	\end{align}
Let $\bb{S}^d, \bb{S}_{\succ 0}^d, \bb{S}_{\succcurlyeq 0}^d $ denote the sets of symmetric,  symmetric positive-definite, and symmetric positive-semidefinite matrices in $\RR^{d\times d}$ respectively. We say that  $A\in SO(d)$ if $A^T A={\bf{I}}_{d\times d}$  and ${\rm{det}}(A)=1$. For a matrix $C\in\bb{S}^d_{\succeq 0}$, $C^{1/2}$ stands for the symmetric square root of $C$, and $\lambda_{\max}(C), \lambda_{\min}(C)$ denote the largest, smallest eigenvalues of $C$, respectively. For a vector $\xi\in \RR^d$ and $C\in \bb{S}_{\succ 0}^d$, $|\xi|_C^2$ stands for $\langle \xi, C^{-1}\xi\rangle$. We denote by $\|C\|_2$ the spectral norm of $C\in \RR^{d\times d}$ given by $|\lambda_{\max}(CC^T)|^{1/2}$, and by $|C|$ the Hilbert--Schmidt norm given by $|\sum_{i, j=1}^d C_{ij}^2|^{1/2}$.  Given $\mu\in \ca{P}_{2, +}(\RR^d)$ with mean $m(\mu)\coloneqq \int_{\RR^d} y\,\d \mu\in \RR^d$ and  covariance matrix ${\rm{Cov}}(\mu)\coloneqq\int_{\RR^d}(y-m(\mu))\otimes(y-m(\mu))\d \mu$, we define 
		\begin{align}
			T_{m(\mu), \sqrt{{\rm{Cov}}(\mu)}}:\,&\RR^d\to \RR^d\\
			&\,y\mapsto  \left(\sqrt{{\rm{Cov}}(\mu)}\right)^{-1}(y-m(\mu))\,.
		\end{align}
		%i.e. $T_{m(\mu), \sqrt{{\rm{Cov}}(\mu)}}^{-1}(y)=\sqrt{{\rm{Cov}}(\mu)}\,y+m(\mu)$. 
        We call $(T_{m(\mu), \sqrt{{\rm{Cov}}(\mu)}})\mu$ a normalization of $\mu$.
	
	Let
        \begin{align}
        \mathcal{P}_{2, +}(\RR^d)\coloneqq\,\,& \big\{\mu\in \mathcal{P}_2(\RR^d): {\rm{Cov}}(\mu)\succ 0\big\}\,,\\
	\widetilde{\ca P}_{2,+}(\RR^{2d})
	\coloneqq\,\,&
	\Big\{
	\eta\in\ca P_2(\RR^{2d}):
	\pi_1\eta,\pi_2\eta\in\ca P_{2,+}(\RR^d)
	\Big\}.
\end{align}

		Consider the coupled ensemble Kalman sampler system
		\begin{subequations}\label{eksdecoupled}
			\begin{align}
				\d Y_t= & \, -{\rm{Cov}}(\mu_t)\nabla V(Y_t)\,\d t+\sqrt{2{\rm{Cov}}(\mu_t)}\,\d W_t\, , \quad Y_0=y\in \RR^d\, , \label{eksdecoupledSDE}\\
				\partial_t \mu_t =&\,D^2:\left({\rm {Cov}}(\mu_t)\mu_t\right) +\nabla\cdot({\rm {Cov}}(\mu_t) \nabla V \mu_t)\, ,\;\;\;\;\;\mu_0=\mu\in \mathcal{P}_{2, +}(\RR^d) \,. \label{eksdecoupledPDE}
			\end{align}
		\end{subequations}
In \cite{gess2025random}, we constructed an RDS  $\varphi^{\rm{EKS}}$ associated with \eqref{eksdecoupledSDE}-- \eqref{eksdecoupledPDE}
	 via rough path theory.	
	\begin{theorem}
		Let 
        \begin{equation}
		 V(y)\coloneqq \frac{1}{2}|y-{\sf{m}}|_{{\sf{C}}}^2,\,\,\,\,\,{\text{with mean}}\,\,{\sf{m}}\in \RR^d \,\,{\text{and covariance}}\,\,{\sf{C}}\in\,\bb{S}_{\succ 0}^d\,. 
         \end{equation}
         The RDS $\varphi^{\rm{EKS}}$ associated with \eqref{eksdecoupled} exhibits conditional weak synchronisation on  $\RR^d\times \ca{P}_{2, +}(\RR^d)$ and its %minimal 
         weak point $\RR^d\times \ca{P}_{2, +}(\RR^d)$-attractor is given by $$\Big\{{\sf{m}}+\int_{-\infty}^0 \sqrt{2 {\sf{C}}}\, {\rm{e}}^{s}  \,\d W_s(\cdot)\Big\}\times \{\mathsf{N}_{{\sf{m}}, {\sf{C}}}\}.$$
    In particular, the two-point ensemble Kalman sampler is generalised strongly mixing in $\widetilde{\ca P}_{2,+}(\RR^{2d})$ with respect to $\mu^\infty_\Delta$.
	\end{theorem}
	\begin{proof}
		First of all, it follows from \cite[Theorem 1.20]{BurgerEtAl2025} that, for any $\mu_0\in \ca{P}_{2, +}(\RR^d)$, the solution to \eqref{eksdecoupledPDE} admits an equilibrium $\mu^{\infty}={\sf{N}}_{{\sf{m}}, {\sf{C}}}$  and  satisfies
		\begin{align}
			\begin{split}
				\mathcal{W}_2(\mu_t, {\sf{N}}_{{\sf{m}}, {\sf{C}}})
				\leq\,\,& {\rm{e}}^{-t}\bar{\lambda}({\sf{C}}, {\rm{Cov}}(\mu_0))\Big[\inf_{R\in SO(d)}\mathcal{W}_2(R_\#\bar{\mu}_0, {\sf{N}}_{0, {\rm{Id}}})^2+|m(\mu_0)-{\sf{m}}|_{{\rm{Cov}}(\mu_0)}^2\\
				&\,\,\,\,\,\,\,\,\,\,\,\,\,\,\,\,\,\,\,\,\,\,\,\,\,\,\,\,\,\,\,\,\,\,\,\,\,\,\,\,\,\,\,\,\,\,\,\,\,\,\,\,\,\,\,\,\,\,+|{\rm{Id}}-({\sf{C}}^{1/2}{\rm{Cov}}(\mu_0)^{-1}{\sf{C}}^{1/2})^{1/2}|^2\Big]^{1/2}\,,\label{es:ergodicityEKS}
			\end{split}
		\end{align}
		 where $\bar{\mu}_0$ is the normalization of $\mu_0$ and  $\bar{\lambda}({\sf{C}}, {\rm{Cov}}(\mu_0))\coloneqq \|{\sf{C}}\|_2\max \Big\{1, \|{\sf{C}}^{-1/2}{\rm{Cov}}(\mu_0){\sf{C}}^{-1/2}\|_2\Big\}$. For notational simplicity, we set $\lambda_0(\mu_0)\coloneqq \bar{\lambda}({\sf{C}}, {\rm{Cov}}(\mu_0))\Big[\inf_{R\in SO(d)}\mathcal{W}_2(R_\#\bar{\mu}_0, {\sf{N}}_{0, {\rm{Id}}})^2+|m(\mu_0)-{\sf{m}}|_{{\rm{Cov}}(\mu_0)}^2+|{\rm{Id}}-({\sf{C}}^{1/2}{\rm{Cov}}(\mu_0)^{-1}{\sf{C}}^{1/2})^{1/2}|^2\Big]^{1/2}$.
         
Substituting $\mu^{\infty}={\sf{N}}_{{\sf{m}}, {\sf{C}}}$ into \eqref{eksdecoupledSDE}, we have the frozen SDE 
		\begin{equation}
			\d Y^{\infty}_t=-(Y^{\infty}_t-{\sf{m}})\d t+\sqrt{2 {\sf{C}}}\,\d W_t,\,\,\,\,\,\,Y^{\infty}_0=y\,,
		\end{equation}
and its corresponding RDS 
		\begin{equation}
			\varphi^{\rm{Fro}}(t, \omega)y={\sf{m}}+{\rm{e}}^{-t}(y-{\sf{m}})+\int_0^t\sqrt{2 {\sf{C}}} {\rm{e}}^{-(t-s)}\d W_s(\omega)\,.
		\end{equation}
		Applying \cite[Example 3.8, Theorem 3.12]{FlandoliGessScheutzow.2017.PTRF511}, we know that $\varphi^{\rm{Fro}}$  exhibits weak synchronisation with the minimal weak point attractor ${\rm{supp}}(\mu^{\rm{Fro}}_{\omega})$, $\PP$-a.s. where
		\begin{equation}
			\mu_{\omega}^{\rm{Fro}}=\lim_{n\to\infty}
			\varphi^{\rm{Fro}}(n, \theta_{-n}\omega)\mathsf{N}_{{\sf{m}}, {\sf{C}}}=\delta_{{\sf{m}}+\int_{-\infty}^0 \sqrt{2 {\sf{C}}} \,{\rm{e}}^{s}  \,\d W_s(\omega)}\,,
		\end{equation}
		where for the second identity, we used the fact that for any $y\in\RR^d$,
		\begin{align}
			\lim_{n\to \infty}	\varphi^{\rm{Fro}}(n, \theta_{-n}\omega)y=\,\,&	\lim_{n\to \infty}\left({\sf{m}}+{\rm{e}}^{-n}(y-{\sf{m}})+\int_{-n}^0\sqrt{2 {\sf{C}}}\, {\rm{e}}^{s}\d W_s(\omega)\right)\\
            =\,\,&{\sf{m}}+\int_{-\infty}^0 \sqrt{2 {\sf{C}}}\,{\rm{e}}^{s}  \,\d W_s(\omega)\,.
		\end{align}
	According to \cref{thm:weak_synchronisation_MVSDE}, it suffices to show that  ${\bf{E}_0}=\RR^d\times \ca{P}_{2, +}(\RR^d)$. To prove it, we verify the conditions in \cref{prop:characterisation_of_E_0}. 	
	To verify \eqref{es:exponential_ergodicity_frozen}, let $d_w$ denote
the bounded-Lipschitz metric, which metrises the weak topology on
$\ca{P}(\RR^d)$. For any $f$ with
$\|f\|_\infty\leq1$ and $\operatorname{Lip}(f)\leq1$, define
\begin{align}
	g_t(x)
	\coloneqq
	\EE f\big(
	{\sf m}+{\rm e}^{-t}(x-{\sf m})
	+\int_0^t\sqrt{2{\sf C}}\,{\rm e}^{-(t-s)}\,\d W_s
	\big).
\end{align}
Then $\|g_t\|_\infty\leq1$ and
$\operatorname{Lip}(g_t)\leq {\rm e}^{-t}\leq1$. Hence,
\begin{equation}
	\big|\EE f(Y_t^{\infty,\mu})-\EE f(Y_t^{\infty,\nu})\big|
	=
	|\mu(g_t)-\nu(g_t)|
	\leq d_w(\mu,\nu).
\end{equation}
Taking the supremum over $f$ yields
\begin{equation}
	d_w\big(\Law(Y_t^{\infty,\mu}),
	\Law(Y_t^{\infty,\nu})\big)
	\leq d_w(\mu,\nu),
	\qquad t\geq0.
\end{equation}
Condition \eqref{eq:E_0_local_Lipschitz} follows directly from the estimates,  for all $\mu,\nu\in\mathcal{P}_2(\RR^d),$
		\begin{align}
			|\sqrt{{\rm{Cov}}(\mu)}-\sqrt{{\rm{Cov}}(\nu)}\,|\leq &\,\,\sqrt{2}\mathcal{W}_2(\mu,\nu)\,,\\[1mm]
			|{\rm{Cov}}(\mu)-{\rm{Cov}}(\nu)|\leq&\,\,2\big(M_2(\mu)^{1/2}+M_2(\nu)^{1/2}\big)\mathcal{W}_2(\mu,\nu)\,.\label{eq:Lip_of_cov}
		\end{align}
        
It remains to show that for any $(y, \mu_0)\in\RR^d\times \ca{P}_{2, +}(\RR^d)$, $(\Law(Y_t, \mu_t))_{t\geq 0}$ is tight. Due to the exponential ergodicity, and the time-continuity in $(\ca{P}_2(\RR^d), \ca{W}_2)$ \footnote{In \cite[Theorem 3.4]{gess2025random}, we proved the time continuity of $(\mu_t)_{t\geq 0}$ with respect to the tailor-made metric $\ca{D}_2$. The time continuity of $(\mu_t)_{t\geq 0}$ with respect to  $\ca{W}_2$ then follows from the equivalence of the topologies induced by $\ca{D}_2$ and $\ca{W}_2$; see \cite[Lemma 2.2]{gess2025random}.}, we know that the set $\{\mu_t:t\geq0\}$ is relatively compact in
$(\ca{P}_2(\RR^d),\ca{W}_2)$ and $(\delta_{\mu_t})_{t\geq0}$ is tight in $\ca{P}(\ca{P}_2(\RR^d))$. Now, we show that $\sup_{t\geq 0}\EE|Y_t|^2<\infty$. Considering 		\begin{align}
	M_2(\mu^\infty)
	=\,&\int_{\RR^d}|y|^2\,\d\mu^\infty(y)\\
	=\,&\int_{\RR^d}|y-{\sf m}|^2\,\d\mu^\infty(y)+|{\sf m}|^2\\
	=\,&{\rm {Tr}}({\sf C})+|{\sf m}|^2\,,
	\label{eq:momentestimatinfinityEKS}
\end{align}
        and 
\begin{equation}
			M_2(\mu_t)\leq\, 2\mathcal{W}_2(\mu_t, \mu^{\infty})^2+2M_2(\mu^{\infty})
			\leq\,2\lambda_0^2(\mu_0){\rm{e}}^{-2t}+2|{\sf{m}}|^2+2{\rm{Tr}}({\sf{C}})\,,
		\end{equation} together with \eqref{eq:Lip_of_cov}, we obtain 
\begin{align}
|{\rm{Cov}}(\mu_t)-{\sf{C}}|\leq\,&2\big(\sqrt{2\lambda_0(\mu_0)^2{\rm{e}}^{-2t}+2|\sf{m}|^2+2{\rm{Tr}}(\sf{C})}+\sqrt{2|\sf{m}|^2+2{\rm{Tr}}(\sf{C})}\big)\lambda_0(\mu_0) {\rm{e}}^{-t}\,,\\[1mm]
    |{\rm{Cov}}(\mu_t)|\leq\,&|{\sf{C}}|+2\big(\sqrt{2\lambda_0(\mu_0)^2{\rm{e}}^{-2t}+2|\sf{m}|^2+2{\rm{Tr}}(\sf{C})}+\sqrt{2|\sf{m}|^2+2{\rm{Tr}}(\sf{C})}\big)\lambda_0(\mu_0) {\rm{e}}^{-t}\,.
\end{align}
Using It\^{o}'s formula and taking expectations, we obtain
\begin{align}
	\frac{\d}{\d t}\EE|Y_t-{\sf{m}}|^2
	=\,&-2\EE\left\langle
	{\rm{Cov}}(\mu_t){\sf{C}}^{-1}(Y_t-{\sf{m}}),
	Y_t-{\sf{m}}
	\right\rangle
	+2{\rm{Tr}}({\rm{Cov}}(\mu_t))\\
	=\,&-2\EE|Y_t-{\sf{m}}|^2
	-2\EE\left\langle
	({\rm{Cov}}(\mu_t)-{\sf{C}}){\sf{C}}^{-1}(Y_t-{\sf{m}}),
	Y_t-{\sf{m}}
	\right\rangle\\
	&+2{\rm{Tr}}({\rm{Cov}}(\mu_t))\\
	\leq\,&
	\left(-2+2|{\sf{C}}^{-1}|
	|{\rm{Cov}}(\mu_t)-{\sf{C}}|\right)
	\EE|Y_t-{\sf{m}}|^2
	+2{\rm{Tr}}({\rm{Cov}}(\mu_t)).
\end{align}
By the above estimates, there exists a constant $K>0$, depending only on
${\sf{m}}$, ${\sf{C}}$, and $\lambda_0(\mu_0)$, such that
\begin{align}
	|{\rm{Cov}}(\mu_t)-{\sf{C}}|
	\leq K{\rm{e}}^{-t},
	\qquad
	{\rm{Tr}}({\rm{Cov}}(\mu_t))\leq K,
	\qquad t\geq0.
\end{align}
Hence, setting
$
	h(t)\coloneqq\EE|Y_t-{\sf{m}}|^2,
$
we obtain
\begin{equation}
    h'(t)\leq \left(-2+K{\rm{e}}^{-t}\right)h(t)+K.
\end{equation}
Choose $T>0$ sufficiently large such that $K{\rm{e}}^{-t}\leq1$ for all
$t\geq T$. Then
$$h'(t)\leq K h(t)+K,\,\,\,\,t\in[0,T],\,\,\,\,\,\,\,\,\,\,\,\,h'(t)\leq-h(t)+K,\,\,\,\, t\geq T.$$
Hence, by Gr\"{o}nwall's inequality,
\begin{equation}
	h(t)
	\leq {\rm e}^{Kt}h(0)
	+K\int_0^t {\rm e}^{K(t-s)}\,\d s,
	\,\, 0\leq t\leq T,
\end{equation}
and \begin{equation}
	h(t)\leq {\rm{e}}^{-(t-T)}h(T)+K,\qquad t\geq T.
\end{equation} 
Consequently,
\begin{equation}
	\sup_{t\geq0}\EE|Y_t|^2<\infty,
\end{equation}
%and hence $(\Law(Y_t))_{t\geq0}$ is tight in
%$\ca{P}(\RR^d)$. Together with the tightness of
%$(\delta_{\mu_t})_{t\geq0}$ in
%$\ca{P}(\ca{P}_2(\RR^d))$, this implies that
%\begin{equation}
%	\big(\Law(Y_t,\mu_t)\big)_{t\geq0}
%\end{equation}
%is tight in
%$\ca{P}(\RR^d\times\ca{P}_2(\RR^d))$.
which completes the proof.
\end{proof}

\begin{remark}
Although \cref{thm:weak_synchronisation_MVSDE} and \cref{prop:characterisation_of_E_0} are stated on $\ca{P}_2(\RR^d)$, the same results hold on any subspace $\tilde{\ca{P}}_2(\RR^d)\subseteq\ca{P}_2(\RR^d)$ on which the McKean–Vlasov PDE satisfies \cref{assum:existence_of_mu_infty}. Indeed, the proofs only rely on the separability of the underlying space of probability measures, which is inherited by every subspace of the separable metric space $\ca{P}_2(\RR^d)$.
\end{remark}

    \section{Acknowledgements}

  BG acknowledges support by the Max Planck Society through the Research Group “Stochastic Analysis in the Sciences (SAiS)” and the DFG CRC/TRR 388 “Rough Analysis, Stochastic Dynamics and Related Fields”, Project A11.     RSG is partially funded by the Deutsche Forschungsgemeinschaft (DFG, German Research Foundation) - SPP 2410 Hyperbolic Balance Laws in Fluid Mechanics: Complexity, Scales, Randomness (CoScaRa, DFG Project no. 500873292).  SH is funded by the European Union (ERC, FluCo, grant agreement No. 101088488). Views and opinions expressed are, however,  those of the author(s) only and do not necessarily reflect those of the European Union or of the European Research Council. Neither the European Union nor the granting authority can be held responsible for them.

%\appendix

%\endappendix

\bibliographystyle{Martin}
\bibliography{zong0328}

@article{BurgerEtAl2025,
  author  = {Burger, Martin and Erbar, Matthias and Hoffmann, Franca and Matthes, Daniel and Schlichting, Andr\'e},
  title   = {Covariance-Modulated Optimal Transport and Gradient Flows},
  journal = {Archive for Rational Mechanics and Analysis},
  volume  = {249},
  pages   = {7},
  year    = {2025},
  doi     = {10.1007/s00205-024-02065-w}
}

@misc{LelievrePavliotisRobinSantetStoltz2025,
  author        = {Leli{\`e}vre, Tony and Pavliotis, Grigorios A. and Robin, Gabriel and Santet, Romain and Stoltz, Gabriel},
  title         = {Optimizing the Diffusion Coefficient of Overdamped Langevin Dynamics},
  year          = {2025},
  eprint        = {2404.12087},
  archivePrefix = {arXiv},
  primaryClass  = {math.NA}
}

@article {DelgadinoGvalaniPavliotisSmith2023CMP,
	AUTHOR = {Delgadino, Mat\'{\i}as G. and Gvalani, Rishabh S. and
	Pavliotis, Grigorios A. and Smith, Scott A.},
	TITLE = {Phase transitions, logarithmic {S}obolev inequalities, and
	uniform-in-time propagation of chaos for weakly interacting
	diffusions},
	JOURNAL = {Comm. Math. Phys.},
	FJOURNAL = {Communications in Mathematical Physics},
	VOLUME = {401},
	YEAR = {2023},
	NUMBER = {1},
	PAGES = {275--323},
	ISSN = {0010-3616,1432-0916},
	MRCLASS = {60K35 (82B26)},
	MRNUMBER = {4604897},
	MRREVIEWER = {Adrian\ Muntean},
	DOI = {10.1007/s00220-023-04659-z},
	URL = {https://doi.org/10.1007/s00220-023-04659-z},
}

@article{neamctu2025negativity,
	title={On the negativity of the top Lyapunov exponent for stochastic differential equations driven by fractional Brownian motion},
	author={Neam{\c{t}}u, Alexandra Blessing and Varzaneh, Mazyar Ghani},
	journal={arXiv preprint arXiv:2510.11531},
	year={2025}
}

@article {ChueshovSchmalfu2010JMP,
	AUTHOR = {Chueshov, Igor and Schmalfu\ss , Bj\"{o}rn},
	TITLE = {Master-slave synchronization and invariant manifolds for
	coupled stochastic systems},
	JOURNAL = {J. Math. Phys.},
	FJOURNAL = {Journal of Mathematical Physics},
	VOLUME = {51},
	YEAR = {2010},
	NUMBER = {10},
	PAGES = {102702, 23},
	ISSN = {0022-2488,1089-7658},
	MRCLASS = {60H30 (60H25 82C31)},
	MRNUMBER = {2761307},
	DOI = {10.1063/1.3493646},
	URL = {https://doi.org/10.1063/1.3493646},
}

@article {CaraballoRobinson2004SCL,
	AUTHOR = {Caraballo, Tom\'{a}s and Robinson, James C.},
	TITLE = {Stabilisation of linear {PDE}s by {S}tratonovich noise},
	JOURNAL = {Systems Control Lett.},
	FJOURNAL = {Systems \& Control Letters},
	VOLUME = {53},
	YEAR = {2004},
	NUMBER = {1},
	PAGES = {41--50},
	ISSN = {0167-6911,1872-7956},
	MRCLASS = {60H15 (35R60 93D21)},
	MRNUMBER = {2077187},
	MRREVIEWER = {Sergey\ V.\ Lototsky},
	DOI = {10.1016/j.sysconle.2004.02.020},
	URL = {https://doi.org/10.1016/j.sysconle.2004.02.020},
}

@article {MartinelliScoppola1988CMP,
	AUTHOR = {Martinelli, Fabio and Scoppola, Elisabetta},
	TITLE = {Small random perturbations of dynamical systems: exponential
	loss of memory of the initial condition},
	JOURNAL = {Comm. Math. Phys.},
	FJOURNAL = {Communications in Mathematical Physics},
	VOLUME = {120},
	YEAR = {1988},
	NUMBER = {1},
	PAGES = {25--69},
	ISSN = {0010-3616,1432-0916},
	MRCLASS = {58F30 (35R60 58F13 60J05)},
	MRNUMBER = {972542},
	MRREVIEWER = {S.\ Rajasekar},
	URL = {http://projecteuclid.org/euclid.cmp/1104177655},
}

@article{blessing2026synchronisation,
	title={Synchronization by noise for stochastic differential equations driven by fractional Brownian motion},
	author={Blessing, Alexandra and Varzaneh, Mazyar Ghani},
	journal={arXiv preprint arXiv:2603.12774},
	year={2026}
}

@article{liu2025synchronisation,
	title={Synchronization by degenerate noise},
	author={Liu, Xianming and Sun, Xu},
	journal={arXiv preprint arXiv:2512.18278},
	year={2025}
}

@article {Crauel1991Markovmeasures,
	AUTHOR = {Crauel, Hans},
	TITLE = {Markov measures for random dynamical systems},
	JOURNAL = {Stochastics Stochastics Rep.},
	FJOURNAL = {Stochastics and Stochastics Reports},
	VOLUME = {37},
	YEAR = {1991},
	NUMBER = {3},
	PAGES = {153--173},
	ISSN = {1045-1129},
	MRCLASS = {60J05 (58F11 60H10)},
	MRNUMBER = {1148346},
	MRREVIEWER = {A.\ N.\ Al-Hussaini},
	DOI = {10.1080/17442509108833733},
	URL = {https://doi.org/10.1080/17442509108833733},
}

@book {Crauel2002Randommeasures,
	AUTHOR = {Crauel, Hans},
	TITLE = {Random probability measures on {P}olish spaces},
	SERIES = {Stochastics Monographs},
	VOLUME = {11},
	PUBLISHER = {Taylor \& Francis, London},
	YEAR = {2002},
	PAGES = {xvi+118},
	ISBN = {0-415-27387-0},
	MRCLASS = {60-02 (60B05 60G57)},
	MRNUMBER = {1993844},
	MRREVIEWER = {Natesan\ Renganathan},
}

@article {WFYSPA2018,
	AUTHOR = {Wang, Feng-Yu},
	TITLE = {Distribution dependent {SDE}s for {L}andau type equations},
	JOURNAL = {Stochastic Process. Appl.},
	FJOURNAL = {Stochastic Processes and their Applications},
	VOLUME = {128},
	YEAR = {2018},
	NUMBER = {2},
	PAGES = {595--621},
	ISSN = {0304-4149,1879-209X},
	MRCLASS = {60J75 (47G20 60G52 60H10)},
	MRNUMBER = {3739509},
	MRREVIEWER = {Hong\ Zhang},
	DOI = {10.1016/j.spa.2017.05.006},
	URL = {https://doi.org/10.1016/j.spa.2017.05.006},
}

@article {BedrossianBlumenthalSam2022Inven,
	AUTHOR = {Bedrossian, Jacob and Blumenthal, Alex and Punshon-Smith, Sam},
	TITLE = {A regularity method for lower bounds on the {L}yapunov
	exponent for stochastic differential equations},
	JOURNAL = {Invent. Math.},
	FJOURNAL = {Inventiones Mathematicae},
	VOLUME = {227},
	YEAR = {2022},
	NUMBER = {2},
	PAGES = {429--516},
	ISSN = {0020-9910,1432-1297},
	MRCLASS = {60H15 (35B65 35H10 37D25 37H15 58J65)},
	MRNUMBER = {4372219},
	MRREVIEWER = {Xiaobin\ Sun},
	DOI = {10.1007/s00222-021-01069-7},
	URL = {https://doi.org/10.1007/s00222-021-01069-7},
}

@article {BedrossianBlumenthalJEMS,
	AUTHOR = {Bedrossian, Jacob and Blumenthal, Alex and Punshon-Smith, Sam},
	TITLE = {Lagrangian chaos and scalar advection in stochastic fluid
	mechanics},
	JOURNAL = {J. Eur. Math. Soc. (JEMS)},
	FJOURNAL = {Journal of the European Mathematical Society (JEMS)},
	VOLUME = {24},
	YEAR = {2022},
	NUMBER = {6},
	PAGES = {1893--1990},
	ISSN = {1435-9855,1435-9863},
	MRCLASS = {60H15 (35Q30 37D25 37L55 76F20)},
	MRNUMBER = {4404792},
	MRREVIEWER = {Xuhui\ Peng},
	DOI = {10.4171/jems/1140},
	URL = {https://doi.org/10.4171/jems/1140},
}

@article {BlumenthalEngelAlexandra2023PTRF,
	AUTHOR = {Blumenthal, Alex and Engel, Maximilian and Neam\c{t}u,
	Alexandra},
	TITLE = {On the pitchfork bifurcation for the {C}hafee-{I}nfante
	equation with additive noise},
	JOURNAL = {Probab. Theory Related Fields},
	FJOURNAL = {Probability Theory and Related Fields},
	VOLUME = {187},
	YEAR = {2023},
	NUMBER = {3-4},
	PAGES = {603--627},
	ISSN = {0178-8051,1432-2064},
	MRCLASS = {60H15 (37H20 37L55 60H50)},
	MRNUMBER = {4664581},
	MRREVIEWER = {Peter\ E.\ Kloeden},
	DOI = {10.1007/s00440-023-01235-3},
	URL = {https://doi.org/10.1007/s00440-023-01235-3},
}

@incollection {ScheutzowVorkastner2018Springer,
	AUTHOR = {Scheutzow, Michael and Vorkastner, Isabell},
	TITLE = {Synchronization, {L}yapunov exponents and stable manifolds for
	random dynamical systems},
	BOOKTITLE = {Stochastic partial differential equations and related fields},
	SERIES = {Springer Proc. Math. Stat.},
	VOLUME = {229},
	PAGES = {359--366},
	PUBLISHER = {Springer, Cham},
	YEAR = {2018},
	ISBN = {978-3-319-74929-7; 978-3-319-74928-0},
	MRCLASS = {37H10 (37D25)},
	MRNUMBER = {3828181},
	DOI = {10.1007/978-3-319-74929-7_2},
	URL = {https://doi.org/10.1007/978-3-319-74929-7_2},
}

@article {Newman2018ETDS,
	AUTHOR = {Newman, Julian},
	TITLE = {Necessary and sufficient conditions for stable synchronization
	in random dynamical systems},
	JOURNAL = {Ergodic Theory Dynam. Systems},
	FJOURNAL = {Ergodic Theory and Dynamical Systems},
	VOLUME = {38},
	YEAR = {2018},
	NUMBER = {5},
	PAGES = {1857--1875},
	ISSN = {0143-3857,1469-4417},
	MRCLASS = {37H10 (60J35)},
	MRNUMBER = {3820004},
	MRREVIEWER = {Maximilian\ Engel},
	DOI = {10.1017/etds.2016.109},
	URL = {https://doi.org/10.1017/etds.2016.109},
}

@incollection {Baxendale1991,
	AUTHOR = {Baxendale, Peter H.},
	TITLE = {Statistical equilibrium and two-point motion for a stochastic
	flow of diffeomorphisms},
	BOOKTITLE = {Spatial stochastic processes},
	SERIES = {Progr. Probab.},
	VOLUME = {19},
	PAGES = {189--218},
	PUBLISHER = {Birkh\"{a}user Boston, Boston, MA},
	YEAR = {1991},
	ISBN = {0-8176-3477-0},
	MRCLASS = {60H20 (34D08 58G32)},
	MRNUMBER = {1144097},
	MRREVIEWER = {R.\ W. R. Darling},
}

@article {ArnoldCrauelWihstutz1983,
	AUTHOR = {Arnold, L. and Crauel, H. and Wihstutz, V.},
	TITLE = {Stabilization of linear systems by noise},
	JOURNAL = {SIAM J. Control Optim.},
	FJOURNAL = {SIAM Journal on Control and Optimization},
	VOLUME = {21},
	YEAR = {1983},
	NUMBER = {3},
	PAGES = {451--461},
	ISSN = {0363-0129},
	MRCLASS = {93E15 (34F05)},
	MRNUMBER = {696907},
	MRREVIEWER = {O.\ K.\ Zakusilo},
	DOI = {10.1137/0321027},
	URL = {https://doi.org/10.1137/0321027},
}

@article {Tearne2008PTRF,
	AUTHOR = {Tearne, Oliver M.},
	TITLE = {Collapse of attractors for {ODE}s under small random
	perturbations},
	JOURNAL = {Probab. Theory Related Fields},
	FJOURNAL = {Probability Theory and Related Fields},
	VOLUME = {141},
	YEAR = {2008},
	NUMBER = {1-2},
	PAGES = {1--18},
	ISSN = {0178-8051,1432-2064},
	MRCLASS = {60H10 (34D45 34F05 37H10)},
	MRNUMBER = {2372963},
	MRREVIEWER = {Bj\"{o}rn\ Schmalfuss},
	DOI = {10.1007/s00440-006-0051-0},
	URL = {https://doi.org/10.1007/s00440-006-0051-0},
}

@article {ButkovskyScheutzowCMP2020,
	AUTHOR = {Butkovsky, Oleg and Scheutzow, Michael},
	TITLE = {Couplings via comparison principle and exponential ergodicity
	of {SPDE}s in the hypoelliptic setting},
	JOURNAL = {Comm. Math. Phys.},
	FJOURNAL = {Communications in Mathematical Physics},
	VOLUME = {379},
	YEAR = {2020},
	NUMBER = {3},
	PAGES = {1001--1034},
	ISSN = {0010-3616,1432-0916},
	MRCLASS = {60J35 (47D07 60H15)},
	MRNUMBER = {4163359},
	MRREVIEWER = {Sonja\ Cox},
	DOI = {10.1007/s00220-020-03834-w},
	URL = {https://doi.org/10.1007/s00220-020-03834-w},
}

@article {GessJDDE2013,
	AUTHOR = {Gess, Benjamin},
	TITLE = {Random attractors for degenerate stochastic partial
	differential equations},
	JOURNAL = {J. Dynam. Differential Equations},
	FJOURNAL = {Journal of Dynamics and Differential Equations},
	VOLUME = {25},
	YEAR = {2013},
	NUMBER = {1},
	PAGES = {121--157},
	ISSN = {1040-7294,1572-9222},
	MRCLASS = {37Lxx (35B41 35R60 60H15 76S05)},
	MRNUMBER = {3027636},
	DOI = {10.1007/s10884-013-9294-5},
	URL = {https://doi.org/10.1007/s10884-013-9294-5},
}

@article {CaraballoPAMS2007,
	AUTHOR = {Caraballo, Tom\'{a}s and Crauel, Hans and Langa, Jos\'{e} A.
	and Robinson, James C.},
	TITLE = {The effect of noise on the {C}hafee-{I}nfante equation: a
	nonlinear case study},
	JOURNAL = {Proc. Amer. Math. Soc.},
	FJOURNAL = {Proceedings of the American Mathematical Society},
	VOLUME = {135},
	YEAR = {2007},
	NUMBER = {2},
	PAGES = {373--382},
	ISSN = {0002-9939,1088-6826},
	MRCLASS = {60H15 (35K20 35K55 35R60)},
	MRNUMBER = {2255283},
	MRREVIEWER = {Sandra\ Cerrai},
	DOI = {10.1090/S0002-9939-06-08593-5},
	URL = {https://doi.org/10.1090/S0002-9939-06-08593-5},
}

@article {CrauelJDDE1998,
	AUTHOR = {Crauel, Hans and Flandoli, Franco},
	TITLE = {Additive noise destroys a pitchfork bifurcation},
	JOURNAL = {J. Dynam. Differential Equations},
	FJOURNAL = {Journal of Dynamics and Differential Equations},
	VOLUME = {10},
	YEAR = {1998},
	NUMBER = {2},
	PAGES = {259--274},
	ISSN = {1040-7294,1572-9222},
	MRCLASS = {58F12 (34C23 34F05 60H10)},
	MRNUMBER = {1623013},
	MRREVIEWER = {Bohdan\ Maslowski},
	DOI = {10.1023/A:1022665916629},
	URL = {https://doi.org/10.1023/A:1022665916629},
}

@book {ChueshovMonotonelecturenotes,
	AUTHOR = {Chueshov, Igor},
	TITLE = {Monotone random systems theory and applications},
	SERIES = {Lecture Notes in Mathematics},
	VOLUME = {1779},
	PUBLISHER = {Springer-Verlag, Berlin},
	YEAR = {2002},
	PAGES = {viii+234},
	ISBN = {3-540-43246-9},
	MRCLASS = {37H10 (34F05 37C65 37G35 37L30 60H10 82C05)},
	MRNUMBER = {1902500},
	MRREVIEWER = {Bj\"{o}rn\ Schmalfuss},
	DOI = {10.1007/b83277},
	URL = {https://doi.org/10.1007/b83277},
}

@article {ArnoldChueshovDSC,
	AUTHOR = {Arnold, Ludwig and Chueshov, Igor},
	TITLE = {Order-preserving random dynamical systems: equilibria,
	attractors, applications},
	JOURNAL = {Dynam. Stability Systems},
	FJOURNAL = {Dynamics and Stability of Systems. An International Journal},
	VOLUME = {13},
	YEAR = {1998},
	NUMBER = {3},
	PAGES = {265--280},
	ISSN = {0268-1110},
	MRCLASS = {58F12 (28D99 34C35 34F05 35R60 60H25)},
	MRNUMBER = {1645467},
	MRREVIEWER = {Pierre\ A.\ Vuillermot},
	DOI = {10.1080/02681119808806264},
	URL = {https://doi.org/10.1080/02681119808806264},
}

@article{engel2026random,
	title={Random Quadratic Form on a Sphere: Synchronization by Common Noise},
	author={Engel, Maximilian and Shalova, Anna},
	journal={arXiv preprint arXiv:2603.06187},
	year={2026}
}

@article{gess2025random,
	title={Random dynamical systems for McKean--Vlasov SDEs via rough path theory},
	author={Gess, Benjamin and Gvalani, Rishabh S and Hu, Shanshan},
	journal={arXiv preprint arXiv:2507.02449},
	year={2025}
}

@article {CHSV24,
	AUTHOR = {Carrillo, J. A. and Hoffmann, F. and Stuart, A. M. and Vaes,
	U.},
	TITLE = {The mean-field ensemble {K}alman filter: near-{G}aussian
	setting},
	JOURNAL = {SIAM J. Numer. Anal.},
	FJOURNAL = {SIAM Journal on Numerical Analysis},
	VOLUME = {62},
	YEAR = {2024},
	NUMBER = {6},
	PAGES = {2549--2587},
	ISSN = {0036-1429},
	MRCLASS = {60G35 (62F15 65C35)},
	MRNUMBER = {4825256},
	DOI = {10.1137/24M1628207},
	URL = {https://doi.org/10.1137/24M1628207},
}

@ARTICLE{Garbuno-InigoHoffmannLiStuart.2020.SJoADS412,
	author = {Garbuno-Inigo, Alfredo and Hoffmann, Franca and Li, Wuchen and Stuart,
	Andrew M.},
	title = {Interacting Langevin Diffusions: Gradient Structure and Ensemble
	Kalman Sampler},
	journal = {SIAM Journal on Applied Dynamical Systems},
	year = {2020},
	volume = {19},
	pages = {412-441},
	number = {1},
	doi = {10.1137/19M1251655},
	url = {https://doi.org/10.1137/19M1251655}
}

@article {POCReview1,
	AUTHOR = {Chaintron, Louis-Pierre and Diez, Antoine},
	TITLE = {Propagation of chaos: a review of models, methods and
	applications. {I}. {M}odels and methods},
	JOURNAL = {Kinet. Relat. Models},
	FJOURNAL = {Kinetic and Related Models},
	VOLUME = {15},
	YEAR = {2022},
	NUMBER = {6},
	PAGES = {895--1015},
	ISSN = {1937-5093,1937-5077},
	MRCLASS = {82C22 (35Q70 65C35 65P20 82C40 92-10)},
	MRNUMBER = {4489768},
	MRREVIEWER = {Julian\ Tugaut},
	DOI = {10.3934/krm.2022017},
	URL = {https://doi.org/10.3934/krm.2022017},
}

@article {POCReview2,
	AUTHOR = {Chaintron, Louis-Pierre and Diez, Antoine},
	TITLE = {Propagation of chaos: a review of models, methods and
	applications. {II}. {A}pplications},
	JOURNAL = {Kinet. Relat. Models},
	FJOURNAL = {Kinetic and Related Models},
	VOLUME = {15},
	YEAR = {2022},
	NUMBER = {6},
	PAGES = {1017--1173},
	ISSN = {1937-5093,1937-5077},
	MRCLASS = {82C22 (35Q70 65C35 82C40 92-10)},
	MRNUMBER = {4489769},
	MRREVIEWER = {Hui\ Huang},
	DOI = {10.3934/krm.2022018},
	URL = {https://doi.org/10.3934/krm.2022018},
}

@article {meanfieldDuke,
	AUTHOR = {Serfaty, Sylvia},
	TITLE = {Mean field limit for {C}oulomb-type flows},
	NOTE = {With an appendix by Mitia Duerinckx and Serfaty},
	JOURNAL = {Duke Math. J.},
	FJOURNAL = {Duke Mathematical Journal},
	VOLUME = {169},
	YEAR = {2020},
	NUMBER = {15},
	PAGES = {2887--2935},
	ISSN = {0012-7094,1547-7398},
	MRCLASS = {35Q82 (82C22)},
	MRNUMBER = {4158670},
	DOI = {10.1215/00127094-2020-0019},
	URL = {https://doi.org/10.1215/00127094-2020-0019},
}

@article {harrisergodicityTransection,
	AUTHOR = {Eberle, Andreas and Guillin, Arnaud and Zimmer, Raphael},
	TITLE = {Quantitative {H}arris-type theorems for diffusions and
	{M}c{K}ean-{V}lasov processes},
	JOURNAL = {Trans. Amer. Math. Soc.},
	FJOURNAL = {Transactions of the American Mathematical Society},
	VOLUME = {371},
	YEAR = {2019},
	NUMBER = {10},
	PAGES = {7135--7173},
	ISSN = {0002-9947,1088-6850},
	MRCLASS = {60J60 (60H10)},
	MRNUMBER = {3939573},
	MRREVIEWER = {Julian\ Tugaut},
	DOI = {10.1090/tran/7576},
	URL = {https://doi.org/10.1090/tran/7576},
}

@book {bookGeometriccontroltheory,
	AUTHOR = {Jurdjevic, Velimir},
	TITLE = {Geometric control theory},
	SERIES = {Cambridge Studies in Advanced Mathematics},
	VOLUME = {52},
	PUBLISHER = {Cambridge University Press, Cambridge},
	YEAR = {1997},
	PAGES = {xviii+492},
	ISBN = {0-521-49502-4},
	MRCLASS = {93-02 (58E25 93B27)},
	MRNUMBER = {1425878},
	MRREVIEWER = {Heinz\ Sch\"{a}ttler},
}

@BOOK{Arnold.1998.586,
  title = {Random dynamical systems},
  publisher = {Springer-Verlag, Berlin},
  year = {1998},
  author = {Arnold, Ludwig},
  pages = {xvi+586},
  series = {Springer Monographs in Mathematics},
  doi = {10.1007/978-3-662-12878-7},
  isbn = {3-540-63758-3},
  mrclass = {37Hxx (37-02 60H10)},
  mrnumber = {1723992},
  mrreviewer = {Yuri\ Kifer},
  url = {https://doi.org/10.1007/978-3-662-12878-7}
}

@misc{eberle2019stochastic,
  author       = {Andreas Eberle},
  title        = {Stochastic Analysis},
  year         = {2019},
  note         = {Lecture notes, University of Bonn},
  howpublished = {\url{https://uni-bonn.sciebo.de/s/kzTUFff5FrWGAay/download}},
  month        = nov
}

@ARTICLE{Butkovsky.2014.TPA661,
  author = {Butkovsky, O. A.},
  title = {On ergodic properties of nonlinear {M}arkov chains and stochastic
	{M}c{K}ean-{V}lasov equations},
  journal = {Theory Probab. Appl.},
  year = {2014},
  volume = {58},
  pages = {661--674},
  number = {4},
  doi = {10.1137/S0040585X97986825},
  fjournal = {Theory of Probability and its Applications},
  issn = {0040-585X},
  mrclass = {60J10 (37A30 60B10 60H10)},
  mrnumber = {3403022},
  url = {https://doi.org/10.1137/S0040585X97986825}
}

@ARTICLE{FlandoliGessScheutzow.2017.AP1325,
  author = {Flandoli, Franco and Gess, Benjamin and Scheutzow, Michael},
  title = {Synchronization by noise for order-preserving random dynamical systems},
  journal = {Ann. Probab.},
  year = {2017},
  volume = {45},
  pages = {1325--1350},
  number = {2},
  doi = {10.1214/16-AOP1088},
  fjournal = {The Annals of Probability},
  issn = {0091-1798},
  mrclass = {37B25 (37G35 37H15 60H10)},
  mrnumber = {3630300},
  mrreviewer = {Jos\'{e} A. Langa},
  url = {https://doi.org/10.1214/16-AOP1088}
}

@ARTICLE{FlandoliGessScheutzow.2017.PTRF511,
  author = {Flandoli, Franco and Gess, Benjamin and Scheutzow, Michael},
  title = {Synchronization by noise},
  journal = {Probab. Theory Related Fields},
  year = {2017},
  volume = {168},
  pages = {511--556},
  number = {3-4},
  doi = {10.1007/s00440-016-0716-2},
  fjournal = {Probability Theory and Related Fields},
  issn = {0178-8051,1432-2064},
  mrclass = {37H15 (37B25 37G35 60H10)},
  mrnumber = {3663624},
  mrreviewer = {Min\ Zhao},
  url = {https://doi.org/10.1007/s00440-016-0716-2}
}

@ARTICLE{GessTsatsoulis.2024.AP1903,
  author = {Gess, Benjamin and Tsatsoulis, Pavlos},
  title = {Lyapunov exponents and synchronisation by noise for systems of {SPDE}s},
  journal = {Ann. Probab.},
  year = {2024},
  volume = {52},
  pages = {1903--1953},
  number = {5},
  doi = {10.1214/24-aop1690},
  fjournal = {The Annals of Probability},
  issn = {0091-1798},
  mrclass = {60 (35 37H15 37L30)},
  mrnumber = {4791422},
  url = {https://doi.org/10.1214/24-aop1690}
}

@ARTICLE{GessTsatsoulis.2020.SD17,
  author = {Gess, Benjamin and Tsatsoulis, Pavlos},
  title = {Synchronization by noise for the stochastic quantization equation
	in dimensions 2 and 3},
  journal = {Stoch. Dyn.},
  year = {2020},
  volume = {20},
  pages = {2040006, 17},
  number = {6},
  doi = {10.1142/S0219493720400067},
  fjournal = {Stochastics and Dynamics},
  issn = {0219-4937},
  mrclass = {60H15 (35K58 35R60 47D07)},
  mrnumber = {4161970},
  mrreviewer = {Kunwoo Kim},
  url = {https://doi.org/10.1142/S0219493720400067}
}

@article {Hairer.2011,
    AUTHOR = {Hairer, Martin},
     TITLE = {On {M}alliavin's proof of {H}\"{o}rmander's theorem},
   JOURNAL = {Bull. Sci. Math.},
  FJOURNAL = {Bulletin des Sciences Math\'{e}matiques},
    VOLUME = {135},
      YEAR = {2011},
    NUMBER = {6-7},
     PAGES = {650--666},
      ISSN = {0007-4497},
   MRCLASS = {60H07 (60H30 60J60)},
  MRNUMBER = {2838095},
MRREVIEWER = {Shi Zan Fang},
       DOI = {10.1016/j.bulsci.2011.07.007},
       URL = {https://doi.org/10.1016/j.bulsci.2011.07.007},
}

@ARTICLE{JabinWang.2018.IM523,
  author = {Jabin, Pierre-Emmanuel and Wang, Zhenfu},
  title = {Quantitative estimates of propagation of chaos for stochastic systems
	with {$W^{-1,\infty}$} kernels},
  journal = {Invent. Math.},
  year = {2018},
  volume = {214},
  pages = {523--591},
  number = {1},
  doi = {10.1007/s00222-018-0808-y},
  fjournal = {Inventiones Mathematicae},
  issn = {0020-9910,1432-1297},
  mrclass = {35Q30 (35R60 60F10 60F17 60H10 60K35 76R99)},
  mrnumber = {3858403},
  mrreviewer = {Paul\ Andr\'{e}\ Razafimandimby},
  url = {https://doi.org/10.1007/s00222-018-0808-y}
}

@incollection {MilletSanz1994,
    AUTHOR = {Millet, Annie and Sanz-Sol\'{e}, Marta},
     TITLE = {A simple proof of the support theorem for diffusion processes},
 BOOKTITLE = {S\'{e}minaire de {P}robabilit\'{e}s, {XXVIII}},
    SERIES = {Lecture Notes in Math.},
    VOLUME = {1583},
     PAGES = {36--48},
 PUBLISHER = {Springer, Berlin},
      YEAR = {1994},
   MRCLASS = {60J60},
  MRNUMBER = {1329099},
MRREVIEWER = {S. Ramasubramanian},
       DOI = {10.1007/BFb0073832},
       URL = {https://doi.org/10.1007/BFb0073832},
}

@ARTICLE{LiuWuZhang.2021.CMP179,
  author = {Liu, Wei and Wu, Liming and Zhang, Chaoen},
  title = {Long-time behaviors of mean-field interacting particle systems related
	to {M}c{K}ean-{V}lasov equations},
  journal = {Comm. Math. Phys.},
  year = {2021},
  volume = {387},
  pages = {179--214},
  number = {1},
  doi = {10.1007/s00220-021-04198-5},
  fjournal = {Communications in Mathematical Physics},
  issn = {0010-3616},
  mrclass = {82C22 (35Q70 60B10)},
  mrnumber = {4312362},
  url = {https://doi.org/10.1007/s00220-021-04198-5}
}

@ARTICLE{RenRocknerWang.2022.JDE1,
  author = {Ren, Panpan and R\"{o}ckner, Michael and Wang, Feng-Yu},
  title = {Linearization of nonlinear {F}okker-{P}lanck equations and applications},
  journal = {J. Differential Equations},
  year = {2022},
  volume = {322},
  pages = {1--37},
  doi = {10.1016/j.jde.2022.03.021},
  fjournal = {Journal of Differential Equations},
  issn = {0022-0396},
  mrclass = {60J60 (58J65)},
  mrnumber = {4398417},
  mrreviewer = {Huaqiao Wang},
  url = {https://doi.org/10.1016/j.jde.2022.03.021}
}

@INCOLLECTION{Sznitman.1991.165,
  author = {Sznitman, Alain-Sol},
  title = {Topics in propagation of chaos},
  booktitle = {\'{E}cole d'\'{E}t\'{e} de {P}robabilit\'{e}s de {S}aint-{F}lour
	{XIX}---1989},
  publisher = {Springer, Berlin},
  year = {1991},
  volume = {1464},
  series = {Lecture Notes in Math.},
  pages = {165--251},
  doi = {10.1007/BFb0085169},
  mrclass = {60J60 (60K35 82C40)},
  mrnumber = {1108185},
  mrreviewer = {Maria E. Vares},
  url = {https://doi.org/10.1007/BFb0085169}
}

@ARTICLE{Wang.2023.SPA265,
  author = {Wang, Feng-Yu},
  title = {Exponential ergodicity for singular reflecting {M}c{K}ean-{V}lasov
	{SDE}s},
  journal = {Stochastic Process. Appl.},
  year = {2023},
  volume = {160},
  pages = {265--293},
  doi = {10.1016/j.spa.2023.03.009},
  fjournal = {Stochastic Processes and their Applications},
  issn = {0304-4149,1879-209X},
  mrclass = {60H10 (60G65)},
  mrnumber = {4567526},
  url = {https://doi.org/10.1016/j.spa.2023.03.009}
}

@article {LPP15,
    AUTHOR = {Lemaire, Vincent and Pag\`es, Gilles and Panloup, Fabien},
     TITLE = {Invariant measure of duplicated diffusions and application to
              {R}ichardson-{R}omberg extrapolation},
   JOURNAL = {Ann. Inst. Henri Poincar\'e{} Probab. Stat.},
  FJOURNAL = {Annales de l'Institut Henri Poincar\'e{} Probabilit\'es et
              Statistiques},
    VOLUME = {51},
      YEAR = {2015},
    NUMBER = {4},
     PAGES = {1562--1596},
      ISSN = {0246-0203,1778-7017},
   MRCLASS = {60G10 (60F05 60J60 65C05)},
  MRNUMBER = {3414458},
MRREVIEWER = {Antoine\ J.\ Lejay},
       DOI = {10.1214/13-AIHP591},
       URL = {https://doi.org/10.1214/13-AIHP591},
}

@article {S00,
    AUTHOR = {Stevens, Angela},
     TITLE = {The derivation of chemotaxis equations as limit dynamics of
              moderately interacting stochastic many-particle systems},
   JOURNAL = {SIAM J. Appl. Math.},
  FJOURNAL = {SIAM Journal on Applied Mathematics},
    VOLUME = {61},
      YEAR = {2000},
    NUMBER = {1},
     PAGES = {183--212},
      ISSN = {0036-1399},
   MRCLASS = {92C17 (60K99 92D50)},
  MRNUMBER = {1776393},
MRREVIEWER = {John G. Milton},
       DOI = {10.1137/S0036139998342065},
       URL = {https://doi.org/10.1137/S0036139998342065},
}

@Book{Kunita.1997.346,
  author    = {Kunita, Hiroshi},
  publisher = {Cambridge University Press, Cambridge},
  title     = {Stochastic flows and stochastic differential equations},
  year      = {1997},
  isbn      = {0-521-35050-6; 0-521-59925-3},
  note      = {Reprint of the 1990 original},
  series    = {Cambridge Studies in Advanced Mathematics},
  volume    = {24},
  mrclass   = {60H10 (35R60 60H15)},
  mrnumber  = {1472487},
  pages     = {xiv+346},
}

@Article{Ruelle.1979.IHESPM27,
  author     = {Ruelle, David},
  journal    = {Inst. Hautes \'Etudes Sci. Publ. Math.},
  title      = {Ergodic theory of differentiable dynamical systems},
  year       = {1979},
  issn       = {0073-8301,1618-1913},
  number     = {50},
  pages      = {27--58},
  fjournal   = {Institut des Hautes \'Etudes Scientifiques. Publications Math\'ematiques},
  mrclass    = {58F18 (58F15 60G99 82A05)},
  mrnumber   = {556581},
  mrreviewer = {L.\ A.\ Bunimovich},
  url        = {http://www.numdam.org/item?id=PMIHES_1979__50__27_0},
}

@Article{BaxendaleStroock.1988.PTRF169,
  author     = {Baxendale, P. H. and Stroock, D. W.},
  journal    = {Probab. Theory Related Fields},
  title      = {Large deviations and stochastic flows of diffeomorphisms},
  year       = {1988},
  issn       = {0178-8051},
  number     = {2},
  pages      = {169--215},
  volume     = {80},
  doi        = {10.1007/BF00356102},
  fjournal   = {Probability Theory and Related Fields},
  mrclass    = {58G32 (58D25 60F10 60H99)},
  mrnumber   = {968817},
  mrreviewer = {Yu. E. Gliklikh},
  url        = {https://doi.org/10.1007/BF00356102},
}

@Article{Carverhill.1985.S273,
  author     = {Carverhill, Andrew},
  journal    = {Stochastics},
  title      = {Flows of stochastic dynamical systems: ergodic theory},
  year       = {1985},
  issn       = {0090-9491},
  number     = {4},
  pages      = {273--317},
  volume     = {14},
  doi        = {10.1080/17442508508833343},
  fjournal   = {Stochastics},
  mrclass    = {58F11 (58F25 60H20)},
  mrnumber   = {805125},
  mrreviewer = {R. W. R. Darling},
  url        = {https://doi.org/10.1080/17442508508833343},
}
\end{document}